\documentclass[11pt,letterpaper]{amsart}
\usepackage{amssymb,amsmath,amsthm,mathrsfs,color,upgreek,ulem}
\usepackage{mathtools}
\definecolor{links}{rgb}{0.5,0.2,0}
\definecolor{back}{rgb}{0.8,0.8,0.7}
\definecolor{dgreen}{rgb}{0,0.4,0.15}
\definecolor{indigo}{rgb}{0.4,0,0.9}
\usepackage[colorlinks=true,allcolors=links]{hyperref}
\usepackage[nobysame,initials,non-compressed-cites]{amsrefs}

\usepackage{txfonts}
\usepackage[T1]{fontenc}

\usepackage{color}
\newcommand{\rev}[1]{{#1}}
 
\newcommand{\norm}[1]{\left\Vert #1 \right\Vert}
  
\newcommand{\avg}[1]{\left\langle #1 \right\rangle}
 
\newcommand{\lb}{(}  
\newcommand{\rb}{)}   
  
\newcommand{\Sh} {{\mathsf{Sh}}}
\newcommand{\SH} {{\mathsf{SH}}}
\newcommand{\PM} {\mathsf{P}}

\newcommand{\bb} {{\mathfrak{b}}}
\renewcommand{\aa} {{\mathfrak{a}}}

\newcommand{\DW}{{\mathcal D(\mathcal W)}}
\newcommand{\W}{\mathcal W}
\newcommand{\I}{{\mathfrak I}}
\newcommand{\CL}{\mathcal{U}(\Omega)}

\newcommand{\Laa}{{\partial^\alpha\Lambda}}
\newcommand{\fL}{{\mathfrak{L} }}

\newcommand{\Sy}{\mathsf{Sy}}
\newcommand{\Tr}{\mathsf{Tr}}

\newcommand{\M}{\mathrm{M}}
\newcommand{\A}{\mathrm{A}}
\newcommand{\RH}{\mathrm{RH}}
\newcommand{\ip}[2]{\left\langle {#1},{#2} \right\rangle}
\newcommand{\abs}[1]{\left\vert {#1}\right\vert}

\newcommand{\R}{\mathbb R}
\newcommand{\Sc}{\mathcal S}

\newcommand{\C}{\mathcal C} 
\renewcommand{\l}{\langle}
\renewcommand{\r}{\rangle}
\newcommand{\ep}{\varepsilon}
\DeclareMathOperator{\BMO}{{BMO}}

\DeclareMathOperator{\supp}{{supp}}
\DeclareMathOperator{\diam}{{diam}}

\newcommand{\B}{\mathsf B}
 
\newcounter{thms}

\newcounter{other}
\numberwithin{other}{section}

\newtheorem{proposition}[other]{Proposition}
\newtheorem{theorem}[thms]{Theorem}
\newtheorem*{theorem*}{Theorem}
\newtheorem*{proposition*}{Proposition}
\newtheorem{cor}{Corollary}
\newtheorem*{corollary*}{Corollary}
\numberwithin{cor}{thms}
\newtheorem{lemma}[other]{Lemma}
\theoremstyle{definition}
\newtheorem{remark}[other]{Remark}
\newtheorem{definition}[other]{Definition}

\numberwithin{equation}{section}

\title[Sobolev regularity of Calder\'on-Zygmund operators]{Wavelet resolution and Sobolev regularity of Calder\'on-Zygmund operators on domains}

\author[F. Di Plinio]{Francesco Di Plinio}
\thanks{F. Di Plinio was partially supported by the National Science Foundation under the grant NSF-DMS-2000510. This material is based upon work supported by the National Science Foundation under Grant No. DMS-1929284 while this author was in residence at the Institute for Computational and Experimental Research in Mathematics in Providence, RI, during the \textit{Harmonic Analysis and Convexity} program of Fall 2022.}

\address[F. Di Plinio]{Dipartimento di Matematica e Applicazioni, Universit\`a di Napoli \\ \newline \indent Via Cintia, Monte S.\ Angelo 80126 Napoli, Italy}
\email{\href{mailto:francesco.diplinio@unina.it}{\textnormal{francesco.diplinio@unina.it}}}

\author[A. W. Green]{A. Walton Green}
\thanks{A. W. Green's research supported by NSF grant NSF-DMS-2202813. }
\address[A. W. Green]{Department of Mathematics, Washington University in Saint Louis\\ \newline \indent 1 Brookings Drive, Saint Louis, Mo 63130, USA}
\email{\href{mailto:awgree1@ilstu.edu}{\textnormal{awgree1@ilstu.edu}}}
\curraddr{Department of Mathematics, Illinois State University, Normal, IL 61761}

\author[B. D. Wick]{Brett D. Wick}
\thanks{B. D. Wick's research partially supported in part by NSF grant NSF-DMS-1800057, NSF-DMS-2000510, NSF-DMS-2054863 as well as ARC DP 220100285.}
\address[B. D. Wick]{Department of Mathematics, Washington University in Saint Louis\\ \newline \indent 1 Brookings Drive, Saint Louis, Mo 63130, USA}
\email{\href{mailto:bwick@wustl.edu}{\textnormal{bwick@wustl.edu}}}


\begin{document}

\mathtoolsset{showonlyrefs}

\begin{abstract}
Given a uniform domain $\Omega \subset \R^d$, we resolve each element of a suitably defined class of Calder\`on-Zygmund (CZ) singular integrals on $\Omega$ as the linear combination of Triebel wavelet operators and paraproduct terms. Our resolution formula entails a testing type characterization, loosely in the vein of the David-Journ\'e theorem, of  weighted Sobolev  space bounds in terms of Triebel-Lizorkin and tree Carleson measure norms of the paraproduct symbols, which is new already in the case $\Omega=\R^d$ with Lebesgue measure. Our characterization covers the case of compressions to $\Omega$ of global CZ operators, extending and sharpening past results of Prats and Tolsa for the convolution case. The  weighted estimates we obtain, particularized to the Beurling operator on a Lipschitz domain with normal to the boundary in the corresponding sharp Besov class, may be used to deduce quantitative estimates for quasiregular mappings with dilatation in the Sobolev space $W^{1,p}(\Omega)$, $p>2$.\end{abstract}
\maketitle
\setcounter{tocdepth}{1}
%\tableofcontents

\section{Introduction}

The landmark result by David and Journ\'e \cite{david84} provides necessary and sufficient conditions for $L^2(\R^d)$-boundedness of  a singular integral kernel operator $T$ in terms of the $\mathrm{BMO}$ membership of the \textit{paraproduct symbols} $T\mathbf{1}, T^*\mathbf{1}$. Testing-type boundedness results of this type have since played a prominent role in applications of singular integrals to elliptic regularity, geometric measure theory, theoretical fluid mechanics, quasiconformal geometry.

This article contains the first sharp testing type theorems for singular integrals \textit{on smoothness spaces}.  More precisely, given a domain $\Omega\subset \R^d$ satisfying minimal structural  requirements, a natural definition for smoothness $k$ singular integral operator $T$ on $\Omega$ may be given by prescribing suitable  derivative kernel estimates and weak boundedness: see Definition \ref{def:si}. In loose analogy with the David-Journ\'e theorem, this article  establishes  necessary and sufficient conditions for Lebesgue and weighted $W^{k,p}(\Omega)$-mapping properties of $T$ in terms of suitable  \textit{paraproduct symbol norms}.
  Relationship with past literature, techniques, and motivation, in particular coming from quantitative estimates for quasiregular maps, follows beyond the forthcoming summary of our main results.

\addtocounter{other}{1}
\subsection{Summary of the main results and comparison with past literature} We work in the large class of possibly unbounded uniform domains $\Omega\subset \R^d$. These can be roughly described as those domains admitting a resolution by Whitney cubes $W$ with the property that any two Whitney cubes can be connected by a chain with length controlled by a suitably defined long distance.  The spark for our analysis is an orthonormal multiresolution of  $L^2(\Omega)$ into wavelet projections due to Triebel. Loosely speaking, this decomposition associates to each Whitney cube $W$ of $\Omega$ a noncancellative wavelet, while a $k$-th order cancellative wavelet is associated to each dyadic subcube of $W$. 
Our first main result, Theorem \ref{thm:main}, \S\ref{sec:rep}, is a resolution formula for the $k$-th derivatives of a Calder\'on-Zygmund singular integral on $\Omega$. Broadly, we construct a \textit{wavelet operator $V$} and \textit{paraproduct operators $\pi^\gamma$} such that
	\[ \nabla^k Tf = V(\nabla^k f) + \sum_{|\gamma| \le k} \pi^\gamma(f).\] 
Wavelet operators are essentially almost diagonal perturbations of the Triebel decompositions and admit $L^p$ and sparse bounds, cf. \S\ref{s:sparse}. The paraproducts of each order $\gamma$ are associated with symbols $\{{  \vec \bb^{\gamma}[T]} : |\gamma| \le k  \}$  arising from local testing {   of the operator $T$} on translated monomials, namely
	\begin{equation}\label{e:intro-bb} \left \langle {  \vec \bb^{\gamma}[T]}, \varphi_Q \right \rangle = \left \langle \nabla^k T(x\mapsto (x-c(Q))^\gamma), \varphi_Q \right \rangle, \end{equation}
where $\varphi_Q$ is the wavelet associated to the  multiresolution cube $Q$ with center $c(Q)$. In \S\ref{sec:pp}, we obtain necessary and sufficient conditions for $W^{k,p}(\Omega)$-boundedness of the paraproducts. Unlike the classical case, these conditions depend on the exponent $p$ and are formulated in terms of localized Triebel-Lizorkin norms, cf. \S\ref{ss:tl-norms}, generalizing $\mathrm{BMO}$ membership, for the cancellative portion and tree Carleson measures in the vein of the work by Arcozzi, Rochberg and Sawyer \cite{arcozzi2002carleson}, cf.\ \S\ref{ss:pp-norms}, for the noncancellative portion of the symbol. In combination with Theorem \ref{thm:main}, the estimates of \S\ref{s:sparse} and \S\ref{sec:pp} lead to Lebesgue and weighted testing type theorems. We will sketch a few results below, sending instead to  Corollaries \ref{cor:sob} and \ref{cor:weight} in \S\ref{sec:rep} for the precise statements, and compare them to  relevant literature. 

\subsubsection*{Representation theorems, rough and smooth}
In common subject parlance, the fact that the class of normalized Calder\'on-Zygmund operators may be realized as the closed convex hull of band-limited operators in the Haar basis of $L^2(\R^d)$, usually termed \textit{dyadic shifts}, is referred to as \textit{dyadic representation}.  Dyadic representation theorems find their ancestors in the  martingale proof of the $T\mathbf{1}$ theorem by Figiel \cite{figiel90}, and in the work by Nazarov-Treil-Volberg on nonhomogeneous $Tb$-theorems \cite{NTV2}. Their first actual occurrence may be identified with
the work of Petermichl for the Hilbert transform \cite{PetAJM}, culminating with Hyt\"onen's proof of the $A_2$ conjecture \cite{hytonen12} for general Calder\'on-Zygmund operators. See \cites{li2019bilinear,martikainen2012representation,m2022,CDPOMRL,HL2022} and references therein for more recent extensions of this principle. 

Dyadic resolutions, at least in their current form, are not able to capture additional  regularity of the singular  kernel and smoothness space boundedness with it, as, for instance, dyadic shifts are in general not bounded on Sobolev spaces.
 Recent work by  the authors \cites{diplinio22wrt,diplinio23bilin} kickstarted an alternative approach based on   \textit{wavelet resolution}---or \textit{representation}, and these words will be used interchangeably---of Calder\'on-Zygmund operators which, among other results, leads to sharp weighted   \textit{homogeneous} Sobolev space bounds. The representation Theorem \ref{thm:main} of this paper, while building on the partial progress obtained with \cites{diplinio22wrt,diplinio23bilin}, represents a full step forward. First, the construction of the paraproduct symbols is direct, as in \eqref{e:intro-bb}, and not iterative as in  \cites{diplinio22wrt,diplinio23bilin}, leading to transparent testing conditions. Second, the domain setting and the presence of noncancellative \textit{top} wavelets introduces nontrivial technical complications in both the construction and the estimation of the paraproduct terms. The former leverages a sort of Poincar\'e \textit{equality}, see Lemma \ref{lemma:tele}. The latter combines this lemma with a sparse domination argument to efficiently estimate the noncancellative contributions by derivative averages.

\subsubsection*{Smooth testing theorems on $\R^d$}  For the purpose of this discussion,  $X^{s,p}$ and $\dot X^{s,p}$ with $1<p<\infty$ and $s>0$,  respectively stand for  generic $L^p$-scaling nonhomogeneous and homogeneous smoothness $s$ spaces. Examples are the fractional Sobolev spaces $W^{s,p}(\R^d)$, the Besov    and Triebel-Lizorkin scales $B_{p,q}^s(\R^d),F_{p,q}^s(\R^d)$ and their homogeneous counterparts.
Classical past literature on  singular integrals on smoothness spaces \cite{meyer85,lemarie,frazier87,frazier88,torres91}  focused on the $\dot X^{s,p}$ case,  under the strong cancellation assumptions
	\begin{equation}\label{e:strong} \nabla^{|\gamma|}T(x\mapsto x^\gamma)=0, \qquad |\gamma| < \lceil s \rceil, \end{equation}
corresponding to the case $\bb^{\gamma}=0$ in \eqref{e:intro-bb}. The same assumptions led to weighted bounds, see  \cite{han-hofmann}, the more recent synoptic work \cite{CHO2020}, and   \cite{diplinio22wrt} for sharp quantification.

 Y.-S.~Han and S.~Hofmann  explicitly note in \cite{han-hofmann} that the strong cancellation conditions \eqref{e:strong} might be removed with the use of paraproducts, with the caveat that ``however, it is an open problem to determine necessary and sufficient conditions on [the symbol]   so that [the corresponding paraproduct] is bounded on [smoothness spaces]''. For the $\dot X^{s,p}$ case,   \eqref{e:strong} is shown to be necessary in the  \textit{supercritical} range $(s-|\gamma|)p > d$. In the \textit{subcritical} range $s- \frac{d}{p} < |\gamma| < s$  embeddings of Sobolev, Besov, and Triebel-Lizorkin spaces may be invoked to deduce sufficient conditions. If $s=k$ is an integer, then the condition on the highest order symbol $|\gamma|=k$ is of $\BMO$ nature. Finally, in the Hilbert case $p=2$,   a capacitary-type characterization may be found in \cite{coifman-murai}. The authors are not aware of similar results when $p\neq 2$.

In contrast with the homogeneous case,  \eqref{e:strong} is not a necessary testing condition for $X^{s,p}$ for any $\gamma$. In addition, unlike  the $s=0$ case as well, necessary  testing conditions for the spaces $X^{s,p}$ depend on $s$ \textit{and on $p$}. Therefore, it is not possible, in general, to obtain sharp conditions via interpolation. Further,  extrapolation  poses restrictions on the allowed  Muckenhoupt  weight classes.  
Our Proposition \ref{paraproduct-main} thus appears to be the first characterization of strong-type (supercritical) and restricted strong-type (critical and subcritical) bounds for paraproducts on the Sobolev space $W^{n,p}$ for all $n \in \mathbb N$ and any $1<p<\infty$. While we have focused on the integer Sobolev space for simplicity, the proof techniques may be easily adapted to other smoothness scales.

\subsubsection*{Testing theorems for compressions} %Our work also encompasses smoothness space estimates for
A special, although of primary interest, subclass of our singular integrals are those of type  
\[
T_\Omega f(x) \coloneqq T({\mathbf{1}}_{\overline \Omega}  f)(x), \qquad x\in \Omega,
\]
 obtained from the \textit{compression to $\Omega$} of a smooth Calder\'on-Zygmund operator $T$ on $\R^d$. These are formally introduced in Definition \ref{def:comp}.  Note that  $T_\Omega f$ need not be smooth even when $T,f$ are smooth.
 % \sout{, unless geometric constraints are placed on $\Omega$}.
Building on past work of several authors \cites{cruz2013beltrami,mateu-extra,anikonov,bertozzi} Prats and Tolsa \cite{prats-tolsa} studied Sobolev bounds for compressions of a class of smooth convolution-type singular integrals, in particular satisfying the strong cancellation assumption \eqref{e:strong}. Though not explicitly using paraproducts, Prats and Tolsa deduce sufficient conditions for $W^{k,p}(\Omega)$-boundedness of $T_\Omega$ in terms of tree Carleson measure norms of $\nabla^kT_\Omega P$ for polynomials $P$ of degree less than $k$. These norms are akin to those on analytic Besov spaces introduced by Arcozzi, Rochberg, and Sawyer \cite{arcozzi2002carleson}.  Even within the class of compressions of globally defined convolution operators, the conditions of \cite[Thms.\ 1.1, 1.2]{prats-tolsa} are only necessary for the case $k=1$ or, if the domain is bounded, the purely supercritical case $p>d$. Their approach has also been applied to other smoothness spaces \cites{doubtsov,vasin19t1,vasin20singular-russian,prats-saksman}
 
Our Proposition \ref{paraproduct-main} yields $W^{k,p}(\Omega)$-boundedness of compressions in terms of  a more general family of tree Carleson measure norms encoding the order of the polynomial $P$ in  $T_\Omega P$. 
The norm $\| f\|_{\mathrm{Carl}^{n,p}(\Omega)}$ is precisely defined in \eqref{eq:carleson-measure} below. Imprecisely, we may say it is the smallest $K>0$ such that for all Whitney cubes $W$ with $\ell(W) \lesssim 1$,
	\begin{equation} \label{eq:carleson-measure-intro}
	\begin{aligned} 
		\sum_{P \subset \Sh(W)} \|   f \|_{L^p(\Sh(P))}^{p p'} \ell(P)^{(n p-d)(p'-1)} 
			&\le K^{p'} \|  f\|_{L^p(\Sh(W))}^p, & n p \le d; \\
		\|  f \|_{L^p(\Sh(W))} 
			&\le K, & n p > d. 
	\end{aligned}
	\end{equation}	
The shadows $\Sh(P)$ may informally be thought of as a large ball (say of radius $3 \mathrm{dist}(P,\partial \Omega)$) centered on $P$. In combination with Theorem \ref{thm:main}, we then obtain a completely new testing type result for compressions, Theorem \ref{thm:compress} in \S\ref{sec:compress}, extending the scope of \cite{prats-tolsa} in several different directions. For illustrative purposes, an abbreviated and somewhat informal version of Theorem \ref{thm:compress} is stated below. 
\begin{theorem*}
Let $T$ be a  $k$-smooth Calder\'on-Zygmund operator on  $\R^d$. Let $\Omega$ be a uniform domain and associate to the  corresponding compression  $T_\Omega=\mathbf {1}_\Omega T(\cdot \mathbf{1}_\Omega)$ the family of symbols $\{\vec{\mathfrak b}^\gamma[T_\Omega] : |\gamma| \le k-1\}$ defined by \eqref{e:intro-bb}. Then,
	\[ \left\Vert \nabla^k T_\Omega f \right\Vert_{L^p(\Omega)} \lesssim \left( 1 + \norm{T}_{\dot F(k,p,\R^d)} + \sup_{|\gamma| \le k-1} \norm{\vec{\mathfrak b}^\gamma[T_\Omega] }_{\mathrm{Carl}^{k-|\gamma|,p}(\Omega)} \right) \norm{f}_{W^{k,p}(\Omega)}.\]
\end{theorem*}
The norm $\norm{T}_{\dot F(k,p, \mathbb R^d)}$ is a special case of the definitions  \eqref{e:Fn} below. {   In short, it summarizes the  Triebel-Lizorkin norms, of suitable order, of the paraproduct symbols $\{\mathfrak b^\gamma[T] : |\gamma| \le k\}$ obtained by testing  the \textit{global} operator $T$ and thus independent of the domain $\Omega$. Control of these norms in effect replaces the strong cancellation assumption \eqref{e:strong} with a much weaker, in fact necessary and sufficient condition. On the other hand, the Carleson measure conditions on $\|\vec{\mathfrak b}^\gamma[T_\Omega] \|_{\mathrm{Carl}^{k-|\gamma|,p}(\Omega)} $ involve the compression $T_\Omega$, and are in fact also necessary for all $k\geq 1$.} An additional advantage compared to those of  \cite{prats-tolsa} is that our testing conditions do not require that the domain $\Omega$ be bounded. Lastly, the results extends to the weighted Sobolev spaces, for weights in the Muckenhoupt-type class $[w]_{\A_p(\Omega)}$; see e.g. Theorem \ref{t:belt} below. We send to Lemma \ref{l:worse} for the formal relationship of  \cite[Thms.\ 1.1, 1.2]{prats-tolsa} with our main results.

{   The theorem for compressions $T_\Omega$ we have just described is a special case of a  testing-type theorem for singular integrals on $\Omega$, see Corollary \ref{cor:sob} for details.} 
%The main benefit to the previous result over what follows is that for compressions, one can connect the Triebel-Lizorkin norms of the symbols of $T_\Omega$ \eqref{e:intro-bb} with the symbols of the uncompressed operator $T$. And in fact those may be easier to compute, for example when $T$ is a convolution.
% 
%\begin{theorem*}
%Let $T$ be a smooth singular integral operator on $\Omega$, a uniform domain.
%	\[ \left\Vert \nabla^k T f \right\Vert_{L^p(\Omega)} \lesssim \left( 1 + \norm{T}_{\dot F(k,p)} + \sup_{|\gamma| \le k-1} \norm{\mathfrak b^\gamma}_{\mathrm{Carl}^{k-|\gamma|,p}(\Omega)} \right) \norm{f}_{W^{k,p}(\Omega)}.\]
%\end{theorem*}
%
%The norm $\norm{T}_{\dot F(k,p)}$ is precisely defined in \eqref{e:Fn} below in terms of Triebel-Lizorkin norms of the symbols $\{\mathfrak b^\gamma : |\gamma| \le k\}$. 
Both previous results have weighted counterparts,  pursued  in particular with the following application in mind.

%\addtocounter{other}{1}
\subsubsection*{Weighted estimates on domains and Beltrami regularity}

Beyond the intrinsic interest for a sharp $T{\mathbf 1}$-type theorem on $X^{s,p}$-type spaces on domains, our work, and  in particular its weighted component,  is strongly motivated by prospective applications to partial differential equations. We next expand on the specific connection with the Beltrami equation.

Weighted $L^p$-norm inequalities for the Ahlfors-Beurling operator $B$ on $\mathbb C\equiv \mathbb R^2$ may be used to obtain resolvent estimates  for the Beltrami equation \cite{astala01} with dilatation coefficient $\mu$.   
Furthermore, sharp quantification of the weighted norm of $B$  in terms of the Muckenhoupt characteristic, originally due to Petermichl and Volberg \cite{petermichl02} has been used to deduce injectivity at the lower endpoint of the critical interval. The connection is rooted in the deep result, due to Astala \cite{astala94}, that suitable powers of Jacobians of homeomorphic solutions to the Beltrami equations belong to appropriate Muckenhoupt classes.  See \cite[Ch.\ 13,14]{astala-book} for a detailed account of these landmark results and more.

In the same spirit of  \cite{astala01}, the novel Sobolev weighted boundedness results for singular integrals obtained in this paper, particularized to the compression $B_\Omega$ of the Ahlfors-Beurling operator, may be used to obtain quantitative Sobolev resolvent estimates for the Beltrami equation when the dilatation   belongs to some smoothness space $X^s(\Omega)$ with $s>0$.  This is carried out by the authors in \cite{DGW23beltrami}.   This approach via weighted theory is, at least in part, antithetical to previous literature on the subject \cite{clop-et-al,mateu-extra,cruz2013beltrami,prats19} whose common blueprint   is the Iwaniec compactness method \cite{iwaniec}.
The latter is a suitable replacement for the Astala-Iwaniec-Saksman strategy \cite{astala01} insofar as weighted estimates for $B_\Omega$ are unavailable. However, quantitative estimates for the resolvent will not follow from this kind of compactness arguments. 

Our Corollary \ref{cor:compress} below applies to fully characterize the \textit{weighted} $W^{k,p}(\Omega)$- boundedness of compressions in terms of weighted tree Carleson measure conditions. In the supercritical range, for simplicity when $k=1$, these conditions  can be deduced  from Sobolev estimates on $\nabla B(\mathbf 1_{\Omega})$, see Lemma \ref{l:Apd}. In turn, the latter are known  to be explicitly related to Besov regularity of $\partial\Omega$ \cite{cruz-tolsa,tolsa2013regularity,vasin16regularity-russian}. Combining these observations   with Corollary \ref{cor:compress}  leads to the following result: see \S\ref{pf:belt} for a formal proof.
 
\setcounter{thms}{2}
\begin{theorem} \label{t:belt}Let $\Omega\subset \mathbb{C}$ be a bounded Lipschitz domain and $2<p<\infty$. For each  $2<r<p$ there is a positive increasing function $\mathcal G$ such that, for each weight $w$
\[ \left \Vert B_{ \Omega} f \right \Vert_{W^{1,p}(\Omega,w)} \lesssim \mathcal G\left( [w]_{\mathrm A_{\frac{r}{2}}(\Omega)}\right) \left( 1 + \frac{\norm{ B_\Omega \mathbf 1 }_{W^{1,p}(\Omega,w)}}{w(\Omega)^{\frac 1p}} \right)\left \Vert f \right \Vert_{W^{1,p}(\Omega,w)}.\] 
Furthermore, for each $q>p$ there is a function $\tilde{\mathcal G}$ such that if the boundary normal vector $N_\Omega$ to $\Omega$ lies in the Besov space $B^{1-\frac{1}{q}}_{q,q}(\partial \Omega)$, then
\[ \left \Vert B_{ \Omega} f \right \Vert_{W^{1,p}(\Omega,w)} \lesssim \tilde{\mathcal G}\left( [w]_{\mathrm A_{\frac{r}{2}}(\Omega)}, [w]_{\mathrm{RH}_{\frac{q}{q-p}}(\Omega)}\right) \left( 1 + \norm{ N_\Omega }_{B^{1-\frac{1}{q}}_{q,q}(\partial \Omega)} \right) \left \Vert f \right \Vert_{W^{1,p}(\Omega,w)}.\] 
 \end{theorem}
\setcounter{thms}{0}
In \cite{DGW23beltrami}, the authors apply Theorem \ref{t:belt} in order to implement the Astala-Iwaniec-Saksman strategy \cite{astala01} to obtain quantitative $W^{1,p}(\Omega)$ estimates on the Beltrami resolvent $(I-\mu B_\Omega)^{-1}$ 
in terms of the Sobolev regularity of $\mu$ and the Besov regularity of $\Omega$ and $f(\Omega)$ where $f$ is the principal $\mu$-quasiconformal solution.

\subsection*{Notational conventions} Throughout this article, $d\geq 2$ denotes the dimension of the ambient Euclidean space, and functions on  $\R^{nd}\coloneqq \bigtimes_{j=1}^n \R^d$ will be considered, for $n\geq 1$. The parameter $n$ stands for the order of linearity of the forms appearing in our treatment. Unless indicated otherwise, these forms act  on $n$-tuples of complex or vector-valued functions in $L^\infty(\R^d) \cap L^1(\R^d)$.
Multi-indices on $\R^d$ are usually indicated with Greek letters such as e.g.\ $ \alpha=(\alpha_1,\ldots, \alpha_d)$. The suggestive notation
\[
\mathsf{x}^\alpha:\R^d \to \R, \qquad 
\mathsf{x}^\alpha (x) = \prod_{j=1}^d x_j^{\alpha_j } 
\]
is adopted for the corresponding monomials.
The bracket  $(f,g)\mapsto\langle f,g \rangle$ always stands for the appropriate  duality pairing of $f$ and $g$ and is linear in both arguments separately. For instance, if $\Omega$ is a domain in $\R^d$ and $f,g\in L^2(\Omega)$, then
\[
\langle f,g\rangle_\Omega =\int_\Omega fg \,\mathrm{d} x.
\]
The subscript $\Omega$ in the duality product is omitted when the domain is clear from context.
Cartesian products of $d$  left-closed, half-open intervals of equal length are referred to as \textit {cubes}. A distinguished role is played by the cubes of  the standard dyadic grid on $\R^d$, which is denoted by $\mathcal D$ throughout. The notations
\[
\mathcal D(Q)\coloneqq\{R\in \mathcal D: R\subsetneq Q\}, \qquad  \mathcal D_+(Q)\coloneqq\{R\in \mathcal D: R\subset Q\}, \qquad
Q\in \mathcal D\] 
are also used.
For any cube $Q\subset \R^d$, indicate by $c(Q)$ its center and $\ell(Q)$ its common  sidelength. As usual, $aQ$ stands for the cube with the same center as $Q$, but with sidelength $a \ell(Q)$. The long distance between any two cubes $Q,S\subset \R^d$ is defined as	\[ {\mathfrak d}(Q,S)=\max\{|c(Q)-c(S)|,\ell(Q),\ell(S)\}.\]
We will need the two transformations
\[\begin{split}
&\Tr_Q f(x)\coloneqq f\left( x-\big(c(Q), \ldots, c(Q)\big)\right), \\
&\Sy_Q f(x)\coloneqq f\left(\frac{ x-\big(c(Q), \ldots, c(Q)\big)}{\ell(Q)}  \right), \qquad x  \in \R^{dn}
\end{split}
\]
 acting on functions $f$ on $\R^{dn}$ and parametrized by cubes $Q\subset \R^d$.
As customary, local  norms of $f\in L^1_{\mathrm{loc}}(\R^d)$ on cubes $Q$ are indicated by
\[
\| f\|_{L^p(Q)} \coloneqq
 \|\mathbf{1}_Q f\|_{L^p(\R^d)}, \qquad 
 \langle f \rangle_{p,Q} \coloneqq  |Q|^{-\frac1p} \| f\|_{L^p(Q)},   \quad 0\leq p<\infty.
\]
The simplified notation $\langle f \rangle_{Q}\coloneqq  \langle f \rangle_{1,Q}$ is also used.
Our main results are stated in the scale of    Lorentz-Sobolev spaces. First, for $0<p,q\leq \infty$, the local Lorentz norm on $Q$ is given by
\[
\begin{split}
 &\langle f \rangle_{(p,q),Q} \coloneqq \left\| t^{\frac1p} f^{*,Q}(t) \right\|_{L^q\left( (0,\infty),\frac{\mathrm{d} t}{t} \right)}, \\ &
f^{*,Q}:(0,\infty) \to [0,\infty), \quad f^{*,Q}(t)\coloneqq \inf\left\{s>0: \left| \left\{x\in Q: |f(x)|>s\right\}\right| \leq t|Q|\right\}.
\end{split}
\]
The Lorentz-Sobolev space norms on  $\Omega\subset \R^d$ are defined by
\[
\|f\|_{W^{k,(p,q)}(\Omega)} \coloneqq \sum_{j=0}^{k-1} \| \mathbf{1}_{\Omega} \nabla^j f\|_{L^p(\R^d)} +  \| \mathbf{1}_{\Omega}\nabla^k f\|_{L^{p,q}(\R^d)}.\] 
The standard notation  $W^{k,p}(\Omega)$ is used when $p=q$. A few of the corollaries involve weighted Sobolev spaces. In this article, a \textit{weight} $w$ is a positive, locally integrable function on $\R^d$. If $w$ is a weight, we identify $w$ with the corresponding absolutely continuous measure and write  $w(A)=\int_A w$. Throughout, for $1<p<\infty$, the $p$-th dual of a weight $w$ is the weight $\sigma\coloneqq w^{-\frac{1}{p-1}}$.
The weighted Sobolev norms are then defined by
\[
\|f\|_{W^{k,p}(\Omega,w)} \coloneqq \sum_{j=0}^{k} \left\| w^{\frac1p}\mathbf{1}_{\Omega} \nabla^j f\right\|_{L^p(\R^d)}. \] 
\subsubsection*{A note on constants and increasing functions} The symbol $C$ stands for a generic large positive constant whose value may change at each occurrence. The usage of the almost inequality $A\lesssim B$ follows the same convention, in the sense that $A\lesssim B$ if and only if $A\leq CB$. We write $A\sim B$ when $A\lesssim B$ and $B\lesssim A$. 
The symbol $\mathcal G$ appearing in the statement of certain theorems stands for a positive increasing polynomial function of its arguments, which also may change from statement to statement. The explicit form of each occurrence of $\mathcal G$  is usually specified  within the proofs.

\section{The basic toolbox of singular integrals on domains}\label{sec:wavelets}

\addtocounter{other}{1} 
\subsection{CZ localization on domains} Let $\Omega \subset \mathbb R^d$ be a nonempty open set. To each  $\mathsf{w} \gg 1$, associate   a Whitney decomposition $\mathcal W(\Omega)$ of $\Omega$ with parameter $\mathsf{w} $. Namely,   $\mathcal W(\Omega)\subset \mathcal D$ has the properties	\begin{itemize}	
	\item[(W1)] $\Omega = \bigcup\{Q :Q \in \mathcal W(\Omega) \}$ and the union is pairwise disjoint;
	\item[(W2)] $\mathrm{dist}(Q,\Omega^c) \ge 2^4\mathsf{w} \ell(Q)$ for all $Q\in \mathcal W(\Omega)$.
	\end{itemize}
As a consequence of (W1)-(W2), the additional properties\begin{itemize}	
	\item[(W3)] if $Q,Q' \in \mathcal W(\Omega)$ and $\overline{Q} \cap \overline{Q'} \ne \varnothing$, then $ {\ell(Q)}\le C {\ell(Q')}$;
	\item[(W4)] $\displaystyle \sum_{Q \in \mathcal W} 1_{2^5\mathsf{w}Q} \le C$
	\end{itemize}
follow with a suitable choice of positive constant $C=C(\mathsf w,d)$. Having fixed the Whitney decomposition $\mathcal W$, a multiscale resolution $\mathcal M$ of $\Omega$ is then obtained by 
\[
\mathcal M \coloneqq \mathcal W \cup \mathcal S, \qquad \mathcal S\coloneqq  \bigcup_{Q\in \mathcal W} \mathcal D(Q).
\]
%To each  $S\in \mathcal S$, associate the partition of $\mathcal M$ given by		\begin{equation}\label{e:ASW0}
%	A(S) \coloneqq \{ P \in \mathcal M : \ell(P) \gtrsim \ell(S), \ |c(P)-c(S)| \lesssim \ell(P)\},  \qquad I(S) \coloneqq \mathcal M \setminus A(S).
%\end{equation}
%The constants in $\lesssim$  in the definition of $A(S)$ are picked to ensure
%	\begin{equation}\label{e:ASW} \{ W \in \mathcal W \cap A(S) : \mathsf w W \cap \mathsf w S \ne \varnothing \} = \{ W \in \mathcal W : \mathsf w W \cap \mathsf w S \ne \varnothing \}.\end{equation}
\begin{remark}\label{rem:global} The case $\Omega=\R^d$ may be included in our framework by setting $\mathcal W$ to be the empty collection and $\mathcal S\coloneqq \mathcal D $. For instance, the resolution formula of Proposition \ref{prop:triebel} below holds with a void first summation. This will be used to discuss the relationship between representations of global singular integral forms on $\R^d$ with their compression on $\Omega\subsetneq \R^d$. \end{remark}
A fundamental result of Triebel \cite{triebel2008}, building on the work of Daubechies \cite{daubechies-lectures}, yields a smooth, compactly supported orthonormal basis of $L^2(\Omega)$ subordinated to the multiresolution $\mathcal M$. To describe this result within our more general framework, it is convenient to introduce the general notion of localized wavelets as follows. 
{{For $n\geq 1$, $0 \le \delta < 1$ and $\rho \ge 0$,   introduce the norm \[	\|\varphi\|_{n,\delta,\rho} := \sup_{x \in \R^{nd}} \left(1+|x|\right)^{nd+\rho}  \left[|\varphi(x)| + \sup_{0 \le |h| \le 1} \dfrac{|\varphi(x+h)-\varphi(x)|}{|h|^{\delta}} \right]. \] 
Now let $\sigma \geq 0$ be a smoothness parameter, and $\{\sigma\}=\sigma- \lfloor \sigma \rfloor$ be its fractional part.
 For each cube $R$ of $\R^d$, define the $L^1$-normalized class 
	\[ \Phi^{\sigma,\rho}_{n} (R)\coloneqq  \left\{ |R|^{-n} \Sy_R\varphi: \varphi \in W^{\lfloor \sigma \rfloor,\infty}(\R^d): \sup_{0 \le |\alpha| \le \lfloor \sigma\rfloor }\left\|\partial^\alpha \varphi\right\|_{n,\{\sigma\},\sigma+\rho} \leq 1 \right\}.
	   \]
\rev{We use the term $L^1$-normalized since $\| \varphi \|_{L^1(\mathbb R^{nd})} = \| |R|^{-n} \Sy_R \varphi \|_{L^1(\mathbb R^{nd})} $}.
The  simplified notation $\Phi^{\sigma,\rho} (R)$ is used in place of $\Phi^{\sigma,\rho}_1 (R)$. In the case $n=1$,  suitably chosen cancellative subclasses are needed. When $\{\sigma\} + \rho>0$, define 
\begin{equation} 
\label{e:cancelcond}
\Psi^{\sigma,\rho} (R)\coloneqq  \left\{  \varphi\in \Phi^{\sigma,\rho} (R):\int_{\R^d} x^\alpha \varphi(x)\,  \mathrm{d}x =0 \quad  \forall    0 \le |\alpha| \le \lfloor \sigma \rfloor \right\}.
\end{equation}
Notice that the restriction $\{\sigma\}+\rho>0$ ensures that all integrals in \eqref{e:cancelcond} are absolutely convergent.
  Finally, we  replace the decay parameter $\rho$ by the superscript   $\Subset$  to indicate the subset of the corresponding class of functions having compact support in $\bigtimes_{j=1}^n \mathsf{w}R$. For instance
$\Phi^{\sigma,\Subset}_n (R)\coloneqq\{\varphi\in \Phi^{\sigma,n}_{n} (R): \mathrm{supp} \,\varphi \Subset \bigtimes_{j=1}^n\mathsf{w} R \}$.
To each $\varphi_R\in \Phi^{\sigma,\rho}_{n} (R) $, we associate naturally the $n$-linear form
\begin{equation}\label{eq:rank1}
\varphi_R(f_1,\ldots, f_n) \coloneqq \int_{\R^{nd}}  \varphi_R(x_1,\ldots,x_n) \left(\prod_{j=1}^n f_j(x_j) \right) \, \mathrm{d} x.
\end{equation}
Overloading the same notation, the  definition of each class is  extended to include vectors $\varphi_R\coloneqq (\varphi_R^{(j)}:1\leq j\leq m)$ where each $\varphi^{(j)}_R$ belongs to the same scalar class. Given a vector function $f = (f^{(j)}:1\leq j\leq m)$, instead of writing $\sum_{j=1}^m \varphi_R^{(j)}(f^{(j)})$,
we will simply write $\varphi_R(f)$. This  allows for vector function inputs in wavelet forms.}

At this point, we are ready to state Triebel's multiresolution theorem in the following form. See \cite{triebel2008}*{Theorem 2.33} for the original statement and \cite{triebelIII}*{\S4.2} for further details. 
\begin{proposition}\label{prop:triebel}  Let $k\geq 1$ be a fixed integer. Then there exist% s a positive integer $\jmath$, 
$\jmath \in \mathbb N$ and $c>0$ depending on $k$, $d$, and the geometric parameter $\mathsf w$ of the multiresolution $\mathcal M$ of $\Omega$, and an $L^1$-normalized family
\begin{equation}
\label{e:triebelref}  \mathfrak B(\Omega,k)\coloneqq
\left\{\chi_{{W},j}\in c\Phi^{k+1,\Subset}({W}):\begin{array}{l}  {W}\in \mathcal W\\ 1\leq j \leq \jmath \end{array}\right\}\cup
\left\{\varphi_{S}\in c\Psi^{k+1,\Subset}({S}): {S}\in \mathcal S \right\}
\end{equation}
with the property that \[
\mathfrak B^2(\Omega,k)\coloneqq
\left\{\sqrt{|{W}|}\chi_{{W},j}: {W}\in \mathcal W, 1\leq j \leq \jmath \right\}\cup \left\{\sqrt{|{S}|}\varphi_{S} : {S}\in \mathcal S \right\} \] is an orthonormal basis of $L^2(\Omega)$, and in particular there holds
\begin{equation}
\label{e:triebel} 
\langle f_1,f_2 \rangle= \sum_{{W}\in \mathcal W } \sum_{j=1}^{\jmath}|{W}| \chi_{{W},j}(f_1) \overline{\chi_{{W},j}}(f_2) + \sum_{{S}\in \mathcal S } |{S}|\varphi_{S}(f_1) \overline{\varphi_{S}}(f_2) 
\end{equation}
for all $ f_1,f_2\in L^2(\Omega).$
\end{proposition}
\begin{remark}\label{rem:vec}
For consistency with the multilinear case, the orthogonal expansion \eqref{e:triebel} is performed on  the bilinear pointwise product form instead of the customary skew-linear scalar product on $L^2(\Omega)$. The latter may be recovered exchanging $f_2$ with its conjugate.
\end{remark}
\begin{remark}[Projections]  \label{rem:proj}  In accordance with the reproducing formula \eqref{e:triebel}, the space $L^2(\Omega)$ is the orthogonal sum of the range of the projection operators
\begin{equation}
\label{e:proj}
\langle   \W f_1,f_2 \rangle\coloneqq \sum_{{W}\in \mathcal W } \sum_{j=1}^{\jmath}|{W}| \chi_{{W},j}(f_1) \overline{\chi_{{W},j}}(f_2) ,\qquad \langle   \Sc f_1,f_2 \rangle\coloneqq \sum_{{S}\in \mathcal S } |{S}|\varphi_{S}(f_1) \overline{\varphi_{S}}(f_2).\end{equation}
\end{remark}

\begin{remark}[Dense class] \label{rem:dense}    To ensure  absolute convergence whenever the reproducing formula  \eqref{e:triebel} is applied, it is convenient to work with a dense class defined with respect to the wavelet coefficients occurring therein. Namely, define $\mathcal U(\Omega)$ to be the subspace of those  $f\in L^2(\Omega)$ with the property that 
\begin{equation}
\label{e:qualass}
\sup_{\substack{W\in \mathcal W,  1\leq j \leq \jmath}}\left(1+\ell(W)\right)^{-b}\left|\chi_{W,j}(f)\right|+\sup_{S\in \mathcal S} \left(1+\ell(S)\right)^{-b}\left|\varphi_S( f)\right| 
<\infty\qquad \forall b\geq 0. 
\end{equation}
It is not difficult to  see that $\mathcal U(\Omega)$ contains $\mathcal C^\infty_0(\Omega)$, and  is dense in $W^{k,2}(\Omega)$ and thus in $W^{m,p}(\Omega)$ for all $0\leq m \leq k$ and $1\leq p<\infty$. 
\end{remark}
Our analysis  relies in first instance on the following main building blocks, which we term \textit {wavelet forms}. While this article is concerned with linear operators, paraproducts arise naturally in the representation of the latter, leading us to considering general linearities. 
{
\begin{definition}[Wavelet forms] \label{def:forms}
Let $n\geq 2$ and $\eta>0$. To collections
 \begin{equation} \label{e:collwave}
 \begin{split}
& \{\phi_W\in \Phi^{0,\eta}_{n} (W) : W\in \mathcal W \}, \\  &\{\psi^{j}_S\in \Psi^{\eta,0} (S) : S\in \mathcal S \}, \; j=1,2,\qquad 
 \;\{\phi^{j}_S\in \Phi^{\eta,0} (S) : S\in \mathcal S \}, \; 2< j \leq n
 \end{split}
 \end{equation} associate the $n$-linear $\Omega$-\textit{wavelet form}
 \begin{equation}
\label{e:triebel2}
 \Lambda(f_1,\ldots, f_n) \coloneqq \sum_{W\in \mathcal W }|W|\phi_W(f_1,\ldots, f_n)+ \sum_{S\in \mathcal S }|S|  \left( \bigotimes_{j=1}^2 \psi^{j}_S \bigotimes_{2<j \leq n}\phi^{j}_S\right) (f_1,\ldots, f_n).
 \end{equation}{The class of  $\Omega$-\textit{wavelet forms} thus depends on $\Omega$ via the multiresolution $\mathcal M$. However, the action of $\Lambda$ is well defined, with in particular absolute convergence of the sums, e.g.\ for tuples $f_j\in L^\infty_0(\R^d), 1\leq j \leq n$. This is apparent from estimate \eqref{e:sparzglob} below. When $\Omega$ is clear from context, we simply refer to the class \eqref{e:triebel2} as \textit{wavelet forms}.}
  Wavelet forms naturally split into pieces localized by each Whitney cube $W\in \mathcal W$, as in 
  \begin{equation}
\label{e:lambdaW}
\Lambda=\sum_{W\in \mathcal W}|W| \Lambda_W, \qquad 
 \Lambda_W  \coloneqq    \phi_W  + \sum_{S \in \mathcal{D}(W)} \frac{|S|}{|W|} \left( \bigotimes_{j=1}^2 \psi^{j}_S \bigotimes_{2<j \leq n}\phi^{j}_S\right) .
 \end{equation}
When $\Phi^{0,\eta}_{n} (W), \Psi^{\eta,0} (S) ,\Phi^{\eta,0} (S) $ may be respectively replaced by $\Phi^{0,\Subset}_{n} (W), \Psi^{\eta,\Subset} (S) $ and $\Phi^{\eta,\Subset} (S) $ in \eqref{e:collwave}, the form $\Lambda$ defined in \eqref{e:triebel2} is termed  $n$-linear \textit{fully localized wavelet form}. Note that the form $\Lambda\coloneqq \sum_{W\in \mathcal W}|W|\Psi_W$, where $\Psi_W$ is defined in  \eqref{e:PsiZ} below, is an example of fully localized wavelet form, up to a suitable choice of signs for the wavelets $\varphi_S$ appearing therein.
 \end{definition}
 \begin{remark}
The right hand side of \eqref{e:triebel} is, up to a multiplicative factor,  a fully localized wavelet form corresponding to \[
\phi_W \coloneqq \sum_{j=1}^{\jmath} \chi_{W,j} \otimes \overline{ \chi_{W,j} }, \quad W\in \mathcal W, \qquad
\psi^1_S \coloneqq \varphi_S, \;\psi^2_S \coloneqq\overline{ \varphi_S },\quad  S\in \mathcal S.
\]
\end{remark}
 It is immediate to see that wavelet forms are $n$-linear Calder\'on-Zygmund singular integral forms on $\R^d$, e.g. as defined in \cite[Def.\ 3.2]{diplinio23bilin}. As such, they do admit $L^p$ and sparse domination estimates. These are recalled in the \S\ref{s:sparse}, together with more precise estimates that rely on their localization to the domain $\Omega$.}
{
\addtocounter{other}{1} 
\subsection{Triebel-Lizorkin norms}\label{ss:tl-norms} The analysis in \S\ref{sec:pp} relies on a  family of Triebel-Lizorkin norms adapted to the multiscale resolution $\mathcal M$ of $\Omega$. They are formulated efficiently using a family of local  $q$-functions associated to  the basis \eqref{e:triebelref}. 
For  $n\in \mathbb \R$, $n\leq k$,  $1\leq q\leq \infty,$ set
\begin{equation}
\label{e:sfon}
\mathrm{S}^{n}_{q,R} f(x)= \left\Vert \frac{|\varphi_Z(f)|}{\ell(Z)^n} \mathbf{1}_Z(x)\right\Vert_{\ell^q(Z\in \mathcal D(R))}, \qquad R\in \mathcal M.
\end{equation}
The parameter $k$ and the choice of basis $\mathfrak B(\Omega,k)$ are implicit in the notation. Of course, a different choice of basis will produce equivalent norms. 
In our context, the above defined $q$-functions are motivated by the estimate
\begin{equation}
\label{e:PsiZ}
 \Psi_Z(g,h)  \coloneqq \frac{1}{|Z|} \sum_{S \in \mathcal{D}(Z)} |S| \left|\varphi_S(g)\right|\left|\varphi_S(h)\right| \leq \left\langle    \mathrm{S}^{n}_{q,Z}g\right\rangle_{p,Z} \left\langle \mathrm{S}^{-n}_{q',Z} h \right\rangle_{p',Z}, \qquad 
\end{equation} holding for $1\leq p,q \leq \infty, \, |n| \le k$, 
which will be used in connection with  \eqref{e:lsfo}.
When $q=2$, $\mathrm{S}^n_{2,R}$ is essentially the local square function of $|\nabla|^n f$ and is denoted simply by $\mathrm{S}^n_R$. 
As each $\{\varphi_Z:Z\in \mathcal S\}$ has $k+1$ vanishing moments, integration by parts and Littlewood-Paley theory reveal that 
\begin{equation}
\label{e:lsfo}  
\max\{p,p'\}^{-\frac12} \left \langle \mathrm{S}^{n}_R f \right\rangle_{p,R}  \lesssim  \inf_{\mathsf{P} \in \mathbb P^{k-n}}  \left\langle \nabla^{n} f - \mathsf P  \right\rangle_{p,R} \lesssim 
\max\{p,p'\}^{\frac12}\left\langle \mathrm{S}^{n}_R f \right\rangle_{p,R}\end{equation}
where $1<p<\infty,$ $ \mathbb P^{m}$ stands for the ring of (vector) polynomials of degree at most $m$, and the implied constants are absolute.
The scales of respectively non-homogeneous and homogeneous  Triebel-Lizorkin norms  associated to a subset $\mathcal R \subset \mathcal M$ are then defined as follows. For $n,m \in \mathbb R, n,m\leq k$ and $1 \le p,q \le \infty$,
	\begin{align} \label{e:F} \|f\|_{F^{n,m}_{p,q}(\mathcal R)} &\coloneqq \sup_{\substack{W \in \mathcal W \cap \mathcal R \\ 1 \le j \le \jmath}} \ell(W)^{-n-m} \left\vert \chi_{W,j}(f) \right\vert + \sup_{S \in \mathcal R} \ell(S)^{-m} \left\langle \mathrm{S}^{n}_{q,S} f \right\rangle_{p,S} , \\
	\label{e:Fdot}\|f\|_{\dot F^{n,m}_{p,q}(\mathcal R)} &\coloneqq \|\mathcal Sf\|_{F^{n,m}_{p,q}(\mathcal R)}.
	\end{align}
When $\Omega=\R^d$ and $\mathcal R=\mathcal D$ these norms are equivalent to the unified Morrey-Campanato-Besov-Triebel-Lizorkin norms introduced in \cite{sickel05morrey}. More precisely, the corresponding space $F^{n,m}_{p,q}(\mathcal D) $ coincides with the $F^{n,\frac md + \frac 1p}_{p,q}$ space of \cite{sickel05morrey}. In the context of $T\mathbf{1}$-theorems for singular integrals, the spaces  $F^{n,0}_{1,2}(\mathcal M)$ play a familiar role. They coincide with the space of those $f\in L^1_{\mathrm{loc}}(\Omega)$ with $|\nabla|^n f\in \mathrm{bmo}(\Omega)$ \rev{(see \cite{CDS} for the canonical definition of $\mathrm{bmo}(\Omega)$)} and are usually indicated by $\mathrm{bmo}^n(\Omega)$; we follow this convention \rev{of identifying $\mathrm{bmo}^n(\Omega)$ and $F^{n,0}_{1,2}(\mathcal M)$} below, see  for instance Proposition \ref{p:sparse2}. The John-Nirenberg inequality  tells us that the norms  $F^{n,0}_{1,2}(\mathcal M)$  and $ F^{n,0}_{p,2}(\mathcal M)$ are equivalent for each $1<p<\infty$. A similar identification of the homogeneous space $\dot F^{n,0}_{p,2}(\mathcal M)$ with  $\mathrm{BMO}^n(\Omega)$ may be performed.

\begin{remark}
The  embeddings
\begin{align}\label{e:Femb}
 \|f\|_{F^{n,m}_{p,q}(\mathcal R)} & \le \|f\|_{F^{n+u,m-u}_{r,s}(\mathcal R)}, \qquad u \ge 0, \ r \ge p, \ s \le q,
\\ 
\|f\|_{F^{n,m}_{p,q}(\mathcal R)} &\leq \|f\|_{\dot F^{n,m}_{p,q}(\mathcal R)} + \|f\|_{F^{u,v}_{1,\infty}(\mathcal R)}, \qquad u+v =m+n 
\label{e:Fembdot} 
\end{align}
 follow immediately  from the definition.  Furthermore, an application of \eqref{e:lsfo} and  simple inspection of the exponents leads to the equivalence
 \begin{equation}\label{e:Fsup} \|f\|_{F^{n,m}_{p,2}(\mathcal M)} \sim \sup_{W\in \mathcal W} \ell(W)^{-m} \langle [1+|\nabla|^n ]f \rangle_{p,\mathsf{w}W},   \qquad -mp \geq d.
 \end{equation}
\end{remark}

}

\addtocounter{other}{1} 
\subsection{Tree Structure}  \label{ss:ts} Analyzing the effect of the boundary $\partial \Omega$ on the smooth boundedness properties of CZ operators requires a tree structure on the  Whitney decomposition of $\Omega$. Hereby follows an axiomatic definition and some examples   satisfying our requirements. This is a slight modification of the general framework of uniform domains described in \cite{prats-thesis}.

\begin{definition}
An ordered list of cubes $\{Q_1,\ldots,Q_m\} \subset \mathcal W$ is a \textit{chain} from $Q=Q_1$ to $W=Q_m$ if each $\bar Q_j \cap \bar Q_{j+1} \ne \varnothing$, $j=1,\ldots,m-1$. If so, we write
	\[ [Q,W] = \{Q_1,\ldots,Q_m\}.\]
Furthermore, for $L\ge 1$,  a chain $[Q,W]$ is \textit{$L$-admissible} if 
\begin{enumerate}
	\renewcommand{\labelenumi}{(U\theenumi)}
	\item \label{adm1} $\displaystyle \sum_{P \in [Q,W]} \ell(P) \le L {\mathfrak d}(Q,W)$;
	\item \label{adm2} there exists $Q_W \in [Q,W]$ such that
	\begin{enumerate}
	\item \label{qs1}for all $P \in [Q,Q_W]$, $\ell(P) \ge L^{-1} {\mathfrak d}(Q,P)$,
	\item \label{qs2}for all $P \in [Q_W,W]$, $\ell(P) \ge L^{-1} {\mathfrak d}(P,W)$.
	\end{enumerate}
\end{enumerate}
Call $[Q,W]$ an \textit{oriented chain} if $Q_W=W$. Notice that $[Q,Q_W]$ and $[W,Q_W]$ are both oriented chains.\end{definition}

\begin{lemma}\label{lemma:adm}
If $[Q,W]$ is $L$-admissible and $Q_W \in [Q,W]$ satisfies \textnormal{(U\ref{qs1})} and \textnormal{(U\ref{qs2})} above, then for any $P \in [Q,Q_W]$ and $P' \in [Q_W,W]$,
	\[ {\mathfrak d}(Q,W) \sim {\mathfrak d}(P,P').\]
\end{lemma}
\begin{proof}
The upper estimate follows from the triangle inequality and the definition of $Q_W$:
	\[ {\mathfrak d}(Q,W) \le {\mathfrak d}(Q,P) + {\mathfrak d}(P,P') + {\mathfrak d}(P',W) \le L \ell(P) + {\mathfrak d}(P,P') + L \ell(P') \le (2L+1) {\mathfrak d}(P,P').\]
The lower estimate relies additionally on (U\ref{adm1}) so that
	\[ {\mathfrak d}(P,P') \le {\mathfrak d}(P,Q) + {\mathfrak d}(Q,W) + {\mathfrak d}(W,P') \le L \ell(P) + {\mathfrak d}(Q,W) + L \ell(P') \le (2L^2 +1){\mathfrak d}(Q,W)\] as claimed.
\end{proof}

\begin{definition} \label{def:uniform-domain}
A Whitney decomposition $\W$ of $\Omega$ is said to be \textit{$L$-uniform} if for every $Q,W \in \mathcal W$, there is an $L$-admissible chain $[Q,W]$ for which every subchain is also $L$-admissible. \rev{In this case, we use $[Q,W]$ to denote one such fixed (for convenience) $L$-admissible chain.}
The domain $\Omega$ is said to be a \textit{uniform domain} if for each $\mathsf w \ge 1$ there exists $L \ge 1$ such that every Whitney decomposition with parameter $\mathsf w$ is $L$-uniform. If such an $L$ exists, label it $L(\Omega,\mathsf w)$.
\end{definition}

Uniform domains have a prominent role in geometric analysis and function theory. They were introduced by F.~John in \cite{John61} and formalized by O.~Martio and J.~Sarvas in \cite{MS79} in their study of injectivity properties of functions. Since then, many equivalent characterizations have been given \cites{jones80,GO79}, which connect to the extension problem for various function classes \cite{jones80,jones81,holden,BD21}.

\rev{For us, the uniform domain property is most crucial in the triad formed by the Lemmata \ref{lemma:whitney}, \ref{lemma:avg-2}, and \ref{lemma:tele}. The latter allows us to integrate by parts along a chain, which can be converted into a single wavelet coefficient by Lemma \ref{lemma:avg-2}, as long as the cubes in such a chain satisfy suitable packing estimates which are proved in Lemma \ref{lemma:whitney} for admissible chains. Furthermore, admissible chains provides a natural tree structure for defining our Carleson measure norms; see \S \ref{ss:pp-norms} below.}
\addtocounter{other}{1}
\subsection{Packing estimates for uniform Whitney coverings}
Throughout this section, let $\Omega$ be a uniform domain and $\W$ a uniform Whitney decomposition of $\Omega$.
\begin{lemma}\label{l:reg}
	\[ \inf_{\substack{ x \in \partial \Omega \\ 0<r\lesssim \diam(\Omega)}} \frac{|B(x,r) \cap \Omega|}{r^d} > 0\]
\end{lemma}
\begin{proof}
We recall the simple argument from \cite{jones81}*{Lemma 2.3}. Let $0<a<b<\frac 12$ to be chosen later. Let $x \in \Omega$ and $r<\frac {1}{2b} \operatorname{diam}\Omega$. Since $\Omega$ is connected, there exists $y \in \Omega$ with $ar <|x-y|< br$. Let $Q,W$ be Whitney cubes satisfying
	\[ Q \subset B(x,ar), \quad y \in W.\]
By construction, $\mathfrak d(Q,W) \sim r$ and notice that $Q_W$ from (3) satisfies $\ell(Q_W) \gtrsim r$. By choosing $a,b$ small enough, one also has $Q_W \subset B(x,r)$ which proves the claim.
\end{proof}
Lemma \eqref{l:reg} implies that the boundary of a uniform domain has Lebesgue measure zero. We can say something more, that in fact it has Minkowski dimension strictly less than $d$.

\begin{lemma}\label{l:pack}
There exists $\ep>0$ such that the estimate
	\begin{equation}\label{e:pack}\sum_{\substack{Q \in \W:\\ \mathfrak d(Q,W) \le A \ell(W), \\  \ell(Q) \le t \ell(W)}} |Q| \lesssim A^{d-\ep} t^\ep |W|\end{equation}
holds uniformly over $W \in \W$, $A \ge 1$, and $t >0$.
\end{lemma}
\begin{proof}
For ease of notation, define $\B_A(W) = \{Q \in \W : \mathfrak d(Q,W) \le A \ell(W)\}$, and notice that if $Q \in \B_A(W)$, there is a chain $[Q,W]=\{Q_1,Q_2,\ldots,Q_m\}$ which satisfies
	\begin{equation}\label{e:near-chain} Q_i \in \B_D(Q_j), \quad i \le j, \quad D \sim A. \end{equation}
Indeed, for $i=1$ or $j=m$ \eqref{e:near-chain} is immediate from (U\ref{adm1}) and (U\ref{adm2}). Then the general case follows since every subchain can also be chosen to be admissible.

\vskip1mm \noindent\textit{Step 1: Logarithmic packing condition.}  We claim that
	\begin{equation}\label{e:log} \sum_{Q \in \B_A(W)} |Q| \log \left( \frac{|W|}{|Q|} \right) \lesssim_A |W|,\end{equation}
Let $k \ge 0$, and $Q \in \B_A(W)$ with $\ell(Q) \sim 2^k \ell(W)$. Construct a Whitney chain going from $Q$ back to $W$ and enumerate the elements $Q=Q_1,\ldots,Q_n=W$. By virtue of the Whitney decomposition, the neighboring cubes $Q_i$ and $Q_{i+1}$ satisfy
	\[ \frac{1}{\mathsf w} \le \frac{\ell(Q_i)}{\ell(Q_{i+1})} \le \mathsf w\]
for some $B \ge 1$. Therefore,
	\[ 2^k \lesssim \frac{\ell(Q)}{\ell(W)} = \prod_{i=1}^{n-1} \frac{\ell(Q_i)}{\ell(Q_{i+1})} \le \mathsf w^{n-1},\]
i.e. $n \gtrsim k$. Furthermore, from \eqref{e:near-chain}, we get that for each $i=1,\ldots,n$, 
	\[ Q \in \B_{D}(Q_i), \quad Q_i \in \B_{D}(W), \quad D \sim A.\]
In this way,
	\[ \sum_{Q \in \B_A(W)} |Q| \log \left( \frac{|W|}{|Q|} \right) \lesssim \sum_{k=0}^\infty \sum_{\substack{Q \in \B_A(W) \\ \ell(Q) \sim 2^k \ell(W)}} |Q| (k+1) \lesssim \sum_{P \in \B_{D}(W)} \sum_{Q \in \B_{D}(P)} |Q| \lesssim |W|,\]
establishing \eqref{e:log}

\vskip1mm \noindent\textit{Step 2: Proof of \eqref{e:pack} for $A \lesssim 1$.}
Given \eqref{e:log}, it is evident that there exists $t_0 \in (0,\frac 12)$ such that for all $W \in \W$ and $D \lesssim 1$,
	\begin{equation}\label{e:geo}\sum_{\substack{Q \in \Sh_D(W) \\ \ell(Q) \le t_0 \ell(W) }} |Q| \le \frac 12 |W|.\end{equation}
Then, let $t \in (0,1)$ and select a positive integer $k$ such that $t \in [t_0^{k+1},t_0^k)$. Suppose we can construct a collection $\{Q_i\}_{i=0}^n \subset \W$ with the following properties
	\begin{itemize}
	\item $Q=Q_n$ and $W=Q_0$,
	\item $Q_j \in \B_{D}(Q_{j-1})$, $j=1,\ldots,n$
	\item $\ell(Q_j) \le t_0 \ell(Q_{j-1})$. $j=1,\ldots,n$
	\item $k \lesssim n$.
	\end{itemize} 
With such a collection in hand, we can iterate \eqref{e:geo} to obtain
	\[\sum_{\substack{Q \in \B_A(W) \\ \ell(Q) \le t \ell(W)}} |Q| \le  \sum_{\substack{Q_{1} \in \B_{D}(Q_0) \\ \ell(Q_{1}) \le t_0 \ell(Q_0)}} \cdots \sum_{\substack{Q_{n} \in \B_{D}(Q_{n-1}) \\ \ell(Q_n) \le t_0\ell(Q_{n-1}) }} |Q_n| \le \left( \frac 12\right)^n |W|. \]
However, $(\frac 12 )^{n} \lesssim (t_0^k)^\ep  \lesssim t^\ep$ for $\ep^{-1} \sim -\log t_0$. It remains to construct the desired collection. Let $[Q,W]=\{P_1,\ldots,P_m\}$ satisfying \eqref{e:near-chain}. Begin with $Q_0=W$ and choose the rest inductively from $[Q,W]$. Given $Q_i=P_{j_0}$, let $Q_{i+1}=P_j$ where
	\[ j = \max\{ n \le j_0 : \ell(P_n) \le t_0 \ell(P_{j_0}) \}.\] 
If no maximal element exists then terminate the algorithm and reset $Q_i=Q$. The first and third properties are satisfied by construction, and the second property follows from \eqref{e:near-chain}. To verify the fourth property, first notice that
	\[ t_0^{k} \ge t \gtrsim \frac{\ell(Q)}{\ell(W)} = \prod_{i=0}^{n-1} \frac{\ell(Q_{i+1})}{\ell(Q_{i})}.\]
However, since $Q_{i+1}$ was the first cube in the chain to satisfy $\ell(Q_{i+1}) \le t_0 \ell(Q_i)$, we have, by the Whitney property
	\[ \ell(Q_{i+1})=\ell(P_j) \ge \frac{1}{\mathsf w} \ell(P_{j+1}) \ge \frac{t_0}{\mathsf w} \ell(Q_i),\]
Combing the two above displays and recalling that $t_0 <1$ establishes that $k \lesssim n$.

\vskip1mm \noindent\textit{Step 3: General $A$.}
To prove the case of general $A$, fix $W\in\W$, $A \ge 1$, and $t\ge 0$. There exists $\tilde W \in \W$ with
	\[ \mathfrak d(W,\tilde W)\lesssim A \ell(W), \quad \ell(\tilde W) \sim \sup_{Q,P \in \B_A(W)} \mathfrak d(Q,P).\]
Indeed, choose $W' = Q_P$ from (U\ref{adm2}) taken from an admissible chain connecting $Q$ and $P$ which are near maximizers of the supremum above. The second property above guarantees that $\B_A(W) \subset \B_B(\tilde W)$ where $B \lesssim 1$. Therefore, setting $s = \frac{t \ell(W)}{\ell(\tilde W)}$
	\[ \sum_{\substack{Q \in \B_A(W) \\ \ell(Q) \le t \ell(W)}} |Q| \le \sum_{\substack{ Q \in \B_B(\tilde W) \\ \ell(Q) \le s \ell(\tilde W)}} |Q| \lesssim |\tilde W|t^\ep \left( \frac{\ell(W)}{\ell(\tilde W)} \right)^\ep \lesssim A^{d-\ep} t^\ep |W|.\]
\end{proof}

\begin{lemma}\label{lemma:whitney}
The estimate 
\begin{equation}
 \label{eq:est1}
 \sum_{\substack{ P \in \mathcal W \\ \ell(Q) \le t \\ W \in [P,Q]}} \frac{|P| {\mathfrak d}(W,P)^{k-1}}{{\mathfrak d}(P,R)^{d+k}} \lesssim  \frac{|W
 |\ell(W)^{k-1} \min\{1,\frac{t}{\ell(W)}\}^\ep}{{\mathfrak d}(R,W)^{d+k}}
\end{equation}
holds uniformly over $t\geq 0$, $W,Q
\in \W, R\subset Q$ and $k\geq 1$.
\end{lemma}
\begin{proof}  In this proof $\vee,\wedge$ are used in place of $\max,\min$ respectively.
First, note the intermediate estimate
\begin{equation} \label{eq:est2}
\sum_{\substack{P \in \mathcal W \\ \ell(P) \le t }} \frac{|P|}{{\mathfrak d}(W,P)^{d+1}} \lesssim \frac{\left( 1 \wedge\frac{t}{\ell(W)} \right)^\ep }{\ell(W)}.
\end{equation}
To prove \eqref{eq:est2}, decompose 
	\[ \sum_{\substack{P \in \mathcal W \\ \ell(P) \le t }} \frac{|P|}{{\mathfrak d}(W,P)^{d+1}} \sim \ell(W)^{-d-1} \sum_{j=0}^\infty 2^{-j(d+1)} \sum_{\substack{P \in \mathcal W \\ {\mathfrak d}(W,P) \sim 2^j \ell(W) \\ \ell(P) \le t}} |P|. \]
 
By Lemma \ref{l:pack}, the innnermost sum (over $P \in \W$) is controlled by 
	\[ |W| 2^{j(d-\ep)} \left(\frac{t}{\ell(W)} \wedge 2^j \right)^\ep. \] 
 
Summing in $j$ yields \eqref{eq:est2}.
We return to   \eqref{eq:est1} and first prove it  under the assumption that $R=Q$. The general situation will be later reduced to this case. Let $P_{Q} \in [P,Q]$ satisfying (U\ref{adm2}). Split the sum in \eqref{eq:est1} into two pieces, say $I$ and $II$ such that $W \in [P,P_{Q}]$ and $W \in [P_{Q},Q]$, respectively. Estimate $II$ first,  directly relying on \eqref{eq:est2}. In fact  using ${\mathfrak d}(Q,P) \sim {\mathfrak d}(W,P)$ by Lemma \ref{lemma:adm}, the intermediate estimate \eqref{eq:est2}, and $\ell(W) \sim {\mathfrak d}(W,Q)$ from (U\ref{adm2}),
	\[ \begin{split}
	II &= \sum_{\substack{ P \in \mathcal W \\ \ell(P) \le t \\ W \in [P_{Q},Q]}} \frac{|P|  {\mathfrak d}(W,P)^{k-1}}{ {\mathfrak d}(P,Q)^{d+k}} \lesssim \sum_{\substack{P \in \mathcal W \\ \ell(P) \le t }} \frac{|P|}{{\mathfrak d}(W,P)^{d+1}} \lesssim \frac{\left( 1 \wedge\frac{t}{\ell(W)} \right)^\ep}{\ell(W)} \\ &\lesssim \frac{|W|\ell(W)^{k-1}\left( 1 \wedge\frac{t}{\ell(W)} \right)^\ep}{{\mathfrak d}(W,Q)^{d+k}} .\end{split}\]
To estimate $I$, notice that as consequence of (U\ref{adm2}) and \eqref{eq:est2},
	\begin{equation}\label{eq:est3} \sum_{\substack{P \in \mathcal W \\ \ell(P)\le t \\ W \in [P,P_Q] }} |P| \lesssim \ell(W)^{d+1}\sum_{\substack{P \in \W \\ \ell(P)\le t}} \frac{|P|}{\mathfrak d(P,W)^{d+1}} \lesssim |W| \left( 1 \wedge\frac{t}{\ell(W)} \right)^\ep.\end{equation}
So, for $W \in [P,P_{Q}]$, (U\ref{adm2}) implies ${\mathfrak d}(W,P) \sim \ell(W)$ and Lemma \ref{lemma:adm} implies ${\mathfrak d}(P,Q) \sim {\mathfrak d}(W,Q)$. Therefore, on applying \eqref{eq:est3},
	\[ I = \sum_{\substack{ P \in \mathcal W \\ \ell(P) \le t \\ W \in [P,P_{Q}]}} \frac{|P|{\mathfrak d}(W,P)^{k-1}}{{\mathfrak d}(P,Q)^{d+k}} \lesssim \frac{\ell(W)^{k-1}}{ {\mathfrak d}(Q,W)^{d+k}} \sum_{\substack{ P \in \mathcal W \\ \ell(P) \le t \\ W \in [P,P_{Q}]}} |P| \lesssim \frac{|W|\ell(W)^{k-1} \left( 1 \wedge\frac{t}{\ell(W)} \right)^\ep}{{\mathfrak d}(Q,W)^{d+k}}.\]
The estimate \eqref{eq:est1} is  thus established when $R=Q$. To show \eqref{eq:est1} holds for any $R \subset Q$, we only need to establish ${\mathfrak d}(R,W) \sim {\mathfrak d}(Q,W)$ for any Whitney cube $W$. The triangle inequality shows ${\mathfrak d}(R,W) \le 3 {\mathfrak d}(Q,W)$. For the other inequality, if we can show ${\mathfrak d}(R,W) \gtrsim \ell(Q)$ then again the triangle inequality will yield ${\mathfrak d}(Q,W) \lesssim {\mathfrak d}(R,W)$. If $\bar W \cap \bar{Q} \ne \varnothing$, then indeed ${\mathfrak d}(R,W) \ge \ell(W) \gtrsim \ell(Q)$ holds by the Whitney decomposition. Otherwise, $W \cap aQ = \varnothing$ for some $a>1$ so $|c(R)-c(W)| \gtrsim \ell(Q)$.
\end{proof}

\addtocounter{other}{1}
\subsection{Boundary windows}\label{ss:windows}
For a cube $W \in \W$, define the shadow of $W$ and its realization by
	\[ \SH(W) \coloneqq \{P \in \mathcal W : [P,W] \mbox{ is an oriented chain}\}, \qquad  \Sh(W) \coloneqq \bigcup_{P \in\SH(W)} \mathsf{w} P.\] 
Subsequently, introduce the extended shadow $\SH_1$ and its realization $\Sh_1$,
	\[ \SH_1({W}) = \{ P \in \W : P \in [Q,{W}] \mbox{ for some } Q \in \SH({W})\}, \quad \Sh_1({W}) = \bigcup_{P \in \SH_1({W})} \mathsf{w} P.\]
Our analysis relies on  a tree structure on $\Omega$ whose roots are constructed in the next lemma. With reference to the statement below, each $\SH({W}_j)$ is termed a \textit {boundary window}, while the cubes of the residual collection   
$ \W_0 = \W  \backslash (\cup_j \SH({W}_j)) $
are termed \textit {central}.
\begin{lemma}\label{l:windows}
Let $\W$ be an $L$-uniform Whitney covering with parameter $\mathsf{w}$ and $\ell$ be the side length of some fixed cube ${W}_0$ in $\W$. Then, there exists a collection of cubes $\{{W}_j\}_{j \ge 1} \subset \W$ satisfying
\begin{itemize}
	\item[(S1)] $\ell({W}_j) \sim \ell$;
	\item[(S2)] If ${W} \in \mathcal W$ but ${W}$ does not belong to $\cup_j \SH({W}_j)$, then $\ell({W}) \gtrsim \ell$;
	\item[(S3)] For each ${W}_j$, there exists ${W}_j^*$ with $\ell({W}_j^*) \sim \ell$, ${W}_j \in \SH({W}_j^*)$, and $2^2 \mathsf{w}{W}_j^* \cap \Sh({W}_j) = \varnothing$;
	\item[(S4)] The collections $\{\Sh({W}_j)\}$, $\{\Sh({W}_j^*)\}$, and $\{\Sh_1({W}_j^*)\}$ have finite overlap.
\end{itemize}
\end{lemma}

\begin{proof}
Let us construct such a collection. Let ${W}_0$ be an arbitrary cube in $\W$ with length $r_0=\ell({W}_0)$. For any ${W}$ in $\mathcal W$ with $\ell({W}) \le r_0$, let ${W}_{{W}_0}$ be chosen according to assumption (U\ref{adm2}). By (U\ref{adm2}b), $\ell({W}_{{W}_0}) \ge L^{-1} \ell({W}_0) = L^{-1} r_0$. Take $C$ to be the constant from the Whitney decomposition (W3). There exists $S \in [{W},{W}_{{W}_0}]$ such that $L^{-1}r_0 < \ell(S) \le C L^{-1} r_0$. In particular, this shows that for every ${W}$ with $\ell({W})\le r_0$, there exists $S$ satisfying ${W} \in \SH(S)$ and $L^{-1} r_0<\ell(S)<CL^{-1}r_0$. Define, for any $r>0$,
	\[ \mathcal E_{r} = \{S \in \W : L^{-1} r<\ell(S)<CL^{-1}r\}, \quad \mathcal F_r = \{ {W} \in \W : \ell({W}) \le r \}.\]
We have shown
	\begin{equation}\label{eq:fro} \mathcal F_{r_0} \subset \bigcup_{S \in \mathcal E_{r_0}} \SH(S).\end{equation}
${W}_j^*$ will be chosen from such $\mathcal E_{r_0}$. We will now choose $r_1$ so that for ${W}$ in $\mathcal E_{r_1}$ and $S$ in $\mathcal E_{r_0}$,
	\begin{equation}\label{eq:pqs}\mbox{if} \ P \in \SH({W}) \ \mbox{then} \ 2^2 \mathsf{w}S \cap 2\mathsf{w}P = \varnothing.\end{equation}
Indeed, if $2^2 \mathsf{w}S$ and $2\mathsf{w}P$ had nonempty intersection, then by properties (W3) and (W4) of the Whitney decomposition, $\ell(S) \le D \ell(P)$ for some $D>0$ depending only on $\mathsf{w}$ and $d$. Therefore, choosing $r_1 = (CDL)^{-1}r_0$, we arrive at the contradiction
	\[ D \ell(P) \le DL \ell({W}) \le CDr_1 = L^{-1} r_0 < \ell(S).\]
Therefore \eqref{eq:pqs} is established.
The same argument that led to \eqref{eq:fro} also shows
	\begin{equation}\label{eq:froo} \mathcal F_{r_1} \subset \bigcup_{\substack{P \in \mathcal E_{r_1}}} \SH(P).\end{equation}
Now we take $\{{W}_j\} = \mathcal E_{r_1} \subset \mathcal F_{r_0}$ and take each ${W}_j^*$ to be some $S \in \mathcal E_{r_0}$ such that ${W}_j \in \SH(S)$. Such an $S$ exists by \eqref{eq:fro} and (S3) is satisfied as a consequence of \eqref{eq:pqs}. Then (S1) is satisfied by the definition of $\mathcal E_{r_0}$ and (S2) follows from \eqref{eq:froo}. Finally (S4) is an elementary consequence of the fact that if $\ell(S) \sim \ell(S') \sim r$ and $P \in \SH(S) \cap \SH(S')$ (or $\SH_1(S) \cap \SH_1(S')$) then
$ {\mathfrak d}(S,S') \le {\mathfrak d}(S,P) + {\mathfrak d}(P,S') \lesssim r.$
\end{proof}

\addtocounter{other}{1} 
\subsection{Averaging Wavelets}
Our representation process involves certain averages of elements from $\Phi^{\sigma,\Subset}(Q)$ weighted by coefficients that decay with respect to  separation in scale and position of two cubes. The latter is measured by the quantities
	\[  \lb P,Q \rb_{\delta,\rho} \coloneqq \frac{   \min\{ \ell(P),\ell(Q)\}^\delta}{{\mathfrak d}(P,Q)^{d+\delta+\rho}} \] where $P,Q$ are any two cubes and  $\delta,\rho \ge 0$.
For a point $x \in \Omega$ and a cube $S \in \mathcal M$, the notation $x \sim S$ indicates that $x$ belongs to a fixed dilate of $S$, say $2^4 \mathsf w S$.

\begin{lemma}\label{lemma:avg-1}
Let $\W$ be a Whitney decomposition with parameter $\mathsf w$ and $\delta,\rho \ge 0$. If $\delta>0$, let $0<\eta<\delta$. There exists $C=C(\rho,\eta,\delta,\mathsf w)>0$ such that for any $z:\mathcal M \times \mathcal M \to \mathbb C$ satisfying
	\[ |z(R,S)| \le \lb R,S \rb_{\delta,\rho} \qquad \forall (R,S) \in \mathcal M \times \mathcal M\]
the following holds. If $\delta>0$ and $\varphi_{R} \in \Psi^{\delta,\Subset}(R)$, then
	\begin{equation}\label{eq:avg-psi}\psi_{S} \coloneqq \sum_{\substack{R \in \mathcal S}} |R| z(R,S) \varphi_{R} \in C \Psi^{\eta,\rho}(S).\end{equation}
If $\delta=0$, $\rho>0$, and $\chi_{P} \in \Phi^{0,\Subset}(P)$, then
	\begin{equation}\label{eq:avg-phi1} \phi_Q \coloneqq \sum_{\substack{P \in \W \\ R \in \mathcal D(Q)}} \frac{|R| |P|}{|Q|} z(R,P) \ell(R)^{\rho} \varphi_R \otimes \chi_{P} \in C \Phi^{0,\rho}_{2}(Q).\end{equation}
\end{lemma}
\begin{proof}
Estimate \eqref{eq:avg-psi} is a discrete version of \cite[Lemma 3.2, Eq. (3.4)]{diplinio22wrt} and may be proved along the same lines. To prove \eqref{eq:avg-phi1}, notice that $\phi_Q$ is supported near $Q$  in the first variable. Therefore, for $x \sim Q, y \in \Omega$,
	\[ |\phi_Q(x,y)| \lesssim \frac{1}{|Q|} \sum_{\substack{P\in \mathcal W \\ y \sim P}} \sum_{\substack{R \in \mathcal D(Q) \\ x \sim R}} \frac{\ell(R)^\rho}{\mathfrak d(R,P)^{d+\rho}}.\]
The key is that since $R \in \mathcal D(Q)$ and $P$ is a Whitney cube with $y \sim P$, $\mathfrak d(R,P) \sim |y-c(Q)| + \ell(Q)$. Therefore, summing the geometric series
	\[ \sum_{\substack{R \in \mathcal D(Q) \\ x \sim R}} \ell(R)^\rho \lesssim \ell(Q)^\rho\]
proves \eqref{eq:avg-phi1}.
\end{proof}

The second lemma is more involved and relies on the tree structure imposed on $\Omega$.
\begin{lemma}\label{lemma:avg-2}
Let $\ep>0$, $\mathsf w \ge 1$, $\Omega$ be a uniform domain, and $\W$ a Whitney decomposition of $\Omega$ with parameter $\mathsf w$. Let $\delta \ge 0$, $k \ge 1$, and $0<\rho < k$. Then, there exists $C=C(k,\rho,\ep,L(\Omega,\mathsf w),\mathsf w)>0$ such that for any $z:\mathcal W \times \mathcal W \times \mathcal M \to \mathbb C$ satisfying
	\[ |z({P},{W},{Q})| \le \left\{ \begin{array}{cc} \lb {P},{Q}\rb_{0,k} &\ell({Q}) \ge \ell({P}) \\ \lb {P},{Q} \rb_{\delta,k} & \ell({Q}) \le \ell({P}), \end{array} \right.\]
any $\chi_{{P},{W}} \in \Phi^{0,\Subset}({W})$, $\chi_{{Q}} \in \Phi^{0,\Subset}({Q})$, the following holds. If $\delta=0$, then
	\begin{equation} \label{eq:avg-theta}\theta_{W} \coloneqq \sum_{\substack{{P},{Q} \in \W \\ {W} \in [{P},{Q}]}}  \frac{|{P}| |{Q}|}{|{W}|} \ell({W}) {\mathfrak d}({W},{P})^{k-1}z({P},{W},{Q})\chi_{{P},{W}} \otimes \chi_{{Q}} \in C \Phi^{0,k}_2({W}).\end{equation}
If $\delta>0$, then
	\begin{equation} \label{eq:avg-phi2} \phi_{{W}} \coloneqq \sum_{\substack{{P} \in \mathcal W\\ R \in \mathcal {S} \\ {W} \in [{P},W(R)]}} \frac{|{P}| |{R}|}{|{W}|} \ell({W}) {\mathfrak d}({W},{P})^{k-1}z({P},{W},R) \chi_{{P},{W}} \otimes \chi_{R} \in \Phi^{0,\rho}_{2} ({W}). \end{equation}
\end{lemma}
\begin{proof}
Begin with proving \eqref{eq:avg-theta}. {S}ince $\chi_{{P},{W}}$ is supported near ${W}$, we fix $x$ near ${W}$ and $y \in \Omega$. We sum over all ${P}$ in the specified ranges, but only over ${Q}$ with $y\sim {Q}$. Therefore, ${\mathfrak d}({Q},{W}) \sim |y-c({W})|+\ell({W})$ so by \eqref{eq:est1} with $t = \infty$,
	\[ |\theta_{W}(x,y)| \lesssim \frac{\ell({W})}{\ell({W})^{2d}} \sum_{y \sim {Q}} \sum_{\substack{ {P} \in \mathcal W \\ {W} \in [{P},{Q}]}} \frac{\ell({P})^{d} {\mathfrak d}({W},{P})^{k-1}}{{\mathfrak d}({P},{Q})^{d+k}} \lesssim \frac{1}{\ell({W})^{2d}(1+\frac{|y-c({W})|}{\ell({W})})^{d+k}}. \]
The proof of \eqref{eq:avg-phi2} begins similarly, by letting $x,y \in \Omega$ with $x$ close to ${W}$. Write $\phi_{W}(x,y) = I + II$ where the summations in $I$ and $II$ are taken over the regions $\ell({R}) \le \ell({P})$ and $\ell({R}) \ge \ell({P})$, respectively. Then, relabelling $W({R})={Q}$,
\[ |I| \lesssim \frac{\ell({W})}{\ell({W})^{2d}}\sum_{y \sim {Q}} \sum_{\substack{{P} \in \W\\ {W} \in [{P},{Q}]}} \sum_{\substack{{R} \in \mathcal D( {Q})\\ y \sim {R} \\ \ell({R}) \le \ell({P})}} \frac{\ell({P})^d \ell({R})^\delta {\mathfrak d}({W},{P})^{k-1}}{ {\mathfrak d}({R},{P})^{d+k+\delta}}.\]
Summing first over ${R}$, we obtain
	\[ \sum_{\substack{{R} \in \mathcal D(Q) y \sim {R} \\ \ell({R}) \le \ell({P})}} \frac{\ell({R})^\delta}{{\mathfrak d}({R},{P})^{d+k+\delta}} \lesssim \frac{\min\{\ell({P}),\ell({Q})\}^\delta}{{\mathfrak d}({Q},{P})^{d+k+\delta}} \le \frac{1}{\mathfrak d({Q},{P})^{d+k}}.\]
But this simplifies the bound on $I$ to be the same as $\theta_{W}$ above, which has already been shown to belong to $\Phi^{0,k}_2({W})$. 
$II$ is estimated by applying \eqref{eq:est1} with ${Q}=W({R})$ and $t = \ell({R})$;
	\[ |II| \le \frac{\ell({W})}{\ell({W})^{2d}} \sum_{\substack{{R} \in \mathcal {S} \\ y \sim {R}}} \sum_{\substack{{P} \in \W \\ \ell({P}) \le \ell({R}) \\ {W} \in [W({R}),{P}] }} \frac{\ell({P})^{d} {\mathfrak d}({W},{P})^{k-1}}{\mathfrak d({P},{R})^{d+k}} \lesssim \frac{\ell({W})^k}{\ell({W})^d} \sum_{\substack{{R} \in \mathcal {S} \\ y \sim {R}}} \frac{ \min\{1,\frac{\ell({R})}{\ell({W})} \}^\ep}{ \mathfrak d({R},{W})^{d+k}} . \]
Notice that the $R$ which contribute to the sum satisfy $\ell(R) \le \mathfrak{d}(R,W) \sim \ell(W) + |y-c(W)|$. Therefore, splitting the sum into the regions $\ell(R) \lesssim \ell(W)$ and $ \ell(W) \lesssim \ell(R) \lesssim \ell(W) + |y-c(W)|$, one obtains
	\[ \sum_{y \sim R} \min \left\{ 1, \frac{\ell(R)}{\ell(W)} \right\}^\ep \lesssim 1 + \log\left(1+\frac{|y-c(W)|}{\ell(W)} \right).\]
So, \eqref{eq:avg-phi2} follows from the fact that $\log s \le \eta^{-1} s^\eta$ for any $s,\eta>0$.
\end{proof}

\addtocounter{other}{1} 
\subsection{Anti-Integration by Parts}
Throughout this paragraph,  $Q\in \mathcal W$ and $R\in \mathcal S$ with reference to the previously defined Whitney decomposition of $\Omega$. Fixing   a positive integer $k$ and $\delta>0$, we adopt the standing conventions that  $\chi_Q \in \Phi^{k,\Subset}(Q)$, and $\varphi_R \in \Psi^{k+\delta,\Subset}(R)$.

We will frequently need to convert wavelet coefficients of functions into wavelet coefficients of their higher order derivatives. For this reason, it will be helpful to subtract off a polynomial $\PM^{k}_Q f$ (of degree $k-1$) suitably adapted to $f$ on $Q$. If we were only interested in estimation, then by the Poincar\'e inequality,
	\[ |\chi_Q( f-\PM^{k}_Qf)| \lesssim \ell(Q)^{k} \langle \nabla^k f\rangle_Q. \]
Clever adjustments can be made, see \cite{prats-tolsa}, to telescope and handle the case when $\chi_Q$ is replaced by $\chi_{P}$ for $P \ne Q$. We will do something similar later in this section (Lemma \ref{lemma:tele} below).
However, since we seek equality in our representation, we must be slightly more precise. Fortunately, Poincar\'e inequalities follow from integral representations of Sobolev functions (consult any PDE textbook, for example \cite{mazya-book}*{\S1.1.10}). 
Let ${\theta}$ be a fixed smooth function on $\R^d$ with $\int_{\R^d} {\theta} =1$ and support in the unit cube, and set 
\[{\theta}_Q \coloneqq |Q|^{-1}\mathsf{Sy}_Q{\theta}.\] If $s\geq 1$ and  $f$ is sufficiently smooth, then using the Taylor formula, and a few changes of variable, the equality 
	\begin{equation}\label{eq:sob-func-rep} \begin{aligned} f(x) - \PM^{k}_Q f(x) &= \sum_{|\alpha|=k} \int_{sQ}k_{\alpha,Q}(x,y) \frac{\partial^\alpha f(y)}{|x-y|^{d-k}} \, {\mathrm{d} y}, \\
	\PM^{m}_Q f(x) &\coloneqq \sum_{|\alpha| \le m-1} \frac{1}{\alpha ! } \int_Q {\theta}_Q(y) \partial^\alpha f(y) (x-y)^\alpha \, {\mathrm{d} y}, \end{aligned} \end{equation} holds for holds for all $x$ for which $f$ is defined.
The kernel, $k_{\alpha,Q}(x,y)$, has the explicit formula
	\[ k_{\alpha, Q}(x,y) = \mathbf 1_{sQ}(y) \frac{(-1)^k }{\alpha !  } v^\alpha \int_{|x-y|}^\infty {\theta}_Q(x+rv) r^{d-1} \, \mathrm{d}r, \quad v = \frac{y-x}{|y-x|}, \quad x,y \in \mathbb R^d,\]
but all we will need to know is that $|k_{\alpha,Q}(x,y)| \lesssim \left( 1 + \frac{|x-c(Q)|}{\ell(Q)} \right)^{d-1}$. Defining 
	\begin{equation}\label{eq:anti} \chi^{-\alpha}_{Q}(y) = \ell(Q)^{-|\alpha|}\int_{\mathsf{w} Q}k_{\alpha,Q}(x,y)\frac{\chi_Q(x)}{|x-y|^{d-|\alpha|}} \, \mathrm{d} x, \quad \chi^{-k}_Q = ( \chi^{-\alpha}_Q )_{|\alpha|=k},\end{equation}
one has $\chi^{-k}_Q \in C\Phi^{0,\Subset}(Q)$, and we arrive at  the anti-integration by parts formula
	\begin{equation}\label{eq:pot-q} \chi_Q( f-\PM^{k}_Q f) = \sum_{|\alpha|=k}\chi_{Q}^{-\alpha}(\partial^\alpha f) \ell(Q)^k = \ell(Q)^{k}\chi^{-k}_Q(\nabla^k f).\end{equation}
Equation \eqref{eq:pot-q} is our replacement for the Poincar\'e inequality. This discussion also makes it easy to see that if $f$ is paired with a cancellative wavelet, $\psi_R \in \Psi^{k+\delta,\Subset}(R)$, then the anti-integration can be done without the polynomial. Indeed, defining $\psi^{-\alpha}_R$ and $\psi^{-k}_R$ analogously to \eqref{eq:anti} 
	\begin{equation}\label{eq:ibp-psi} \psi_{R}(f) = \psi_R( f - \PM^{k}_{R}) = \sum_{|\alpha|=k}\psi_R^{-\alpha}( \partial^\alpha f) \ell(R)^k = \ell(R)^k\psi^{-k}_R(\nabla^k f).\end{equation}
As for $\chi_Q^{-k}$ above, $\psi^{-k}_R \in C\Psi^{\delta,\Subset}(R)$. For our later purposes though, we need to go one step further and consider the situation where the polynomial and the wavelet in \eqref{eq:pot-q} are adapted to \textit{different} cubes. 
\begin{lemma}\label{lemma:tele} Let ${P},{Q} \in \mathcal W$. For each $\chi_{P} \in \Phi^{k,\Subset}({P})$ and ${W} \in [{P},{Q}]$, there exists $\chi^{-k}_{{P},{W}} \in \Phi^{0,\Subset}({W})$ such that
\begin{equation}\label{eq:tele} \chi_{P}(f -\PM^{k}_{{Q}}f) =  \sum_{{W} \in [{P},{Q}]} \ell({W})\mathfrak d({W},{P})^{k-1}\chi^{-k}_{{P},{W}}(\nabla^k f).\end{equation}
\end{lemma}
To prove this, we need the following ingredients.
\begin{lemma}\label{l:poly}
Let $Q \in \mathcal M$, $x \in \Omega$, and $0 < m \le k+1$. There exist \rev{$C>0$, depending on $k$, $d$, and $\mathsf w$,} and $\upsilon_{Q,x,m} \in C \Phi^{k+1,\Subset}(Q)$, depending on $x$, $Q$, and $m$ such that
	\begin{equation}\label{e:ibp3} \PM^{m}_Q(f)(x) = \left(1+\frac{|x-c(Q)|}{\ell(Q)} \right)^{m-1} \upsilon_{Q,x,m} (f), \quad \forall f \in \mathcal U(\Omega).\end{equation}
Let $P,Q \in \mathcal M$ such that $P \subset CQ$. Then, for any $S \in \mathcal M$, there exists $\chi^{-k}_{S,Q} \in C\Phi^{0,\Subset}(Q)$ such that
	\begin{equation}\label{e:neighbor} \chi_S(\PM_P^k f - \PM_Q^k f ) =  \ell(Q) \mathfrak{d}(Q,S)^{k-1} \chi^{-k}_{S,Q}(\nabla^k f), \quad \forall f \in \mathcal U(\Omega).\end{equation}
\end{lemma}
\begin{proof}
Fix $x\in \Omega$.  
Notice that
	\[ \upsilon_{Q,x,m}  := \frac{1}{\left(1+\frac{|x-c(Q)|}{\ell(Q)} \right)^{m-1}} \sum_{|\alpha| \le m-1} \frac{(-1)^\alpha}{\alpha!}  \partial^{\alpha} \left[ (x-\cdot)^\alpha\theta_Q \right] 
				\in C\Phi^{k+1,\Subset}(Q). \]
Then \eqref{e:ibp3} follows integrating by parts.
To prove the second statement, notice that the polynomials $\PM$ have the following reproducing property. For $n \ge m$, and any cubes $R,Q$,
	\begin{equation}\label{e:poly-rep} \PM^n_R(\PM^m_Q f) = \PM^m_Q f.\end{equation}
Indeed,
	\begin{equation}\label{e:poly-rep2} \PM^n_R(\PM^m_Q f)(x) = \sum_{\substack{|\alpha| \le n-1 \\ |\beta| \le m-1} } \frac{1}{\alpha ! \beta !} \int_R \int_Q \theta_R(y) \theta_Q(z) \partial^\beta f(z) \partial^\alpha_y\left[ (y-z)^\beta \right](x-y)^\alpha \, \mathrm{d} z \, \mathrm{d} y. \end{equation}
However, since $|\beta| \le m-1 \le n-1$, by Taylor expansion,
	\[ \sum_{|\alpha| \le n-1} \frac{1}{\alpha !} \partial^\alpha_y \left[ (y-z)^\beta \right] (x-y)^\alpha = (x-z)^\beta \]
and the $y$ variable  in \eqref{e:poly-rep2} may be integrated out using $\int \theta_R = 1$. Therefore, for each $x \in \mathbb R$
	\[ \PM_P^k f(x) - \PM_Q^k f(x) = \PM_{CQ}^k ( \PM_P^k f - \PM_Q^k f )(x) = \left(1+ \frac{|x-c(Q)|}{\ell(Q)} \right)^{k-1}\upsilon_{CQ,x,k}( \PM_P^k f - \PM_Q^k f ).\]
Adding and subtracting $f$ inside the argument of $\upsilon_{CQ,x,k}$ and applying \eqref{eq:pot-q}, 
	\[ \upsilon_{CQ,x,k}( \PM_P^k f - \PM_Q^k f ) = \ell(Q)^{k} \upsilon_{CQ,x,k}^{-k}(\nabla^k f).\]
Finally, setting
	\[ \chi_{S,Q}^{-k} := \frac{1}{\mathfrak{d}(Q,S)^{k-1}} \int_S  (\ell(Q)+|x-c(Q)|)^{k-1}\upsilon_{CQ,x,k} \chi_S(x) \, \mathrm{d} x \in C \Phi^{k,\Subset}(Q)\]
establishes \eqref{e:neighbor}.
\end{proof}
Equipped with Lemma \ref{l:poly}, the proof of Lemma \ref{lemma:tele} is immediate.
\begin{proof}[Proof of Lemma \ref{lemma:tele}]
Enumerate the elements of the chain so that $[{P},{Q}]$ equals  $\{{P}_0,{P}_1,\ldots,{P}_N\}$ with ${P}={P}_0$ and ${Q}={P}_N$. As a telescoping sum,
	\[	\chi_{P}( f -\PM^{k}_{{Q}}f) = \chi_{P}( f-\PM^{k}_{{P}}f) + \sum_{i=1}^{N} \chi_{P}( \PM^{k}_{{P}_i}-\PM^{k}_{{P}_{i-1}}f).\]
The first term is handled using \eqref{eq:pot-q} above. Then apply \eqref{e:neighbor} to each summand. \end{proof}

\section{Sparse and weighted domination of wavelet forms}\label{s:sparse}

\addtocounter{other}{1}

\subsection{Sparse operators and multilinear maximal functions}Let $\eta>0$. A  collection  $\mathcal Q$  of  cubes of $\R^d$ is \textit{$\eta$-sparse} if
\[
\inf_{Q\in \mathcal Q }\frac{|E_Q|}{|Q|} \geq \eta, \qquad\textnormal{ where }  E_Q\coloneqq Q \setminus \bigcup_{\substack{ Q'\in \mathcal Q\setminus \{Q\} \\ \rev{Q' \subset Q}}} Q'.
\]
Below, the parameter $\eta$ may vary at each occurrence without explicit mention, and $\eta$-sparse collections are simply referred to as \textit{sparse collections}. Let $m\geq 1$. Associated to a sparse collection $\mathcal Q$ and an exponent vector $\vec{q}\in (0,\infty)^{m}$ is the $m$-linear \textit{sparse operator}
\[
\mathcal Q^{\vec q}(f_1,\ldots, f_{m}) (x) \coloneqq \sum_{Q\in \mathcal Q}  \mathbf{1}_Q \prod_{j=1}^{m} \langle f_j \rangle_{q_j,Q}, \qquad x\in \R^d
\]
acting on $m$-tuples of functions in $L^\infty_0(\R^d)$. Sparse operators, originating in the work of Lerner \cite{lerner-elementary,lerner-a2-13} on the $A_2$ conjecture, dominate Calder\'on-Zygmund operators on $\R^d$, in norm, pointwise or dual form sense. Furthermore, when $q_j=1$ for all $1\leq j\leq m$, each $\mathcal Q^{\vec q}$ coincides with a standard Calder\'on-Zygmund operator on the cone of positive functions.
The treatise  \cite{lerner2019intuitive} provides a thorough introduction to the subject. See  \cite{BFP,conde-alonso17,CDPOMRL,HRT, lacey-a2-17,Ler2016,BCp,culiuc2018domination,CDPV,LSph} for applications within and beyond Calder\'on-Zygmund theory.
The arguments of this section involve two different    multilinear maximal operators associated to an exponent vector $\vec{p}\in (0,\infty)^n$. The first is
the standard multilinear maximal operator 
\begin{equation}
\label{e:stmul}
\mathrm{M}^{\vec p}(f_1,\ldots, f_n) \coloneqq 
\sup_{Q \textrm{ cube of } \R^d} \mathbf{1}_Q \prod_{j=1}^n \langle f_j \rangle_{p_j,Q}. 
\end{equation}
The second is   built upon  the collection  of cubes well inside $\Omega$,
namely
\begin{equation}
\label{e:locmul}
\mathrm{M}^{\vec p}_{{\Omega^\circ}}(f_1,\ldots, f_n) \coloneqq 
\sup_{Q \in {{\mathcal Q}_{\Omega}^\circ}} \mathbf{1}_Q \prod_{j=1}^n \langle f_j \rangle_{p_j,Q}, \qquad {{\mathcal Q}_{\Omega}^\circ}\coloneqq\{Q \textrm{ cube of }\R^d: 3Q\subset \Omega\}.
\end{equation}
 In either case we remove the superscript $\vec p $ when $\vec p = (1,\ldots,1)$. Observe here that the \textit{compression} $\mathrm{M}^{\vec p}_\Omega $  satisfies
\begin{equation}
\label{e:compdefM}
\mathrm{M}^{\vec p}_\Omega (f_1,\ldots, f_n) \coloneqq \mathbf{1}_\Omega\mathrm{M}^{\vec p}(f_1\mathbf{1}_\Omega,\ldots, f_n\mathbf{1}_\Omega) \geq \mathrm{M}^{\vec p}_{{\Omega^\circ}} (f_1,\ldots, f_n) 
\end{equation} pointwise.
 The connection between \eqref{e:stmul}, \eqref{e:locmul} and sparse operators is  summarized  in
\begin{align}
\label{e:equival1} &
\sup_{\mathcal Q \textrm{ sparse}}\langle  \mathcal Q^{\vec q}(f_1,\ldots, f_{n-1}), |f_n|\rangle  \sim \|\mathrm{M}^{(q_1,\ldots, q_{n-1},1)}(f_1,\ldots, f_n) \|_{L^1(\R^d)}, \\ 
& \label{e:equival2}\sup_{\substack{\mathcal Q \textrm{ sparse}\\ \mathcal Q \subset {{\mathcal Q}_{\Omega}^\circ}}}  \langle  \mathcal Q^{\vec q}(f_1,\ldots, f_{n-1}), |f_n|\rangle \sim \|\mathrm{M}^{(q_1,\ldots, q_{n-1},1)}_{{\Omega^\circ}}(f_1,\ldots, f_n) \|_{L^1(\Omega)}.
\end{align}
Equivalence \eqref{e:equival1} is proved in \cite{CDPOBP}, and \eqref{e:equival2} follows with obvious modifications.

\addtocounter{other}{1}

\subsection{Sparse domination of wavelet forms}

The equivalences \eqref{e:equival1}, \eqref{e:equival2} and the customs in current literature motivate us to refer to the inequalities below as \textit{sparse domination principles} for wavelet forms.
First, fully localized wavelet forms on $\Omega$ admit sparse domination with cubes well inside $\Omega$. 
\begin{proposition} \label{p:sparse2} Let  $\Lambda$ be a fully localized wavelet form, as in Definition \ref{def:forms}. Then
\[
 \left|\Lambda (f_1,\ldots,f_n)\right| \lesssim \begin{cases}   \|\mathrm{M}_{{\Omega^\circ}}(f_1,\ldots, f_n) \|_{L^1(\Omega)}  &n\geq 2\\ 
 \|f_1\|_{\mathrm{bmo}(\Omega)}
 \|\mathrm{M}_{{\Omega^\circ}}(f_2,\ldots, f_n) \|_{L^1(\Omega)} & n \geq 3. \end{cases}
  \]
\end{proposition}
\rev{Recall that the $\mathrm{bmo}(\Omega)$ norm is equivalent to our nonhomogeneous Triebel-Lizorkin norm $F^{0,0}_{1,2}(\mathcal M)$ defined in \eqref{e:F}.} The arguments leading to Proposition \ref{p:sparse2} are by now standard in the literature. In particular, by virtue of the fully localized assumption,  each  component $\Lambda_W$ of $\Lambda$, cf.\ \eqref{e:lambdaW}, is supported   in a moderate dilate of $W\in \mathcal W$. One may thus repeat step by step the proof of \cite[Prop. 5.1]{diplinio23bilin} for each  component $\Lambda_W$ and obtain a sparse collection $\mathcal Q$ contained in ${{\mathcal Q}_{\Omega}^\circ}$ as claimed. 

Outside of the fully localized case, we may obtain a sparse domination estimate by looking at a wavelet form as a Calder\'on-Zygmund singular integral form on all of $\R^d$.  
\begin{proposition} \label{p:sparse} Let  $\Lambda$ be a  wavelet form as in \eqref{e:triebel2}. Then 
\begin{align}
 \label{e:sparzglob} 
 &\left|\Lambda (f_1,\ldots, f_n) \right| \lesssim   \left\|\mathrm{M}(f_1,\ldots, f_n) \right\|_{L^1(\R^d)}.
\end{align}
\end{proposition}
\begin{proof} Estimate \eqref{e:sparzglob} follows by observing that $\Lambda$ is a $n$-linear Calder\'on-Zyg\-mund form on $\R^d$, and the latter class admits sparse domination, see e.g.\ \cite{lerner2019intuitive} and references therein. A direct proof may be obtained along the lines of \cite[Prop. 5.1]{diplinio23bilin}.
\end{proof}

\begin{remark} \label{r:Lp} The sparse domination of Proposition  \ref{p:sparse}
together with the estimate  \cite{lerner2009new}	\[ \left \Vert \M(f_1,\ldots,f_n ) \right\Vert_{L^1(\R^d)} \lesssim \prod_{j=1}^n \left\Vert f_j \right\Vert_{L^{p_j}(\R^d)}, \qquad 1 < p_j < \infty, \quad \sum_{j=1}^n \frac 1{p_j} =1\]
imply the corresponding full range of  Lebesgue space bounds for the wavelet forms  	\begin{equation}
\label{e:Lpbds} \left| \Lambda(f_1 ,\ldots,f_n) \right| \lesssim \prod_{j=1}^n \left\Vert f_j \right\Vert_{L^{p_j}(\R^d)}, \qquad 1 < p_j < \infty, \quad \sum_{j=1}^n \frac 1{p_j} =1.\end{equation}
\end{remark}

\addtocounter{other}{1}
\subsection{Weighted Estimates for Sparse Forms} The $n$-linear sparse forms on $\R^d$ obey a sharply quantified range of weighted norm inequalities, see e.g.\ \cite{lerner2019intuitive}. In this article, we are interested in weighted bounds for the compressions of sparse forms, and of the corresponding multilinear maximal operators to $\Omega$,
leading to the definition of several related weight constants which only depend on $w\mathbf{1}_\Omega$.
The first family of constants is obviously related to $\M_{{\Omega^\circ}}$ and defined by
	\begin{equation}\label{e:type3w} [w]_{\A_{p}(\Omega^\circ)} \coloneqq \sup_{Q \in {{\mathcal Q}_{\Omega}^\circ}} \langle w \rangle_{Q} \left\langle w^{-1} \right\rangle_{p'-1,Q} \qquad  [w]_{\RH_{s} (\Omega^\circ)} \coloneqq \sup_{Q \in {{\mathcal Q}_{\Omega}^\circ}} \langle  w \rangle_{Q}^{-1} \left\langle w \right\rangle_{s,Q}. \end{equation}
	When $\Omega=\R^d$,  the constants defined by \eqref{e:type3w} are the usual Muckenhoupt and reverse H\"older constants, and therefore are simply indicated by $[w]_{\A_p}, [w]_{\RH_s}$ respectively. It is of interest to note that the sharp form of the reverse H\"older inequality
\begin{equation}
\label{e:rhiom} [w]_{\RH_s(\Omega^\circ)} \lesssim 1, \qquad s= \frac{1}{2^{d+1} [w]_{\A_p(\Omega^\circ)}-1}
\end{equation}	continues to hold, as it can be seen by repeating the arguments of
\cite[Thm.\ 2.3]{HPR}.	
The constants \eqref{e:type3w} appear in the following  weighted estimates for $\M_{\Omega^\circ}.$
\begin{proposition}\label{p:type3w} Let $1<p<r<\infty$, $\sigma$ be the $p$-th dual of $w$, and $s=\left(\frac{r}{p} \right)'$. Then
	\[\begin{split}
	 \left\Vert \M_{{\Omega^\circ}}^{\vec p}(f_1,f_2) \right \Vert_{L^1(\Omega)}& \lesssim  \left\{ \begin{array}{ll} ~ [w]_{\A_{p}(\Omega^\circ)}^{\max\left\{1,\frac{1}{p-1} \right\} } & \vec p=(1,1) \\ ~
\left( [w]_{\mathrm{A}_p({\Omega^\circ})} [w]_{\mathrm{RH}_s({\Omega^\circ})} \right)^{\max\left\{\frac{r-1}{r-p},\frac{1}{p-1}\right\}} & \vec p =(1,r')  \end{array} \right\} \\ &\times \left\Vert f_1 \right\Vert_{L^p(\Omega,w)} \left\Vert f_2 \right\Vert_{L^{p'}(\Omega,\sigma )}.
\end{split}\]
Furthermore,
	\[ \inf_{r > 2} r^{\frac 12} \left\Vert\M_{{\Omega^\circ}}^{(1,r')}(f_1,f_2) \right \Vert_{L^1(\Omega)} \lesssim 
[w]_{\A_{p}(\Omega^\circ)}^{2\max\left\{1,\frac{1}{p-1} \right\} } \left\Vert f_1 \right\Vert_{L^p(\Omega,w)} \left\Vert f_2 \right\Vert_{L^{p'}(\Omega,\sigma)}.\]
\end{proposition}
\begin{proof}
The first estimate follows from the same  well-known arguments of the case $\Omega=\R^d$. Refer to \cite[Thm 3.1]{moen2012} for the first case and \cite[Proposition 6.4]{BFP} for the second case. In view of \eqref{e:rhiom}, the second estimate can be proved by arguing as in \cite[Proof of Corollary A.1]{conde-alonso17}.
\end{proof}

The $\A_p(\Omega^\circ)$ constants are too small however to control compressions of general singular operators. In fact, standard techniques \cite[Thm.\ IV.5.5]{gc-book} show that the weak-type estimate for the compression of the maximal function, defined in \eqref{e:compdefM}, can be characterized by
	\begin{equation}\label{e:type1weak} 
\left\| \M_{\Omega}: L^{p}(\Omega,w) \mapsto L^{p,\infty}(\Omega,w)  \right\| \sim 		[w]_{\A_p(\Omega)}^{\frac1p}, \qquad 1<p<\infty
\end{equation}
where  
	\begin{equation}\label{e:type1w} 
		[w]_{\A_p(\Omega)} \coloneqq \sup_{\substack{   |Q \cap \Omega| >0}} \langle \mathbf 1_\Omega w \rangle_{Q} \left\langle \mathbf 1_\Omega w^{-1} \right\rangle_{p'-1,Q} \end{equation}
In general, reverse H\"older inequalities of the type \eqref{e:rhiom} fail for the families \eqref{e:type1w} and the corresponding reverse H\"older classes
	\begin{equation}\label{e:type1rh}
		[w]_{\RH_s(\Omega)} \coloneqq \sup_{\substack{   |Q \cap \Omega| >0}} \langle \mathbf 1_\Omega w \rangle_{Q}^{-1} \left\langle \mathbf 1_\Omega w \right\rangle_{s,Q}.
	\end{equation} 
In particular, it is not true for general $\Omega$ that 
	\begin{equation}\label{e:open} \A_p(\Omega) = \bigcup_{1<q<p} \A_q(\Omega).\end{equation}
For this reason, it is not possible to interpolate the weak-type constants into strong type as in \cite{Bu}, and Buckley's estimate  $\|\M\|_{L^p(\R^d, w)}^{p-1}\sim [w]_{\A_p}$ from \cite{Bu} fails to extend to  the compression $\M_\Omega$ in terms of  $[w]_{\A_p(\Omega)} $. Furthermore, an unpublished result of Wolff \cite{wolff-Ap} states that the identity \eqref{e:open} is equivalent to the extension problem for $w \in \A_p(\Omega)$, namely whether there exists $v \in \A_p(\R^d)$ such that $\mathbf 1_\Omega v=\mathbf 1_\Omega w$. See \cite{gc-book}*{Theorem IV.5.5} for a proof and \cite{kurki22,holden} for extensions. Notice also that if such an extension exists, then weighted estimates for the compression immediately follow from weighted estimates for the uncompressed operator.

The main result of this section is now to show that \textit{on uniform domains} $\Omega$, both the compression of maximal functions and sparse forms have sharp quantitative bounds in terms of $[w]_{\A_p(\Omega)}$. We will avoid the identity \eqref{e:open}; in fact, we wonder whether \eqref{e:open} holds for $\Omega$ uniform.

\begin{proposition}\label{p:sp-comp}
Let $\Omega$ be a uniform domain. Then, for any sparse operator $\mathcal Z$, $1<p<\infty$, and $w\in \A_p(\Omega)$,
	\begin{equation}\label{e:sp-comp} \norm{\mathcal Z_\Omega f}_{L^p(\Omega,w)} \lesssim [w]_{\A_p(\Omega)}^{\max\{1,\frac{1}{p-1} \}} \left\Vert f\right\Vert_{L^p(\Omega,w)}.\end{equation}
Furthermore,
	\begin{equation}\label{e:max-comp} \left \Vert \M_\Omega f \right\Vert_{L^p(\Omega,w)} \lesssim [w]_{\A_p(\Omega)}^{\frac{1}{p-1}} \left\Vert f \right\Vert_{L^p(\Omega,w)}.\end{equation}
\end{proposition}
\begin{proof}
  For this proof, let $\mathcal Q $ be the collection of cubes of $\R^d$ which intersect $\Omega$ on a set of positive measure.  For each  $Q\in \mathcal Q$, denote by $\widehat Q$ the cube concentric to $Q$ and whose sidelength is $3 \ell(Q)$. If $Q\in \mathcal Q$ and $Q\cap \partial \Omega \neq \varnothing$, the cube $\widehat Q$ contains a ball of radius $\ell(Q)$ centered at a point $x\in \partial \Omega$. Indeed, let $x \in Q \cap \partial\Omega$ and $y \in \mathbb R^d$ such that $|y-x| \le \ell(Q)$. Then, for each  coordinate $|y_j-c(Q)_j| \le |y-x| + |x_j-c(Q)_j| \le \frac 32 \ell(Q)$, which places $y \in \widehat Q$. Applying Lemma \ref{l:reg} ensures the existence of a  constant $\rho>0$, depending only on the uniform character of $\Omega$, such that $|\widehat Q \cap  \Omega| > \rho |Q|$. This estimate is trivially true, with $\rho= 1$, if $Q\in \mathcal Q$ and $ Q\cap \partial \Omega =\varnothing$. Below we may thus use that 
	\begin{equation}\label{e:Qhat} |\widehat Q \cap  \Omega| > \rho |Q| , \quad  \forall Q\in \mathcal Q. \end{equation}

 Further, recall from \cite{lerner2019intuitive} that there exist $3^d$ shifted dyadic grids $ \mathcal D_{t}, t  \in \left\{0,\pm1 \right\}^d\ $ with the following property: for each cube $L\subset \mathbb \R^d$ there exists $t(L)\in\left\{0,\pm1 \right\}^d $ and a cube  $R (L) \in \mathcal D_{t(L)}$, the \textit{regularized of $L$}, with
$
 L \subset R(L) $ and $ \ell(R( L)) \leq 3 \ell(L).
$
With this in hand, for any $Q \in \mathcal Q$ we may denote by $\widetilde Q \coloneqq R({\widehat Q})$ the regularized  of  ${\widehat Q}$ and
\[
\mathcal Q_t\coloneqq\{Q\in \mathcal Q:t(\widehat Q) = t\}.
\]
Replacing $\rho $ by ${3}^{-d}\rho$, we have that 
$|\widetilde  Q \cap \Omega|> \rho |Q|$ for all $Q\in \mathcal Q_t$ and $t\in \{0,\pm1\}^d$. 
%We will introduce modified versions of the sparse and maximal operators. First construct
%	\[ \tilde{\mathcal D} = \bigcup_{t \in \left\{0,\frac{1}{3},\frac{2}{3} \right\}^d } \mathcal D + t,\]
%which by the well-known three lattice satsfies the following property. Given a cube $Q \subset \mathbb R^d$, there exists $\tilde Q \in \tilde{\mathcal D}$ which approximates 
%$Q$ in the sense that $Q \subset \tilde Q$ and $|Q| \sim |\tilde Q|$.
and define
	\[ \mathsf R_t = \{\widetilde Q \cap \Omega : Q\in \mathcal Q_t\}, \qquad t \in  \{0,\pm1\}^d.\]
%Upon choosing the implicit constant correctly (in terms of $L$ from the definition of uniform domain and the dimension $d$), for any given cube $Q \subset \mathbb R^d$ such that $|Q \cap \Omega| > 0$, we can find $R \in \mathsf R$ such that
By construction, for each $Q\in \mathcal Q_t$ there exists $R=\widetilde Q \cap \Omega \in \mathsf R_t$ such that
	\begin{equation}\label{e:appR} Q \cap \Omega \subset \widehat Q \cap \Omega \subset \widetilde Q\cap \Omega ,\qquad  |Q|\gtrsim |R|\geq \rho |Q|.\end{equation}
%Indeed, choose $\tilde Q \in \tilde {\mathcal D}$ which approximates $3Q$ and by Lemma \ref{l:reg},
%	\[ |Q| \gtrsim |\tilde Q \cap \Omega| \ge |3Q \cap \Omega| \gtrsim |Q|.\]
Introduce the maximal function associated to $\mathsf R_t$ by
	\[ \M_{\mathsf R_t} f = \sup_{R \in \mathsf R_t} \mathbf 1_{R} \avg{f}_{R}.\]
Standard arguments \cite{lerner-elementary} based on the dyadic nature of the collection $\mathsf R_t$ show that \[\norm{\M_{\mathsf R} f}_{L^p(w)} \lesssim [w]_{\A_p(\mathsf R)}^{\frac{1}{p-1}} \norm{f}_{L^p(w)}, \qquad   
	[w]_{\A_p(\mathsf R)} \coloneqq \sup_{t} \sup_{R \in \mathsf R_t} \avg{w}_{R} \avg{\sigma}_{R}^{p-1}.\]
However, \eqref{e:appR} shows that $\M_\Omega f \lesssim \sup_t \M_{\mathsf R_t} f$ and $[w]_{\A_p(\mathsf R)} \lesssim [w]_{\A_p(\Omega)}$, and \eqref{e:max-comp} is proved.

A similar, though more involved line of reasoning will be used to establish the estimate for the sparse operators \eqref{e:sp-comp}. Before we begin, let $\mathsf R\coloneqq \bigcup_{t\in \{0,\pm1\}^d}\mathsf{R}_t$.
In analogy with the definition for cubes, say  $\mathcal R \subset \mathsf R$ is $\eta$-\textit{sparse} if for each $R \in \mathcal R$, there is an associated subset $E_R \subset R$ such that
	\[ |E_R | \geq \eta |R|, \quad \{ E_R: R \in \mathcal R\} \mbox{ is pairwise disjoint.}\]
If $t\in \{0,\pm1\}^d $ and $\mathcal R \subset \mathsf R_t$ is $\eta$-\textit{sparse} for some $\eta>0$, well-known arguments \cite{moen2012} give the bound
	\begin{equation}\label{e:RSPbounds} \norm{\mathcal R f\coloneqq  \sum_{R \in \mathcal R} \avg{f}_R \mathbf 1_R }_{L^p(\Omega,w)} \lesssim_\eta [w]_{\A_p(\mathsf R)}^{\max\{1,\frac{1}{p-1} \}} .\end{equation}
	Notice that just as for the maximal function, we can dominate the averages $\avg{f_j \mathbf 1_\Omega}_{Q}$ with $\avg{f_j \mathbf 1_\Omega}_{\widetilde{Q}}$
and therefore
	\[ \left\langle \mathcal Z_\Omega f_1,f_2 \right\rangle \lesssim \sum_{t \in \{0,\pm 1\}^d} \left\langle [\mathcal Z_t]_{\Omega}f_1,f_2 \right\rangle, \quad \mathcal Z_t = \left\{ \widetilde{Q} : Q \in \mathcal Z, \ t(Q)=t\right\}. \] 
%Each $\widetilde Q$ satisfies $|\widetilde Q| \sim |\wildetilde Q \cap \Omega | \sim |Q|$ by \eqref{e:appR}. 
In this way, the crucial property, is that for each $t \in \{0,\pm 1\}^d$,  
	\begin{equation}\label{e:ZQdoubling} Q \in \mathcal Z_t \implies Q \cap \Omega \in \mathsf R_t, \quad |Q \cap \Omega| \sim |Q|, \end{equation}
where the final property is simply a restatement of \eqref{e:appR}.
We would be finished if the exceptional sets $E_Q$ were contained in $\Omega$ but this may not be the case. So we will rehash the standard sparse algorithm with this end in view. Let us make two final reductions. Standard limiting arguments allow us to assume $\mathcal Z_t$ is a finite collection and thus is contained in some large cube. Second, by splitting into the subcubes of maximal elements of $\mathcal Z_t$, we may restrict our attention to the case 
	\begin{equation}\label{e:ZQ} \mathcal Z_t = \mathcal Z_t(Q_0) \coloneqq \{ Q \in \mathcal Z_t : Q \subset Q_0\}, \quad Q_0 \in \mathcal Z_t.\end{equation}
Taking all this into account, the remaining estimate \eqref{e:sp-comp} in Proposition \ref{p:sp-comp} will be an immediate consequence of \eqref{e:RSPbounds} and the bound
	\begin{equation}
		\label{e:sparse34}
\left\vert \left\langle [\mathcal Z_t(Q_0)]_\Omega f_1,f_2 \right\rangle \right\vert \lesssim \left\vert \left\langle \mathcal R f_1,f_2 \right\rangle \right\vert, 	\end{equation} 
for a suitably constructed sparse collection $\mathcal R \subset \mathsf R_t$ possibly depending on $f_1,f_2\in L^\infty(\R^d)$, nonnnegative and supported on $Q_0 \cap \Omega$.
Let us fix $t,f_1,f_2$ and henceforth we simply write $\mathcal Z$ for $\mathcal Z_t$. Let $\Theta> 1$ stand for a large constant which will be chosen at the end of the proof, and define
	\begin{equation}
		\label{e:badcube} F=F(Q_0)\coloneqq \bigcup_{j=1,2}
\left\{ x\in Q_0: \mathrm{M}(f_j\mathbf{1}_{Q_0})(x) > \Theta \langle f_j \rangle_{Q_0}   \right\}.
	\end{equation} 
The maximal theorem entails the estimate
\begin{equation}
	\label{e:Festimate}|F| < \frac{6^{d}}{\Theta} |Q_0|.
\end{equation}
Let $\mathcal S(Q_0)$ be maximal (disjoint, since $\mathcal Z$ is all from one dyadic grid) cubes $S$ from $\mathcal Z$ which are contained in $F$.
Introduce at this point 
	\[ \mathcal S = \mathcal S(Q_0), \quad \mathcal N \coloneqq \mathcal Z(Q_0) \setminus \bigcup_{S\in \mathcal S} \mathcal Z(S),\]
recalling $\mathcal Z(S)$ from \eqref{e:ZQ}.
Now rewrite  
	\begin{equation}
		\label{e:befpack}
\begin{split} 
		\left\langle \left[\mathcal Z(Q_0)\right]_\Omega f_1,f_2 \right \rangle  = \sum_{Q \in \mathcal N} \avg{f_1}_Q  \avg{f_2}_Q |Q|		+ \sum_{S \in \mathcal S}\left\langle [\mathcal Z(S)]_\Omega f_1,f_2 \right \rangle .  	\end{split}		\end{equation}
Let $Q \in \mathcal N$. If $Q \subset F$, then by the maximality of $\mathcal S$, $Q \subset S$ for some $S \in \mathcal S$, but this contradicts the definition of $\mathcal N$. Therefore, there exists $x \in (Q_0 \backslash F) \cap Q$ and so by the definition of $F$ in \eqref{e:badcube}, $\avg{f_j}_Q \le \M(\mathbf 1_{Q_0} f)(x) \le \Theta \avg{f_j}_{Q_0}$. Therefore,
	\begin{equation}\label{e:nonest} \sum_{Q \in \mathcal N} \avg{f_1}_Q  \avg{f_2}_Q |Q| \lesssim \avg{f_1}_{Q_0}  \avg{f_2}_{Q_0} \sum_{Q \in \mathcal N}  |Q| \lesssim \avg{f_1}_{Q_0}  \avg{f_2}_{Q_0} |Q_0|,\end{equation}
where the last inequality follows from sparseness of $\mathcal N$. Set $\mathcal R_0 = \{ Q_0 \}$, and we will construct the $\mathsf R$-sparse collection inductively. Suppose we have $\mathcal R_k$. Then set
	\[ \mathcal R_{k+1} = \bigcup_{Q \in \mathcal R_k} \mathcal S(Q).\]
This induction process terminates after finitely many steps, once $\mathcal R_{k+1} = \varnothing$, since $\mathcal Z$ was finite and $Q \not\in\mathcal S(Q)$ for $\Theta>6^d$. Therefore, iterating \eqref{e:befpack} and \eqref{e:nonest}, we have
	\[ \left\langle [\mathcal Z(Q_0)]_\Omega f_1,f_1 \right\rangle \lesssim \sum_{k} \sum_{Q \in \mathcal R_k} \avg{f_1}_Q \avg{f_2}_Q |Q| \lesssim \sum_{k} \sum_{Q \in \mathcal R_k} \avg{f_1}_{Q \cap \Omega} \avg{f_2}_{Q \cap \Omega} |Q \cap \Omega|,\]
where the last inequality relied on $f_1,f_2$ being supported in $\Omega$. 
The $\mathsf R$-sparse collection is now $\mathcal R = \cup_k \{Q \cap \Omega : Q \in \mathcal R_k\}$ and it remains to verify that the sets 
	\[ E_Q \coloneqq (Q  \cap \Omega) \setminus \bigcup_{S \in \mathcal S(Q)} S \]
are disjoint and make up a positive portion of $Q \cap \Omega$. They are disjoint since given any two distinct cubes $Q,P \in \cup_k \mathcal R_k \subset \mathcal D_t$, they are either disjoint (in which case $E_Q \subset Q$ and $E_P \subset P$ are disjoint) or one is contained in the other. Say $Q \subset P$, then $Q \subset S$ for some $S \in \mathcal S(P)$ and therefore $Q \cap E_P = \varnothing$ whence $E_Q$ and $E_P$ are disjoint. To establish the estimate of the Lebesgue measure of $E_Q$, first notice that by construction of $\mathcal S(Q)$ and \eqref{e:Festimate},
	\[ \abs{ \bigcup_{S \in \mathcal S(Q)} S } \le |F(Q)| \le \frac{6^d}{\Theta} |Q| \le \frac{1}{2}|Q \cap \Omega|, \]
for $\Theta$ large enough depending on the constant in \eqref{e:ZQdoubling}. Therefore,
	\[ |E_Q| \ge |Q \cap \Omega| - \abs{ \bigcup_{S \in \mathcal S(Q)} S } \ge \frac{1}{2} |Q \cap \Omega|.\]
\end{proof}

Weighted estimates for general wavelet forms are now an immediate consequence of \eqref{e:sparzglob}, \eqref{e:equival1}, and \eqref{e:sp-comp}. 
\begin{proposition}\label{p:WFw}
 Let $\Omega$ be a uniform domain, $\Lambda$ be a wavelet form as in \eqref{e:triebel2} with $n=2$, $1<p<\infty$, $w$ a weight, and $\sigma=w^{-\frac{1}{p-1}}$ be its $p$-th dual. Then 
\begin{equation}
\label{e:formeraw}
 \left|\Lambda (f_1\mathbf{1}_\Omega,f_2\mathbf{1}_\Omega)\right| \lesssim  { [w]_{\mathrm{A}_p(\Omega)}^{\max\left\{1,\frac{1}{p-1} \right\} } } \left \Vert f_1 \right \Vert_{L^p(\Omega,w)}\left \Vert f_2 \right \Vert_{L^{p'}(\Omega,\sigma)}.
 \end{equation} 
\end{proposition}

\addtocounter{other}{1}
\section{Boundedness of paraproduct forms}\label{sec:pp}
This section introduces appropriate paraproduct forms on domains.
Sobolev space bounds for paraproduct forms are characterized by  suitable conditions on the symbol $\bb$. This characterization is formulated, respectively in the unweighted and weighted cases, in Propositions  \ref{paraproduct-main} and \ref{p:pp-weight}. 

\begin{definition}[Paraproduct forms] \label{def:pp}
Hereafter, \[\bb=\{b_{W,j}: W\in \W, 1\leq j \leq \jmath\}\cup\{b_{S}: S\in\Sc\}\] is a sequence of complex numbers indexed by the same parameters of $\mathcal B(\Omega,k)$ from \eqref{e:triebelref}. The   qualitative assumption that the coefficient sequence $\bb$ decays as in \eqref{e:qualass} is made so that   $\bb$ may be identified with the function
\begin{equation}\label{e:seq-id}
\bb\coloneqq {\W} \bb + {\Sc} \bb\in \CL , 
\quad
{\W} \bb \coloneqq \sum_{\substack{W \in \W\\ 1\leq j \leq \jmath}} |W|b_{W,j} \overline{\chi_{W,j}},\quad {\Sc} \bb \coloneqq \sum_{S\in \Sc} |S|b_{S}\overline{ \varphi_{S}}.
\end{equation}
The \textit{paraproduct} of $\bb$
is then defined by
	\begin{equation}
\label{e:defpp}
 \Pi_{\bb}(f_1,f_2) \coloneqq \sum_{\substack{W \in \W\\ 1\leq j \leq \jmath}}  |W| b_{W,j} \zeta_{W}(f_1)\chi_{W,j}( {f_2}) 
 + \sum_{ S\in \Sc}  |S| b_{S} \zeta_{S}(f_1)\varphi_{S}( {f_2})
\end{equation} 
where $\{\zeta_R \in C\Phi^{k,\Subset}(R):R\in \mathcal M\}$ is a generic fixed  family.  Paraproducts are immediately related to wavelet forms. Namely,
	\begin{equation}\label{e:lambdap}
	\begin{split} &
	\Pi_{\mathfrak{b}}(f_1,f_2)=\Lambda (\mathfrak b,f_2,f_1),\\ 
	&\Lambda  \coloneqq \sum_{W \in \mathcal W} |W|
\Lambda_W, \; \,\Lambda_W\coloneqq    \phi_W \otimes \zeta_W + \sum_{S \in \mathcal{D}(W)} \frac{|S|}{|W|} \varphi_S \otimes \overline{\varphi_S} \otimes \zeta_S, \end{split}
	 \end{equation}
	having set for brevity
	$\phi_W =\sum_{1\leq j \leq \jmath} \chi_{W,j}\otimes \overline{\chi_{W,j}} $. It is immediately seen that
 $\Lambda$ is a fully localized wavelet form.
\end{definition}

\addtocounter{other}{1}
\subsection{Carleson measure norms}\label{ss:pp-norms} Paraproduct forms will be controlled by means of a combination of   Triebel-Lizorkin norms of the type  introduced in \S\ref{ss:tl-norms} and suitable Carleson measure norms. To define the latter, with  reference to Lemma \ref{l:windows}, call $   \mathcal W_{\mathsf{bd}}\coloneqq \W\setminus \mathcal W_0$ the \textit {boundary cubes} of $\Omega$, let $\SH(W_j)$ stand for the boundary windows of a uniform domain and define  the following partial ordering on cubes $W,P \in \SH_1(W_j^*)\subset \mathcal \mathcal W_{\mathsf{bd}}$. Say $P \le W$ if for each admissible chain $[P,W_j^*]$, there exists an admissible chain $[W,W_j^*]$ such that
	$ [W,W_j^*] \subset [P,W_j^*].$
Furthermore, define the ordered shadow $\SH_0(W)$ and its realization $\Sh_0(W)$ by
	\[ \SH_0(W) = \{ P \in \SH_1(W_j^*): P \le W \}, \qquad \Sh_0(W) = \bigcup_{P \in \SH_0(W)} \mathsf{w}P.\]
 Let $n \in \mathbb N,$ $1<p<\infty$. Say $  f \in L^1_{\mathrm{loc}}(\Omega)$ is an $(n,p)$-\textit{Carleson measure} for $\Omega$ if there exists $K>0$ such that for all $j \ge 1$,
	\begin{equation} \label{eq:carleson-measure}
	\begin{aligned} 
		\forall W \in \SH_1(W_j^*), 
			\ \sum_{P \le W} \|   f \|_{L^p(\Sh_0(P))}^{p p'} \ell(P)^{(n p-d)(p'-1)} 
			&\le K^{p'} \|  f\|_{L^p(\Sh_0(W))}^p, & n p \le d; \\
		\|  f \|_{L^p(\Sh(W_j))} 
			&\le K, & n p > d. 
	\end{aligned}
	\end{equation}	
Let $\| f\|_{\mathrm{Carl}^{n,p}(\Omega)}$ denote the smallest $K$ such that \eqref{eq:carleson-measure} holds. 
The next lemma contains some elementary properties of the $(n,p)$-Carleson measure norms. The proofs are easy and therefore left to the reader. 
\begin{lemma}\label{l:carl-prop}
Let $1 \le m \le n$, $1 < q \le p < \infty$, $f \in \mathrm{Carl}^{n,p}(\Omega)$, and $g \in L^1_{\mathrm{loc}}(\Omega)$.
\begin{itemize}
	\item[(C1)] $\|f\|_{\mathrm{Carl}^{m,q}(\Omega)} \lesssim \|f\|_{\mathrm{Carl}^{n,p}(\Omega)}$.
	{  \item[(C2)] $\|f\|_{F^{0,-n}_{p,2}(\mathcal W_{\mathrm {bd}})} \lesssim \|f\|_{\mathrm{Carl}^{n,p}(\Omega)}$.}
	\item[(C3)] If $\|g\|_{L^p(sP)} \lesssim \|f\|_{L^p(rP)}$ for some  $ 1\leq r,s \leq \mathsf{w}^2$ and all $P \in \mathcal W_{\mathsf {bd}}$, then
		\[ \|g\|_{\mathrm{Carl}^{n,p}(\Omega)}\lesssim \|f\|_{\mathrm{Carl}^{n,p}(\Omega)}.\]
\end{itemize}
\end{lemma}

\addtocounter{other}{1}
\subsection{Testing conditions for paraproducts}\label{ss:pp-test}
The main results of this section are  the following testing conditions for boundedness of the paraproduct forms. The equivalences in \eqref{e:pplorenz} and \eqref{e:ppstrongsc}--\eqref{e:ppadj}  should be interpreted in the sense that the testing condition based on the right hand side  is necessary and sufficient.
\begin{proposition} 
\label{paraproduct-main} 
If $1\leq n\leq k$, then
\begin{align}
\label{e:pplorenz}
  & \left\|\Pi_\bb : W^{n,(p,1)}(\Omega) \times L^{p'}(\Omega) \right\|  \sim   \left\| \bb  \right\|_{  F^{0,-n}_{p,2}(\mathcal M)} + \left\| \bb  \right\|_{\mathrm{Carl}^{n,p}(\Omega)}, \quad 1<p\leq \frac dn;\\ &
    \label{e:ppstrong}
   \left\|\Pi_{\bb} : W^{n,p}(\Omega) \times L^{p'}(\Omega) \right\|  \lesssim  \left\| \bb  \right\|_{  F^{0,-n}_{q,2}(\mathcal M)} + \left\| \bb  \right\|_{\mathrm{Carl}^{n,q} (\Omega)}, \quad  1<p<q\leq \frac dn;
   \\
     \label{e:ppstrongsc} &
   \left\|\Pi_{\bb} : W^{n,p}(\Omega) \times L^{p'}(\Omega) \right\|  \sim   \left\| \bb  \right\|_{F^{0,-n}_{p,2}(\mathcal M)}+ \left\| \bb  \right\|_{\mathrm{Carl}^{n,p}(\Omega)}, \quad p>\frac dn.
 \end{align}
Furthermore, if ${\W} \bb =0 $ and $0 \le n \le k$, then
\begin{equation}
\label{e:ppadj}
   \left\|\Pi_{\bb} :   L^{p}(\Omega) \times W^{n,p'}(\Omega)  \right\|   \sim \|\bb\|_{\dot F^{-n,0}_{1,2}(\mathcal M)}, \qquad 1<p<\infty.
\end{equation}
\end{proposition}
By virtue of \eqref{eq:ibp-psi}, \eqref{e:ppadj} follows \textit {a fortiori} from Proposition \ref{p:sparse2} and an integration by parts. 
The main line of proof of Proposition \ref{paraproduct-main} is devoted to \eqref{e:pplorenz}--\eqref{e:ppstrongsc} and occupies \S\ref{ss4:mainline}. At the end of that subsection, we also record a  sparse domination estimate which is obtained along the way and may be used to obtain weighted norm versions of \eqref{e:pplorenz}--\eqref{e:ppstrongsc}.

\begin{remark} \label{rem:weakests}  
Proposition \ref{p:sparse2} may also be used to immediately deduce  sparse bounds for \eqref{e:defpp}  under  stronger  conditions on the symbol $\bb$,  in general not necessary, for Sobolev space bounds. More precisely, 
\begin{align}
\label{e:sparsepp}    \left|\Pi_{\bb}  (f_1,f_2) \right| &\lesssim \min \left\{ \left\|\mathrm{M}( f_1,f_2,\bb)\right\|_{1}, 
\| \bb\|_{\mathrm{bmo}(\Omega)} \left\|\mathrm{M}( f_1,f_2)\right\|_{1} \right\}.
\end{align}
As a consequence of   \eqref{e:sparsepp} and Sobolev embeddings, when $p \ne \frac{d}{n}$,
\begin{equation}
\label{e:sparseppLp}
 \left|\Pi_{\bb}  (f_1,f_2) \right| \lesssim \|\bb\|_{L^{p^\star}(\Omega)} \|f_1\|_{W^{n,p}(\Omega)} \|f_2\|_{p'},\quad p^\star\coloneqq \textstyle \max\{p,\frac{d}{n}\}. 
\end{equation}
If  $\Omega$ is a bounded domain,  
$ \bb\in L^p(\Omega)$ is a necessary condition for 
 boundedness of $\Pi^{\gamma}_{\bb} $ on $W^{k,p}(\Omega)\times L^{p'}(\Omega)$. Thus, estimate \eqref{e:sparseppLp} is a characterization when   $np>d$ and $\Omega $ is a bounded domain. 
 In contrast, the second part of \eqref{e:sparsepp}  is stronger but relies on the finiteness of $\|  \bb\|_{\mathrm{bmo}(\Omega)}$, which is not necessary when $n>0.$
\end{remark}

\addtocounter{other}{1}
\subsection{Main line of proof of Proposition \ref{paraproduct-main}} \label{ss4:mainline}
The components $\Lambda_W $ from \eqref{e:lambdap} will be dealt with in two separate steps by adding and subtracting $\PM_W^{n} f_3$ to the third (noncancellative) entry. The proof of the next lemma involves several steps and is postponed to \S\ref{ss:sparsepp2}.
\begin{lemma} Let $W\in \mathcal W$ and  $1<p<\infty$. Then
\label{l:sparsepp2} 
\begin{equation}
\label{e:necsufpweak}  \sup_{  \substack{ \l  f_2\r_{p',\mathsf{w}W}  = 1 \\ \l\nabla^n f_3\r_{(p,1),\mathsf{w}W}  = 1}}
 \left|\Lambda_{W}(f_1,f_2, f_3-\PM^{n}_W f_3) \right|
\sim
\|f_1\|_{ F^{0,-n}_{p,2}(\mathcal D_+(W))}.
   \end{equation}
Furthermore, 
 \begin{equation}
\label{e:sufp}  \sup_{  \substack{  \l  f_2\r_{q',\mathsf{w}W}  = 1 \\ \l\nabla^n f_3\r_{q,\mathsf{w}W}  = 1 }} 
\left\vert \Lambda_{W}(f_1,f_2, f_3-\PM^{n}_W f_3) \right\vert
\begin{cases} \lesssim 
{ \|f_1\|_{ F^{0,-n}_{p,2}(\mathcal D_+(W))},} & 1<q\leq \frac dn, p>q,
\\ \lesssim
{  \ell(W)^n\langle      f_1 \rangle_{q,  \mathsf{w}W} } & q>  \frac dn.
      \end{cases}
    \end{equation}
\end{lemma} 
Fix $W\in \mathcal W$ and $S\in \mathcal D_+(W).$ Applying \eqref{e:ibp3} for $n=m$, $Q=W$ and $x\in\mathrm{supp}\, \zeta_S \subset \mathsf{w}W$ yields
\[\zeta_S(\PM^{n}_{W} f) =\int \zeta_S(x)  \upsilon_{W,x}^\gamma (f) \, \mathrm{d} x \eqqcolon \upsilon_{W,S} (f), \qquad \upsilon_{W,S} \in C\Phi^{0,\Subset }(W). 
\]
Therefore, referring to \eqref{e:PsiZ} for the definition of $\Psi_W$,
\begin{equation}
\label{e:upsWR}
\begin{split} &\quad
\left|\Lambda_{W}( f_1,f_2, \PM_W^{n} f_3) \right| \lesssim \abs{ \phi_W(f_1,f_2) \upsilon_{W,W}(f_3) }+ \sum_{S \in \mathcal D(W)} \frac{|S|}{|W|} \abs{ \varphi_S(f_1)\varphi_S(f_2) \upsilon_{W,S}(f_3) } \\
 &\lesssim  \left[ \left|\phi_W(f_1,f_2)\right| +\Psi_W(f_1,f_2)\right]  \l f_3 \r_{\mathsf{w}W} .
\end{split}
\end{equation}
The lemma below, which will only be  used for central cubes $W\in \mathcal W_0$, follows immediately  by \eqref{e:PsiZ}, \eqref{e:lsfo} and \eqref{e:upsWR}. 
\begin{lemma}

\label{l:central} For $1<p<\infty$, $\displaystyle \left|\Lambda_{W}(f_1,f_2,\PM^{n}_W  f_3)\right|\lesssim  \left\langle f_1\right\rangle_{p,\mathsf{w} W}\left\langle f_2\right\rangle_{p',\mathsf{w} W} \left\langle f_3\right\rangle_{\mathsf{w}W}  $.
\end{lemma}

The next definition will lead us to  a better estimate for \eqref{e:upsWR} when $\ell(W)\ll 1$, namely for boundary cubes $W\in \mathcal W_{\mathsf{bd}}\coloneqq \W\setminus \mathcal W_0$ with reference to Lemma \ref{l:windows}.

\begin{definition}[Hardy paraproduct form]
Let $n \ge 1$, $\SH(W_j)$ be the boundary windows of a uniform domain $\Omega$. 
For $\Psi_W$ defined by \eqref{e:PsiZ} and $\phi_W \in \Phi^{0,\Subset}_{2}(W)$, define the Hardy paraproduct of order $n$ by
	\[
	\begin{split}&\I^{n}(f_1,f_2,f_3) \coloneqq \sum_{j \ge 1} \sum_{{W} \in \SH({W}_j)} \left[ \phi_W(f_1,f_2) +|W| \Psi_W(f_1,f_2) \right]\I^{n}_W(f_3), \\ &
	 \I^{n}_W(f)\coloneqq \sum_{P \in [{W},{W}_j^*]} \ell(P)^{n} \l f \r_{\mathsf{w}P} .
	\end{split}
	\]
Hardy paraproduct forms are related to \eqref{e:upsWR} via the next lemma.
\end{definition}

\begin{lemma} \label{l:decohardy}
There holds
	\[
	\sum_{W\in \mathcal W_{\mathsf{bd}}} {|W|}
		\left|\Lambda_{W}( f_1,f_2, \PM^{n}_W f_3) \right|
		\lesssim    \sum_{t=0}^{n}\I^{{n}}(f_1,f_2, \nabla^t  f_3)
	\]
where $\I^{n}$ is a Hardy paraproduct of order $n$ and the implied constant is absolute.
\end{lemma}

\begin{proof}  
Fix $W \in\mathcal W_\mathsf{bd}$. There exists  $j$ such that $W \in \SH(W_j)$. Referring to \eqref{e:upsWR}, we aim at decomposing the corresponding $\upsilon_{W,S}(f_3)$ coefficient for $S \in \mathcal D_+(W)$. By property (S3) in Lemma \ref{l:windows},  find $\varpi_j \in \C^\infty_0(\Sh(W_j^*))$ with $\varpi_j=1$ on $\Sh(W_j)$ and $\varpi_j=0$ on $W_j^*$.   By Lemma \ref{lemma:tele}, for each $P \in [W,W_j^*]$, there exists $\upsilon_{W,S,P}^{-n} \in C \Phi^{0,\Subset}(P)$ such that
	\[ \begin{split}
	 \upsilon_{W,S}\left(f_3\right)& =\upsilon_{W,S}\left(\varpi_jf_3\right) = \upsilon_{W,S}\left(\varpi_j f_3 - \PM^{n}_{W_j^*}(\varpi_j f_3)\right) \\ &= \sum_{P \in [W,W_j^*]} \ell(P)^{n} \upsilon_{W,S,P}^{-n}\left(\nabla^{n}(\varpi_j f_3)\right).\end{split}\]
By (S1) and (S3), $\|\nabla^t \varpi_j\|_\infty \lesssim 1$,  and  the claim follows  via  Leibniz rule.
\end{proof} 
An observation dating back to the work of Arcozzi, Rochberg, and Sawyer on Sobolev spaces of analytic functions \cite{arcozzi2002carleson} is that
Lebesgue space bounds for the Hardy paraproduct are characterized by the norm $\mathrm{Carl}^{n,p}(\Omega)$. 
\begin{lemma}\label{l:hardy-carl}
Let $n \in \mathbb N$, $1<p<\infty$ and $\I^{n}$ be a  Hardy paraproduct form of order $n$. Then 
	\[ \left| \I^{n}(f_1,f_2,f_3) \right| \lesssim_{p,n}  \|f_1\|_{\mathrm{Carl}^{n,p}(\Omega)} \|f_2\|_{L^{p'}(\Omega)} \|f_3\|_{L^p(\Omega)}.\]
\end{lemma}
The proof of Lemma \ref{l:hardy-carl} occupies  \S\ref{ss:pfhardy}. Notably, $f_1\in \mathrm{Carl}^{n,p}(\Omega)$ is also necessary to ensure that the corresponding order $n$ Hardy paraproduct form $\I^{n}(f_1,\cdot,\cdot)$ extends boundedly to $L^p(\Omega) \times L^{p'}(\Omega)$ (see \cite{arcozzi2002carleson}). A stronger necessity claim is proved in \S\ref{ss:5nec}, showing that the decomposition in Lemma \ref{l:decohardy} is also sharp.

We are now ready to complete the proof of Proposition \ref{paraproduct-main}.
 We prove \eqref{e:pplorenz}, the proofs of \eqref{e:ppstrong} and \eqref{e:ppstrongsc} simply use \eqref{e:sufp} instead of \eqref{e:necsufpweak}. First, applying Lemmata \ref{l:sparsepp2} and \ref{l:central} and H\"older's inequality
	\[
	\begin{split}  
		&\quad \sum_{W\in \mathcal W_0}\abs{  \Lambda_W(\bb,f_2,f_1)} 
			 \leq \sum_{W\in \mathcal W_0}{| W|}\abs{   \Lambda_W(\bb,f_2,f_1-\PM^{n}_Wf_1)}  +{| W|}
				\abs{  \Lambda_W(\bb ,f_2,\PM^{n}_Wf_1)}  \\
			&\lesssim  {\left\| \bb  \right\|_{  F^{0,-n}_{p,2}(\mathcal M)}} 
				\times \left[ \left(\sum_{W\in \mathcal W_0}|W|  \langle\nabla^{n}f_1\rangle_{(p,1),\mathsf{w}W}^p \right)^{\frac1p} +\left(\sum_{W\in \mathcal W_0}|W|  \langle f_1\rangle_{p,\mathsf{w}W}^p \right)^{\frac1p} \right]
				\|f_2\|_{L^{p'}(\Omega)} \\ 
			&\lesssim  {\left\| \bb  \right\|_{  F^{0,-n}_{p,2}(\mathcal M)}}  \|f_1\|_{W^{n,(p,1)}(\Omega)}\|f_2\|_{L^{p'}(\Omega)}.
	\end{split}
	\]  
A similar application of  Lemma  \ref{l:sparsepp2} and H\"older leads to  
	\[
	\begin{split}  
		\sum_{W\in \mathcal W_{\mathsf{bd}}} {| W|}\abs{  \Lambda_W(\bb,f_2,f_1-\PM^{n}_W f_1)} 
			&\lesssim {\left\| \bb  \right\|_{  F^{0,-n}_{p,2}(\mathcal M)}} 				\|\nabla^n f_1\|_{L^{p,1}(\Omega)}\|f_2\|_{L^{p'}(\Omega)}.
	\end{split}
	\]
Finally, a combination of Lemmata  \ref{l:decohardy} and \ref{l:hardy-carl} entails
	\[
	\begin{split}  
		\sum_{W\in \mathcal W_{\mathsf{bd}}} {| W|}\abs{  \Lambda_W(\bb,f_2,\PM^{n}_W f_1)}  
		&\lesssim   \|  \bb\|_{\mathrm{Carl}^{n,p}(\Omega)} \|f_1\|_{W^{n,p}(\Omega)}\|f_2\|_{L^{p'}(\Omega)},
	\end{split}
	\] 
and the proof is complete.

\begin{remark}\label{rem:sparse} Proposition \ref{paraproduct-main} focuses on strong and restricted strong type estimates. The proofs of Lemmata \ref{l:sparsepp2}, \ref{l:central}, \ref{l:decohardy} actually yield a domination by sparse and Hardy-sparse forms which will be used to derive weighted versions of Proposition \ref{paraproduct-main} in \S\ref{ss:wPP} below. Namely, 
Lemma \ref{l:central}  
may be joined with the first inequalities of  \eqref{e:pf45b} and \eqref{e:sparselocal} below to deduce the existence of a sparse collection $\mathcal Z $, $\W\subset \mathcal Z\subset \mathcal M$, with the property that for any $1<r<\infty$,  
\begin{equation}
\begin{split}
\label{e:Sd} &\quad  \left|\Pi_{\bb}(f_1,f_2)\right| \lesssim { \max\{r,r'\}^{\frac12}\left\| \bb  \right\|_{  F^{0,-n}_{r,2}(\mathcal M)}}  \sum_{Z\in \mathcal Z} |Z| \left[ \langle  f_1 \rangle_{\mathsf{w}Z}+  \langle \nabla^n f_1 \rangle_{\mathsf{w}Z} \right]\langle f_2 \rangle_{r',\mathsf{w}Z}
\\ &+ \sum_{W\in \W_{\mathsf{bd}}} \left[ \sum_{ t=0}^n \mathfrak {I}^n_W(\nabla^t f_1) \right] \left[ \sum_{  Z\in \mathcal Z\cap \mathcal D_+(W)} |Z| \langle f_2 \rangle_{\mathsf{w}Z} \langle  \bb \rangle_{\mathsf{w}Z}\right].
\end{split}
\end{equation}
The implicit constant is meant to be absolute and, in particular, independent of $r$. { We remind the reader here that all results of this paper hold with $\Omega =\mathbb R^d$ by taking $\mathcal W = \varnothing$ and $\mathcal S$ to be the full dyadic grid $\mathcal D$. In this case, the sparse bound \eqref{e:Sd} for the paraproduct forms holds without the summation over $W \in \mathcal W_{\mathsf{bd}}$, for any fixed cube $Q_0$ large on which $\bb,f_1,f_2$ are supported, and $\mathcal Z = \mathcal Z(Q_0)$. This of course implies Proposition \ref{paraproduct-main} which is new even on $\mathbb R^d$ for the full range $1<p<\infty$ and all $1 \le n \le k$ (the highest order $n=k$ and supercritical range $np \ge d$ is known in \cite{wang97} but this case is immediate from Remark \ref{rem:weakests} above). Furthermore, it implies weighted estimates (see Corollary \ref{cor:gl} below) which extend some of the results from \cite{diplinio22wrt} to the inhomogeneous case.}
\end{remark}

\addtocounter{other}{1} 
\subsection{Proof of Lemma \ref{l:sparsepp2}}\label{ss:sparsepp2}
 First, we prove the upper bounds in both \eqref{e:necsufpweak},   \eqref{e:sufp} at the same time. Due to the definition of $\PM^{n}_W$, cf.\ \eqref{eq:sob-func-rep}, $ \PM^{n}_W f_1 =  \PM^{n}_W [\mathbf{1}_{\mathsf{w}W}f_1]$; thus in view of the support conditions on the wavelets, we may assume $f_1,f_2,f_3$ are all supported in $\mathsf{w}W$ by possibly replacing $f_j$ with $\mathbf{1}_{\mathsf{w}W}f_j$.   The   $\phi_W$ term in \eqref{e:lambdap} may be easily estimated using the Poincar\'e inequality as \begin{equation}
\label{e:pf45a}   \left| \phi_W(f_1,f_2)  \zeta_W(f_3-\PM^{n}_W f_3) \right| \lesssim \ell(W)^n \langle f_1\rangle_{\mathsf{w}W}   \langle   f_2\rangle_{\mathsf{w}W} \langle \nabla^n  f_3 \rangle_{\mathsf w W}
\end{equation}
which complies with all upper bounds in \eqref{e:necsufpweak},   \eqref{e:sufp}. In the remainder of the proof, it may thus be assumed that $\phi_W=0$ and we need to control the form
\[
\mathrm{V}_Z(f_1,f_2,f_3) =  \frac{1}{|Z|}\sum_{S\in \mathcal D(Z)} |S|\abs{\varphi_S(f_1)}\abs{\varphi_S(f_2)}
 \abs{ \zeta_S(f_3-\PM^{n}_Z f_3)} \]
for $Z=W.$
Now, suppose that $Z\in   \mathcal D_+(W)$,   $P\in \mathcal D_+(Z)$. Our proof is based on the inequality
\begin{equation}\label{eq:tele3} 
\abs{\zeta_{P}(f -\PM^{n}_{Z}f)}+ \sup_{S\in \mathcal D_+(P)}\abs{\zeta_{S}(\PM^{n}_{P}f -\PM^{n}_{Z}f) } \lesssim   \ell(Z)^n   \inf_{P} \mathrm{M} (\mathbf{1}_{\mathsf w Z}\nabla^n f)
\end{equation}which may be obtained by \eqref{e:neighbor} and the telescoping argument that led to \eqref{eq:tele}.
For $Z\in \mathcal D_+(W)$, denote by $\mathcal Z(Z)$  the maximal elements  $P\in \mathcal D(Z)$ with the property that \[
\mathsf{w}P\subset\left\{\mathrm{M} \left(\mathbf{1}_{\mathsf{w} Z}\nabla^n f_3\right)>C  \langle \nabla^n f_3\rangle_{\mathsf{w}Z}\right\}.\]   The maximal theorem tells us that  $\sum_{P\in\mathcal Z(Z) } |P|\leq \frac{|Z|}{2}$ if $C$ is large enough. Calling $\mathcal N(Z)\coloneqq \mathcal D(Z)\setminus\bigcup \{\mathcal D(P):P\in \mathcal Z(Z)\}$, we obtain  the estimate
\begin{equation}\label{e:sparsesublime}
\begin{split}  &\quad 
\mathrm{V}_{Z}(f_1,f_2, f_3)- \sum_{P\in \mathcal Z(Z)}\frac{|P|}{|Z|}\mathrm{V}_{P}(f_1,f_2, f_3) \\ & \leq \sum_{S\in \mathcal N(Z)}\frac{|S|}{|Z|} \abs{\varphi_S(f_1)}\abs{\varphi_S(f_2)} \abs{ \zeta_S(f_3-\PM^n_Z f_3)} 
\\ &\quad  + \sum_{\substack{P\in \mathcal Z(Z)\\ S\in \mathcal D(P)}}  \frac{|S|}{|Z|} \abs{\varphi_S(f_1)}\abs{\varphi_S(f_2)} \abs{ \zeta_S(\PM^n_P f_3-\PM^n_Z f_3)} 
 \\  & \leq \Psi_Z(f_1,f_2)
 \left(\sup_{S\in \mathcal N(W)}\abs{\zeta_S(f_3-\PM^n_Z f_3)}+ \sup_{P\in \mathcal Z(W)}\sup_{S\in \mathcal D(P)}\abs{\zeta_S(\PM^n_Pf_3-\PM^n_Z f_3)}\right) \\ & \lesssim
 \Psi_Z(f_1,f_2) \ell(Z)^n\langle \nabla^n f_3\rangle_{\mathsf{w}Z}
 \lesssim   \left[\ell(Z)^n\langle    \mathrm{S}_Z f_1 \rangle_{p,  Z}\right]  \langle f_2\rangle_{p',\mathsf w Z} \langle \nabla^n f_3\rangle_{\mathsf{w}Z} 
\end{split}
\end{equation}
where the passage to the last line follows from \eqref{eq:tele3} and the stopping condition, while the final inequality relies on  \eqref{e:PsiZ}-\eqref{e:lsfo}. The dependence on $p$ of the constant, which has been kept track of in Remark \ref{rem:sparse}, comes from \eqref{e:lsfo} and is kept implicit in the rest of the proof.   Beginning with $Z=W$,  use \eqref{e:sparsesublime}    iteratively to find  a sparse collection $\mathcal Z\subset \mathcal D_+(W)$  such that
\begin{equation}
\label{e:sparserefer} {\mathrm{V}_{W}(f_1,f_2, f_3)
}   \lesssim  \frac{1}{|W|} \sum_{Z\in \mathcal Z}|Z| \left[\ell(Z)^n\langle    \mathrm{S}_Z f_1 \rangle_{p,  Z}\right]  \langle f_2\rangle_{p',\mathsf w Z} \langle \nabla^n f_3\rangle_{\mathsf{w}Z}.
\end{equation}
Let us first deal with the  $1 < p \leq \frac dn$ range. In this case the definition \eqref{e:Fdot} and \eqref{e:sparserefer} lead to
\begin{equation}
\label{e:pf45b}\begin{split}
\abs{\mathrm{V}_{W}(f_1,f_2, f_3)} & \lesssim { \|f_1\|_{ \dot F^{0,-n}_{p,2}(\mathcal D(W))}}\sum_{Z\in \mathcal Z}\frac{|Z|}{|W|} \langle    f_2\rangle_{p',\mathsf w Z} \langle \nabla^n f_3\rangle_{\mathsf{w}Z}    
\\ &\lesssim    { \|f_1\|_{ \dot F^{0,-n}_{p,2}(\mathcal D(W))}} \begin{cases} C_p \langle    \nabla^n f_3\rangle_{(p,1),\mathsf w W}    \langle    f_2\rangle_{p',\mathsf w W};  
\\ 
C_{p,q}
\langle    \nabla^n f_3\rangle_{q,\mathsf w W}    \langle    f_2\rangle_{q',\mathsf w W}, 
& 1<q<p. 
 \end{cases}\end{split} 
\end{equation}
Refer for instance to \cite[Appendix B]{conde-alonso17} for the first estimate. The second estimate is well known, and we omit the details. We have thus proved the upper bound \eqref{e:necsufpweak} as well as the case $1<q\leq \frac dn$ of \eqref{e:sufp}.  If $p=q>\frac{d}{n}$ in \eqref{e:sufp}, first notice with the help of \eqref{e:lsfo} that  for each $Z \in \mathcal D_+(W)$ and $\ep \coloneqq n - \frac{d}{p}$,

\[
\ell(Z)^n\langle    \mathrm{S}_Z f_1 \rangle_{p,  Z} \leq \ell(Z)^\ep \ell(W)^{\frac dp}\langle    \mathrm{S}_W f_1 \rangle_{p,  W} \lesssim   \ell(Z)^\ep \ell(W)^{\frac dp}\langle      f_1 \rangle_{p,  \mathsf{w}W}.
\]
Then
\eqref{e:sparserefer} yields the first inequality in
\begin{equation}
\label{e:recycle}
\begin{split}
\abs{\mathrm{V}_{W}(f_1,f_2, f_3)
}   &\lesssim  \left[ \ell(W)^n\langle f_1\rangle_{p,\mathsf w W} \right]\sum_{j\geq 0}2^{-j\ep}\sum_{\substack{ Z\in \mathcal S\\ \ell(Z)=2^{-j} \ell(W)}}\frac{|Z| }{|W|}   \langle    f_2\rangle_{p',\mathsf w Z} \langle \nabla^n f_3\rangle_{\mathsf{w}Z}
\\ &
\lesssim_ p  \ell(W)^n\langle      f_1 \rangle_{p,  \mathsf{w}W}   \langle    f_2\rangle_{p',\mathsf w W}   \langle \nabla^n f_3\rangle_{p,\mathsf w W},  
\end{split}\end{equation}
while the concluding step is obtained by H\"older's inequality, finite overlap, and summing over $j$.
This concludes the proof of the case $q>\frac dn$ of  \eqref{e:sufp}.

We turn to the lower bound in \eqref{e:necsufpweak}. Focus on the case $np\leq d$, as the argument for the case $np>d$ is strictly simpler. Let $K$ be the supremum on the left  hand side of \eqref{e:necsufpweak}.  Fix $Z\in \mathcal D(W)$ and $f_1\in \CL$. It is not difficult to construct a function $\upsilon_Z$ satisfying
\begin{equation}
\label{e:propups}
\mathbf{1}_{\mathsf{w}Z} \leq \upsilon_Z \leq \mathbf{1}_{\mathsf{w}^2Z}, \qquad \|\nabla^j \upsilon_Z \|_\infty \lesssim_j \ell(Z)^{-j}, \qquad \PM^{n}_W \upsilon_Z =0.
\end{equation} 
The first property ensures  the uniform bound $|\zeta_S(\upsilon_Z)| \gtrsim 1$. Therefore,  we may find $f_2\in \CL$ with $\langle f_2 \rangle_{p',\mathsf{w}Z}=1$ such that $\varphi_S(f_2)=0$ unless $S\in \mathcal D(Z)$ and  ensuring the first estimate in the chain of inequalities
\[
\begin{split} 
 |Z| \langle \mathrm{S}_Z f_1\rangle_{p,Z} & \lesssim  |W|\Lambda_W(f_1, f_2, \upsilon_Z) \leq K |W|   \langle \nabla^n\upsilon_{Z}\rangle_{(p,1),\mathsf{w}W}\langle f_2\rangle_{p',\mathsf{w}W} 
 \\ & = K |Z|   \langle \nabla^n\upsilon_{Z}\rangle_{(p,1),\mathsf{w}Z}\langle f_2\rangle_{p',\mathsf{w}Z} \lesssim K |Z|\ell(Z)^{-n}.
 \end{split}
\] 
The lower bound in \eqref{e:necsufpweak} follows by rearranging and taking supremum over $Z\in \mathcal D(W)$.

\addtocounter{other}{1} 
\subsection{Proof of  Lemma \ref{l:hardy-carl}}  \label{ss:pfhardy} Applying Proposition \ref{p:sparse} and Remark \ref{r:Lp},\begin{equation}
\label{e:sparselocal} \abs{ \phi_W(f_1,f_2) } + |W|\abs{\Psi_W(f_1,f_2)}  
\lesssim   \langle f_1\rangle_{p,\mathsf{w}W}  \langle f_2\rangle_{p',\mathsf{w}W}\qquad \forall W\in \mathcal W.
\end{equation}
Begin by dealing with the Sobolev embedding range ${n}p > d$. Notice that for each $j \ge 1$ and each ${W} \in \SH({W}_j)$,
	\[ \begin{split}\sum_{ P\in [{W},{W}_j^*]} \ell(P)^{n} \l f_3 \r_{\mathsf{w}P} 
	&\lesssim \left(\sum_{{P} \in [{W},{W}_j^*]  }|P| \langle f_3\rangle_{p,\mathsf{w}P}^p\right)^{\frac 1p} \left(\sum_{{P} \in [{W},{W}_j^*]} \ell({P})^{\left (n-\frac dp \right)p'}\right)^{
	\frac{1}{p'}}
	\\
	& \lesssim \|f_3\|_{L^p(\Sh_1({W}_j^*))}\end{split}\]
since $n-\frac dp > 0$ in this range and $\ell({W}_j^*) \sim 1$ by Lemma \ref{l:windows} (S1). Then, using \eqref{e:sparselocal} and finite overlap of the Whitney cubes	
	\[ 
	\begin{aligned} 
		\I^{n}(f_1,f_2,f_3) 
		&\lesssim \sum_{j \ge 1} \left( \sum_{{W} \in \SH({W}_j)} |W| \langle f_1\rangle_{p,\mathsf{w}W} \langle f_2\rangle_{p',\mathsf{w}W}    \right)  \|f_3\|_{L^p(\Sh_1(W_j^*))} \\ 
	&\lesssim  \|f_1\|_{\mathrm{Carl}^{n,p}(\Omega)} \left( \sum_{j \ge 1} \|f_2\|_{L^{p'}(\Sh({W}_j))}^{p'} \right)^{\frac 1{p'}} \left( \sum_{j \ge 1} \|f_3\|_{L^p(\Sh_1({W}_j^*))}^p \right)^{\frac1p} .\end{aligned}\]
The proof of this case is then finished by taking advantage of the  finite overlap of $\{\Sh({W}_j)\}$ and $\{\Sh_1({W}_j^*)\}$ from Lemma \ref{l:windows} (S4).
The subcritical case, $n p \le d$, is dealt with via the following two-weight Carleson embedding lemma \cite{arcozzi2002carleson}*{Theorem 3}, whose short proof is postponed to the end of the subsection.
\begin{lemma}\label{lemma:carl-emb}
Let $\rho,\mu,F_1,F_2:\SH(W_j^*) \to [0,\infty)$ and $1<p<\infty$. Set $\|\mu\|_{\mathrm{Carl}^p(\rho)}$ to be the smallest $C$ such that  
	\begin{equation}\label{e:carleson-measure-rho} \sum_{P \le W} \left[\rho(P)\sum_{Q \le P} \mu(Q)^p \right]^{p'} \le C^{p'} \sum_{P \le W} \mu(P)^p \qquad \forall W \in \SH(W_j). \end{equation}
Then,
	\begin{equation}\label{eq:tree} \begin{split} &\quad \sum_{W \in \SH(W_j^*)} \sum_{P \in [W,W_j^*]} \abs{ \rho(P) F_1(P) F_2(W) \mu(W) } \\ &\le p \| F_1 \|_{\ell^p(\SH_1(W_j^*))} \|F_2 \|_{\ell^{p'}(\SH(W_j))} \|\mu\|_{\mathrm{Carl}^p(\rho)}.\end{split}\end{equation}
\end{lemma}
To apply the lemma to $\I^{n}(f_1,f_2,f_3)$, first notice that by \eqref{e:sparselocal}, for each $j \ge 1$,
	\[ \abs{ \I^{n}(f_1,f_2,f_3) } \lesssim \sum_{j \ge 1} \sum_{{W} \in \SH({W}_j)}\l f_1 \r_{p,W} \l f_2 \r_{p',W} \I^{n}_W(f_3) |W|.\]
For each $j \ge 1$, the inner sum is bounded by the LHS of \eqref{eq:tree} with
	\[ F_1(P) = \|f_3\|_{L^p(P)}, \ \rho(P)=\ell(P)^{n-d/p}, \ \mu(W) = \| f_1 \|_{L^p(\mathsf{w}W)}, \  F_2(W) = \| f_2\|_{L^{p'}(\mathsf{w}W)}.\] 
Concerning the RHS of \eqref{eq:tree}, some simple computations show
	\[
	\begin{split}& \|F_1\|_{\ell^p(\SH_1(W_j^*))} \lesssim \|f_3\|_{L^p(\Sh_1(W_j^*))}, 
\qquad  \|F_2\|_{\ell^{p'}(\SH(W_j))} \lesssim \|f_2\|_{L^{p'}(\Sh(W_j))}, \\ & \|\mu\|_{\mathrm{Carl}^p(\rho)} = \|f_1\|_{\mathrm{Carl}^{n,p}(\Omega)}, \end{split}
	\]
so that summing in $j$ and again using finite overlap  (S4) completes the proof of Lemma \ref{l:hardy-carl}.

\begin{proof}[Proof of Lemma \ref{lemma:carl-emb}]

For $F,\nu: \SH(W_j^*) \to [0,\infty)$, define the averages and maximal averages
	\[ A_{\nu} F(P) =\frac{\sum_{Q\le P} \nu(Q) F(Q) }{\sum_{Q \le P} \nu(Q)}, \quad M_\nu G(W) = \sup_{P \ge W} A_\nu F(P).\]
By definition, $M_\nu$ is increasing as one moves down $\SH(W_j^*)$, i.e. if $P \le W$ then $M_\nu F(P) \ge M_\nu F(W)$. 
This observation shows that the superlevel sets of $M_\nu F$ are of the form
	\begin{equation}\label{e:suplv-nu} E_t(F) \coloneqq \left\{ P \in \SH_1(W_j^*) : M_\nu F(P) \ge t \right\} = \bigsqcup_i \Sh_0(P_i^t), \qquad P_i^t \in \SH_1(W_j^*).\end{equation}
Therefore $M_\nu$ is  weak type (1,1), and by Marcinkiewicz interpolation, $\|M_\nu\|_{\ell^p(\nu) } \le p'$ for $1 < p \le \infty$.
Now, exchanging the order of summation;  setting $\nu(W) = \mu(W)^p$, $\sigma(P) = \rho(P) \sum_{Q \le P} \nu(Q)$, and $G_2(W) = F_2(W) \mu(W)^{1-p}$; and applying H\"older's inequality,
	\[ \begin{aligned} \sum_{W \in \SH(W_j^*)} \left[ \sum_{P \in [W,W_j^*]} \rho(P) F_1(P) \right] F_2(W) \mu(W) &= \sum_{P \in \SH_1(Q_j^*)} F_1(P) \sigma(P) A_\nu G_2(P) \\
	&\le \|F_1\|_{\ell^p(\SH_1(W_j^*))} \|A_\nu G_2 \sigma \|_{\ell^{p'}(\SH_1(W_j^*))}. \end{aligned}\]
Relying on \eqref{e:suplv-nu}, the Carleson measure condition \eqref{e:carleson-measure-rho} says
	\[ \sigma^{p'}(E_t(G_2)) = \sum_{i} \sum_{P \le P_i^t} \sigma(P)^{p'} \le\|\mu\|_{\mathrm{Carl}^p(\rho)}^{p'} \nu(E_t(G_2)),\]
whence $\|A_\nu G_2 \sigma \|_{\ell^{p'}} \le \|\mu\|_{\mathrm{Carl}^p(\rho)}\|M_\nu G_2\|_{\ell^{p'}(\nu)}$. We conclude by applying the $\ell^{p'}(\nu)$ bound for $M_\nu$ and computing $\left\Vert G_2 \right\Vert_{\ell^{p'}(\nu)} = \left \Vert F_2 \right \Vert_{\ell^{p'}}$.
\end{proof}

\addtocounter{other}{1} 
\subsection{Necessity of Carleson measure condition} \label{ss:5nec}
Lemma \ref{l:hardy-carl} is sharp in the sense that
	\[ \|f_1 \|_{\mathrm{Carl}^{n,p}(\Omega)} \lesssim \sup_{ \|f_3\|_{p} = \|f_2\|_{p'}=1 } \left|\I^{n}(f_1,f_2,f_3)\right|.\]
This is proved by testing on suitable $f_1$ and $f_2$, see \cite{arcozzi2002carleson}. In what follows, we will go further and show that not only the Carleson measure condition, but also the representation of the wavelet forms in \eqref{e:upsWR} as a Hardy paraproduct (Lemma \ref{l:decohardy})
is sharp.
\begin{proposition}\label{prop:nec}
Let $f_3 \in L^1_{\mathrm {loc}}(\Omega)$, $1 < p < \infty$ and $n \in \mathbb N$. Suppose there exists $K>0$ such that for each wavelet form 	\[ \Lambda(f_1,f_2,f_3) = \sum_{W \in \mathcal W}|W|\Lambda_W(f_1,f_2)  \zeta_W(f_3) ,\]
one has the estimate
	\begin{equation}\label{e:Lsob} \abs{ \Lambda(f_1,f_2,f_3) } \le K \|f_3\|_{W^{n,p}(\Omega)} \|f_2\|_{L^{p'}(\Omega)}.\end{equation}
Then, $\|f_1\|_{\mathrm{Carl}^{n,p}(\Omega)} \lesssim K$.  
\end{proposition}
\begin{proof}
Take $\Lambda_W$ as in \eqref{e:lambdaW} with $\phi_W = \sum_{1 \le j \le \jmath} \chi_{W,j} \otimes \overline{\chi_{W,j}}$, $\psi^1_S = \varphi_S$ $\psi^2_S= \overline{\varphi_S}$, and $\zeta_W = \mathbf 1_{W}$.
  Therefore, by \eqref{e:triebel} and the assumption \eqref{e:Lsob},
	\begin{equation}\label{e:Af} \sum_{W \in \W} |W| \l f_1 \r_{p,W}^p \l f_3 \r_W \lesssim \abs{ \Lambda(f_1,f_1^{p-1},f_3) } \lesssim \|f_1\|_{L^p}^{p-1} \|f_3\|_{W^{n,p}}.\end{equation}
If $n p>d$, then \eqref{eq:carleson-measure} follows from \eqref{e:Af} by taking $f_3 \in \mathcal{C}^\infty(\overline{\Omega})$ satisfying $\mathbf 1_{\Sh(W_j)} \le f_3 \le \mathbf 1_{\Sh(W_j^*)}$. Indeed,
	\[ \|f_1\|_{L^p(\SH(W_j))}^p \le \sum_{W \in \SH(W_j^*)}\l f_1 \mathbf 1_{\Sh(W_j)} \r_{p,W} ^{p}  \l f_3 \r_W  |W| \lesssim \|f_1\|_{L^p(\SH(W_j))}^{p-1}.\] 
On the other hand, if $n p \le d$, we will need a more involved argument. 
Fix $W \in \mathcal W$ and define
	\[ g = \sum_{P \le W} \mathbf{1}_{P} \l f_1 \r_{p,P}^p, \qquad 
	h(x) = I_{n} g(x) \coloneqq \int_{\mathbb R^d} \frac{g(y)}{|x-y|^{d-n}} \, \mathrm{d}y, \quad x \in \mathbb R^d.\]
For any $\phi \in C_0^\infty(\Sh(W_j^*))$, by H\"older's inequality and \eqref{e:Af},
	\[ \abs{ \ip{ h }{ \phi } }= \abs{\ip{ g }{ I_{n} \phi } } = \sum_{P \le W}|P| \l I_{n} \phi \r_P \l f_1 \r_{p,P}^{p} \lesssim \left\|I_n \phi \right\|_{W^{n,p}(\Omega)} \|f_1\|_{L^p(\Sh_0(W))}^{p-1}.\]
Standard arguments with Riesz transforms show that $\|I_n \phi\|_{W^{n,p}(\Omega)} \lesssim \|\phi\|_{L^p(\Sh(W_j^*))}$ and thus
	\begin{equation}\label{eq:ell-reg} \|h\|_{L^{p'}(\Sh(W_j^*))}^{p'} \lesssim  \|f_1\|_{L^p(\Sh_0(W))}^{p} .\end{equation}
On the other hand, notice that since $p>1$, the condition $n p \le d$ forces $d-n>0$. So, for any $P \le W$ and $x,y \in \Sh_0(P)$, $|x-y|^{d-n} \lesssim \ell(P)^{d-n}$ whence
	\[ |h(x)| \ge \int_{\Sh_0(P)} \frac{g(y)}{|x-y|^{d-n}} \, \mathrm{d}y \gtrsim \ell(P)^{n-d} \int_{\Sh_0(P)} g(y) \, \mathrm{d}y = \ell(P)^{n-d} \|f_1\|_{L^p(\Sh_0(P))}^p,\] 
which together with \eqref{eq:ell-reg}, implies
\[ \sum_{P \le W}\|f_3\|_{L^p(\Sh_0(P))}^{ pp'}  \ell(P)^{(n -d)p'+d} \lesssim \sum_{P \le W} \int_P |h(x)|^{p'} \, dx \lesssim \|f_3\|_{L^p(\Sh_0(W))}^p.\]
Since $(n-d)p' + d = (np-d)(p'-1)$, the above display is precisely the Carleson condition \eqref{eq:carleson-measure}.
\end{proof}

\addtocounter{other}{1} 
\subsection{Weighted estimates for paraproduct forms}\label{ss:wPP}
The main result of this  paragraph, Proposition \ref{p:pp-weight}, stems from Remark \ref{rem:sparse} and is a weighted version of Proposition \ref{paraproduct-main}.
The Hardy paraproducts  are controlled by
 weighted Carleson measure norms, which we now define.
\begin{definition}[Weighted Carleson measure] Given  a weight $w$, and $1<p<\infty$, denote by $\sigma= w^{\frac{-1}{p-1}}$ the $p$-th dual of $w$. Let $n \in  \mathbb N$ be a fixed integer.
Say $f \in \mathrm{Carl}^{n,p}(\Omega,w)$ if there exists $K>0$ such that either
\begin{align}
\label{e:w-carl-sub}  \sum_{P \le W} \|f\|_{L^p(\Sh_0(P),w)}^{pp'} \sigma(P) \ell(P)^{(n-d)p'} &\le K^{p'} \|f\|_{L^p(\Sh_0(P),w)}^{p} \\
	\label{e:w-carl-sup}\mbox{or} \quad \|f\|_{L^p(\Sh(W_j),w)} \left( \sum_{P \in [W,W_j^*]} \sigma(P)\ell(P)^{(n-d)p'}\right)^{1/p'} &\le K
\end{align}
holds uniformly over   $j \ge 1$, and all $W \in \SH_1(W_j^*)$. If this is the case, denote by $\|f\|_{\mathrm{Carl}^{n,p}(\Omega,w)}$ the least such constant $K$.
\end{definition}

If the symbol of the paraproduct form belongs to $\dot F^{0,-n}_{r,2}$,  we can combine Remark \ref{rem:sparse} with Proposition \ref{p:type3w} and achieve a weighted analogue of Proposition \ref{paraproduct-main} in terms of $[w]_{\mathrm{A}_p({\Omega^\circ})}$ and   $[w]_{\mathrm{RH}_s({\Omega^\circ})}$, provided that $r\geq ps'$.  
\begin{proposition}\label{p:pp-weight}
Let $1 \le n \le k$, $1<p,s<\infty$,  $w$ be a weight, $\sigma $ be the $p$-th dual of $w$.  
 Then,
\begin{equation}
\label{e:wprop0} \left\Vert \Pi_{\bb} : W^{n,p}(\Omega,w) \times L^{p'}(\Omega,\sigma) \right\Vert \lesssim \mathcal G \left([w]_{\A_p(\Omega^\circ)},[w]_{\RH_s(\Omega^\circ)} \right) \left( \|\bb\|_{F^{0,-n}_{ps',2}} + \|\bb\|_{\mathrm{Carl}^n_p(\Omega,w)} \right) \end{equation} with implicit constant depending on $p,s$.
Furthermore
\begin{equation}
\label{e:wprop2}\left\| \Pi_{\bb} : W^{n,p}(\Omega,w) \times L^{p'}(\Omega,\sigma) \right\| \lesssim   [w]_{\mathrm{A}_p(\Omega^\circ)}^{2\max\left\{1,\frac{1}{p-1}\right\}} \left( \|\bb\|_{F^{0,-n}_{\infty,2}} + \|\bb\|_{\mathrm{Carl}^n_p(\Omega,w)} \right). 
\end{equation}
\end{proposition}
\rev{As a reminder, $\mathcal G$ denotes a positive increasing polynomial function which may differ at each occurence.}
\begin{proof} In view of Remark \ref{rem:sparse}, the claims will follow from  corresponding estimates for the two forms on the right hand side of \eqref{e:pf45a}. Let us begin with the form on the second line. 
Applying Proposition \ref{p:type3w} for the first step, and the definition of  $\|\bb\|_{\mathrm{Carl}^n_p(\Omega,w)}$ for the second
	\begin{equation}\label{e:Iw}\begin{split} & \quad \sum_{W\in \W_{\mathsf{bd}}} \left[ \sum_{ t=0}^n \mathfrak {I}^n_W(\nabla^t f_1) \right] \left[ \sum_{  Z\in \mathcal Z\cap \mathcal D_+(W)} |Z| \langle f_2 \rangle_{\mathsf{w}Z} \langle  \bb \rangle_{\mathsf{w}Z}\right] \lesssim [w]_{\mathrm{A}_p(\Omega)}^{\max\left\{1,\frac{1}{p-1}\right\}}\\   & \times \sum_{j \ge 1} \sum_{W \in \Sh(W_j)} \left[ \sum_{P \in [W,W_j^*]}  \|f_1\|_{W^{n,p}(P,w)} \sigma(P)^{\frac{1}{p'}}  \ell(P)^{n-d} \right]|W|  \l f_2 \sigma^{\frac{1}{p'}} \r_{p',W} \l \bb w^{\frac{1}{p} } \r_{p,W},
	\\ & \lesssim [w]_{\mathrm{A}_p(\Omega)}^{\max\left\{1,\frac{1}{p-1}\right\}}\|\bb\|_{\mathrm{Carl}^n_p(\Omega,w)} \|f_1\|_{W^{n,p}(\Omega,w)} \|f_2\|_{L^{p'}(\Omega,\sigma)} .
	\end{split}\end{equation}
 Let us detail how to obtain the second step above. If \eqref{e:w-carl-sup} is satisfied, then, for each $j \ge 1$, and $W \in \SH(W_j)$, estimate
	\[ \begin{split} & \quad \sum_{P \in [W,W_j^*]} \|f_1\|_{W^{n,p}(P,w)} \sigma(P)^{\frac{1}{p'}} \ell(P)^{n-d} \\ &\le \|f_1\|_{W^{n,p}(\Sh_1(W_j^*),w)} \left( \sum_{P \in [W,W_j^*]} \sigma(P) \ell(P)^{(n-d)p'} \right)^{\frac{1}{p'}}.\end{split}\]
In this case, the conclusion of \eqref{e:Iw} follows as in the proof of Lemma \ref{l:hardy-carl}, from H\"older's inequality and the finite overlap of the boundary windows. On the other hand, if \eqref{e:w-carl-sub} is satisfied, apply Lemma \ref{lemma:carl-emb} with
	\[ \rho(P) = \sigma(P)^{\frac{1}{p'}} \ell(P)^{n-d} , \ \mu(W) = \left \Vert \bb  \right \Vert _{L^p(\mathsf{w}W,w)} , \ F_1= \|f_1\|_{W^{n,p}(\mathsf {w} P,w)}, \ F_2= \left \Vert f_2  \right \Vert _{L^{p'}(\mathsf{w}W,\sigma)}\]
and again conclude \eqref{e:Iw} with the finite overlap of the boundary windows.
The remaining component of \eqref{e:pf45a} is controlled by applying the second or third estimates from Proposition \ref{p:type3w}.
\end{proof}

The following estimate which may be used in the supercritical range to give a condition which is easier to check than the weighted Carleson measure norm. The relevant weight characteristics are defined in \eqref{e:type3w} and \eqref{e:type1w}.
\begin{lemma}\label{l:Apd} Let $1<p<\infty$, $np >d$, and $1<r< \frac{np}{d}$. 
Then,
	\[ \|f\|_{\mathrm{Carl}^{n,p}(\Omega,w)} \lesssim \left({  [w]_{\mathrm A_p({\Omega^\circ})} [w]_{\mathrm A_{r}(\Omega)}} \right)^{\frac1p} \sup_{j \ge 1} \frac{\| f\|_{L^p(w,\Sh(W_j))}}{w(\Sh_1(W_j^*))^{\frac1p}} .\]
Furthermore,
 	\[ \|f\|_{\mathrm{Carl}^{n,p}(\Omega,w)} \lesssim \mathcal G \left([w]_{\mathrm A_{\min\{p,r\}}(\Omega)},[w]_{\mathrm{RH}_s(\Omega)} \right) \sup_{j \ge 1} \|f\|_{L^{\frac{sp}{s-1}}(\Sh(W_j))}.\]

\end{lemma}
\begin{proof}
Recall that there exists $M>0$ such that $\Sh_1(W_j^*) \subset M W_j^*$ and thus for any $P \in \Sh_1( W_j^*)$, 
	\[ \ell(P)^{dr} \sim \frac{\ell(P)^{dr}}{\ell(MW_j^*)^{dr}} \le \frac{w(P)}{w(MW_j^* \cap \Omega)} [w]_{\mathrm A_r(\Omega) } \le \frac{w(P)}{w(\Sh_1(W_j^*))}[w]_{\mathrm A_r(\Omega)}. \]
This, together with the definition of the $\A_p$ characteristic, shows 
	\[ \begin{split}\sigma(P) &\le [w]_{\mathrm A_p({\Omega^\circ})}^{p'-1} w(P)^{1-p'} \ell(P)^{dp'}\\ & \le \left( [w]_{ \mathrm A_p(\Omega^\circ)} [w]_{ \mathrm A_r(\Omega)} \right)^{p'-1} \ell(P)^{dp'-dr(p'-1)} w(\Sh_1(W_j^*))^{1-p'}. \end{split}\]
Setting $\ep = \frac{np}{d}-r>0$, one can simplify the exponent to obtain 
	\[ \left( \sum_{P \in [W,W_j^*]} \sigma(P) \ell(P)^{(n-d)p'} \right)^{\frac{1}{p'}} \lesssim w( \Sh_1(W_j^*))^{-\frac{1}{p}} \left( \sum_{P \in [W,W_j^*]} \ell(P)^{d\ep(p'-1)} \right)^{\frac{1}{p'}}.\] 
Since $\ep>0$, the geometric sum is uniformly bounded over $W$ by  $\ep^{-1}$.
To prove the second estimate,  apply H\"older to $\|f\|_{L^p(\Sh(W_j),w)}=\|f w^{\frac1p} \|_{L^p(\Sh(W_j))}$ with the exponents  $s'$ and $s$. Finally, 
	\[ \|w\|_{L^s(MW_j \cap \Omega)} \lesssim [w]_{\RH_s(\Omega)} w(MW_j \cap \Omega) \lesssim [w]_{\RH_s(\Omega)} [w]_{\A_{\max\{p,r\}}(\Omega)}w(\Sh_1(W_j))\]
	using the definition of $\RH_s(\Omega)$ and the fact that $1 \sim |M W_j \cap \Omega| \sim |\Sh_1(W_j)|$.
\end{proof}

\section{Representation and Sobolev regularity of Calder\'on-Zygmund forms on domains}\label{sec:rep}

Throughout this section, $\Omega\subset \R^d$ is a fixed uniform domain and $\Lambda$ stands generically for  a continuous bilinear form on $\CL \times \CL$.
 The nonnegative integer $ k$ and the parameter   $0<\delta\le 1$ are  also thought to be fixed throughout. In the overall context of the article, they  respectively embody the   Sobolev space index to which we expect $f\mapsto \Lambda(f,\cdot)$ extends boundedly,  and the additional fractional smoothness index of the off-diagonal kernel of $\Lambda$ and its distributional derivatives.  For simplicity, the parameter $\delta$ is not kept explicit in the notation below.

\begin{definition}[$(k,\ell )$-singular integral form on $\Omega$]    \label{def:si} Let $\ell\in \{0,\ldots,k\}$. The form $\Lambda$ has the normalized \textit{$(k,\ell)$-weak boundedness} if  
	\begin{equation}\label{WBP}\tag{$\mathrm{WBP}_{k,\ell}$}   |{R}|\ell({R})^{\ell}|\Lambda( \chi_{R}, \zeta_{R} )| \le 1, \quad \forall  \chi_{R} \in \Phi^{k+\delta,\Subset}({R}), \, \zeta_{R} \in \Phi^{\ell+\delta,\Subset}({R}), \quad {R} \in\mathcal M.\end{equation}
The normalized \textit{$(k,\ell)$ kernel estimates} hold for $\Lambda$ if 
\[ f,g\in \mathcal U(\Omega),  \, \supp f \cap \supp g =\varnothing \implies \Lambda( f,g) = \int_{\Omega \times \Omega} K(x,y) f(y) g(x) \, \mathrm{d} x  \mathrm{d} y \]
for some $ K\in L^1_{\mathrm{loc}}(\{(x,y)\in\Omega  \times \Omega: x\neq y\}) $  satisfying  uniformly over $0\leq j\leq k-\ell$ the bounds
	\begin{equation}\label{eq:smooth-kernel}
	\begin{split}
	&  \sup_{\substack{P,Q\in \mathcal M\\ \mathfrak{d}(P,Q) \gg \max\{\ell(P),\ell(Q)\}   } } \mathfrak{d}(P,Q)^{d+j+\ell} \left\| \nabla_1^j K \right\|_{L^\infty(P\times Q)}\leq 1, 
	\\
	&
	\begin{array}{l}\quad\displaystyle \sup_{\substack{P\in \mathcal M, S\in \mathcal S\\ \mathfrak{d}(P,S) \gg \ell(S)   } }  \frac{\mathfrak{d}(P,S) ^{d+j+\ell+\delta} }{\ell(S) ^\delta} \\ \times
	\displaystyle \sup_{\substack{x,z\in S \\ y\in P   }}\left(  |\nabla_{1}^j K(z,y)-\nabla_{1}^j K(x,y)|  +  |\nabla_{1}^j K(y,z)-\nabla_{1}^j K(y,x)|\right) \end{array} \leq 1, 
\end{split} \end{equation}
where $\nabla_1^j$ refers to derivatives in the first variable of $K$. 
Say that $\Lambda$ is a \textit{$(k,\ell )$-singular integral form}, or $(k,\ell)$-form in short, if
$
c^{-1} \Lambda
$
has both the $(k,\ell)$ normalized weak-boundedness and the $(k,\ell)$ normalized kernel estimates for some $c>0$. If so, call $\|\Lambda\|_{\mathrm{SI}(k,\ell)}$   the smallest such $c$.  
\end{definition}
\begin{remark}  \label{rem:symm} Let $ \Lambda^{\mathrm t}(f,g)\coloneqq \Lambda(g,f)$. In general, when $0\leq\ell<k$ it is \textit {not} true that  $ \Lambda^{\mathrm t}$ is a $(k,\ell)$-form if $\Lambda$ is,  as the definition is not symmetric under exchange of the arguments. On the other hand,  simple inspection reveals the equality $\|\Lambda\|_{\mathrm{SI}(k,k)}= \|\Lambda^{\mathrm t}\|_{\mathrm{SI}(k,k)}$. This will be exploited below.
\end{remark}
\begin{remark}[Examples of $(k,\ell)$-forms] \label{rem:forms}With reference to Definition \ref{def:forms}, a bilinear wavelet form $\Lambda$ is a  $(k,0)$-form when $[\eta]\geq k$ and $\{\eta\}>\delta$. Indeed, the $(k,0)$-weak boundedness property holds by virtue of Proposition \ref{p:sparse}, which in turn relies on the almost-orthogonality considerations, cf. \cite[Prop.\ 2.5]{diplinio22wrt}, while the off-diagonal  $(k,0)$ kernel estimates follow from standard computations. 
Paraproducts \eqref{e:defpp} and their adjoints are $(k,\ell)$-forms under appropriate assumptions on the symbols $\bb$. Namely, standard computations show that for $0\leq \ell \leq k$, 
	\begin{equation}
		\label{e:SIquant} \left\| (f,g)\mapsto \Pi_\bb(f,g)\right\|_{\mathrm{SI}(k,\ell)} + \left\| (f,g)\mapsto \Pi_\bb(g,f)\right\|_{\mathrm{SI}(k,\ell)} \lesssim \| \bb\|_{F^{-\ell,0}_{\infty,\infty}}. 
	\end{equation}
Lemma \ref{prop:alpha} below provides more examples of $(k,\ell)$ forms for $0<\ell \leq k$. 
\end{remark}

\begin{lemma}\label{prop:alpha}
Let $\Lambda$ be a   $(k,\ell )$-form on $\Omega$. For any $0 < |\alpha| = j \le k-\ell$, define 
\begin{equation}
\label{e:alpha}
\Laa(f,g) \coloneqq (-1)^{j}\Lambda(f,\partial^\alpha g), \qquad f,g\in  \mathcal U(\Omega).\end{equation}
Then, $\Laa$ is a  $(k,\ell+j)$-form on $\Omega$ with
	$ \| \Laa\|_{\mathrm{SI}(k,\ell+j)} \lesssim \| \Lambda\|_{\mathrm{SI}(k,\ell)}.$
\end{lemma}
\begin{proof} Normalize $\| \Lambda\|_{\mathrm{SI}(k,\ell)}=1$.
First we relate the $\mathrm{WBP}_{k,\ell}$ and the $\mathrm{WBP}_{k,\ell+j}$ bounds. If $\zeta_S \in \Phi^{\ell+j+\delta,\Subset}(S)$, then $\ell(S)^{j} \partial^{\alpha} \zeta_S \in \Phi^{\ell+\delta,\Subset}(S)$. So, for $\chi_S \in \Phi^{k+\delta,\Subset}(S) $,
	\[ \ell(S)^{d+\ell+j} |\Laa(\chi_S,\zeta_S)| = \ell(S)^{d+\ell}|\Lambda(\chi_S,\ell(Q)^j \partial^\alpha \zeta_S)| \le  \| \Lambda\|_{\mathrm{SI}(k,\ell)}=1.\] Second, we compare the kernels. For $f \in \mathcal U(\Omega)$ and $g \in {\mathcal U}(\Omega)$ with disjoint supports, 
	\[ \Laa(f,g) = (-1)^{j} \int_{\Omega\times \Omega}  K(x,y)f(y)\partial^\alpha g(x) \, \mathrm{d}x\mathrm{d}y\]
as the disjoint supports of $f$ and $g$ force the integrals to be absolutely convergent. Since $g$ and all its derivatives vanish on the boundary of $\Omega$, integrating by parts shows
	\[ \Laa(f,g) =  \int_{\Omega\times \Omega }\partial^\alpha_1 K(x,y)f(y)g(x) \, \mathrm{d}x\mathrm{d}y, \]
and by comparison, referring to \eqref{eq:smooth-kernel}, $\Laa$ has standard $(k,\ell+j)$ kernel estimates.
\end{proof}

\begin{definition}[Compressions] \label{def:comp}  This definition exploits the observation that the trivial extension of a function $f$ with $\mathrm{supp}\, f \subset \overline \Omega,  $ denoted as usual by $f\mathbf{1}_{\overline{\Omega}}$, belongs 
 to $\mathcal U(\R^d)$ provided $f\in \CL $ and  
 $\Omega$ is a uniform domain. 
Let $\Lambda $ be a $(k,\ell)$-form on $\Omega=\R^d$. The \textit{compression}  of $\Lambda$ to $\Omega$ is denoted by $\Lambda_\Omega$ and defined by
\begin{equation}
\label{e:compdef}
\Lambda_{\Omega}:  \mathcal U(\Omega) \times \mathcal U(\Omega)  \to \mathbb C, \qquad 
\Lambda_{\Omega}(f,g)\coloneqq \Lambda\big(f\mathbf{1}_{\overline{\Omega}},g\mathbf{1}_{\overline{\Omega}}\big).
\end{equation}
By inspection, it follows that $\Lambda_\Omega $ is a $(k,\ell)$-form on $\Omega$, and $\|\Lambda_\Omega\|_{\mathrm{SI}(k,\ell)} \leq  \|\Lambda\|_{\mathrm{SI}(k,\ell)}$. Of course, on the left side of the last inequality, the $\mathrm{SI}$ norm is referrring to the domain $\Omega$, while on the right side the corresponding domain is $\R^d$. Furthermore, with reference to \eqref{e:alpha} applied respectively for $\Omega=\R^d$   and for $\Omega=\Omega$, 
\begin{equation}
\label{e:compder}
\left(\partial^\alpha \Lambda\right)_\Omega = \partial^\alpha \Lambda_\Omega  
\end{equation}
 in the sense of equality of $(k,\ell+j)$-forms on $\Omega$.
 \end{definition}
\addtocounter{other}{1} 
\subsection{Cancellation and off-diagonal estimates} \label{r:action}
A suitable definition for the action of a $(k,k)$-form  $\fL$ on polynomials is needed in the sequel.
For each $R\in \mathcal M$, let $\iota_R$ be a smooth function  with $\mathbf{1}_{2\mathsf wR} \leq \iota_R \leq \mathbf{1}_{4\mathsf wR} $. Write
\[
\begin{cases}\zeta  \in \big\{ \chi_{R,j}: 1\leq j \leq \jmath \big\}  & R\in \mathcal W \\ 
\zeta =\varphi_{R} &   R\in \mathcal S
\end{cases}
\] 
with reference to \eqref{e:triebelref}.
Also, for each $\ep>0$ small, let $\Omega_\ep\coloneqq\{x\in \Omega: \mathrm{dist}(x,\partial \Omega)>\ep\}$ and pick $\varpi_\ep\in \CL$ which equals 1 on  $\Omega_\ep\coloneqq\{x\in \Omega: \mathrm{dist}(x,\partial \Omega)>\ep\}$.
If   $0\leq |\gamma| \leq k-1$, we may define
\begin{equation}
\label{e:limitint}\fL\big((\mathsf{1}-\iota_R) \mathsf{x}^\gamma,\zeta\big)\coloneqq  \lim_{\ep\to 0^+}  \fL\big(\varpi_\ep(\mathsf{1}-\iota_R) \mathsf{x}^\gamma,\zeta\big)= \int_{\Omega \times \Omega}  K(x,y)[1-\iota_R(y)]y^\gamma\zeta(x) \, \mathrm{d}y \, \mathrm{d}x
\end{equation}
where the  equality  is obtained using first disjointness of supports, then the dominated convergence theorem, leveraging the first estimate in  \eqref{eq:smooth-kernel} and the fact that $\mathsf x^\gamma$ is of degree $k-1$. In the case $R\in \mathcal S$, $\zeta=\varphi_R$ belongs to the cancellative subclass $\Psi^{k+1,\Subset}(R)$, so that definition \eqref{e:limitint} may be extended to $|\gamma|=k$, relying on the cancellation of $\varphi$ and the second bound in \eqref{eq:smooth-kernel} to obtain
an absolutely convergent integral.
In either case,  we are allowed to set
	\begin{equation}\label{eq:poly-action} \fL(\mathsf{x}^\gamma,\zeta) \coloneqq \fL(\iota_R \mathsf{x}^\gamma,\zeta)+ \fL\big((\mathsf{1}-\iota_R) \mathsf{x}^\gamma,\zeta\big)  \end{equation} by simply relying on the definition of $(k,k)$-form of $\fL$ for the first term.
 Relying on Remarks \ref{rem:proj}, \ref{rem:dense}, it is then possible to extend \eqref{eq:poly-action} by linearity  so that 
$
\fL(  \mathsf{x}^\gamma,g)
$  is defined for all $g\in \mathcal U(\Omega)$  for all $0\leq |\gamma|\leq k-1$,  and for all  $g\in {\Sc}\mathcal U(\Omega)$ when $ |\gamma|= k$.

The action of $\fL$ on polynomials is used to obtain almost diagonalization estimates with respect to the wavelet basis. Namely,  the next lemma shows the advantage of subtracting off the $m$-th  degree Taylor-type polynomial $\PM^{m+1}$ from \eqref{eq:sob-func-rep}. 
For $m \in \mathbb N$,  $S\in \mathcal S$ define 	\[ 
	\begin{split}  
 &\tau_{S}^m(\zeta_P) \coloneqq \left\{ \begin{array}{cl} \zeta_P - \PM^{m+1}_{S}(\zeta_P) &  \rev{\frak d (P,S) \lesssim \ell (P)}, \\ \zeta_P & \rev{\frak d(P,S)  \gtrsim \ell (P)},  \end{array} \right. \qquad \zeta \in \Phi^{m+1,\Subset}(P),
	\end{split} \]
%\rev{where $A(S)$ is reregion defined in \eqref{e:ASW0} as
%	\[ A(S) = \{ P \in \mathcal M : \frak d(S,P) \lesssim \ell(P) \}, \] }
\begin{lemma}\label{lemma:t}
Let  $\fL$ be a $(k,k)$ form, $0\leq m \leq k$, $0<\eta<\delta$. Then 
	\begin{align}
	\label{eq:size-est}  &\sup_{\substack{\zeta_P  \in \Phi^{k+1,\Subset}(P) \\ \zeta_Q  \in \Phi^{k+1,\Subset}(Q)} }  |\fL(\zeta_P,\zeta_{Q})| +|\fL(\zeta_{Q},\zeta_P)|   \lesssim  \lb P,Q \rb_{0,k},  \\
&
	\sup_{\substack{\zeta_P  \in \Phi^{k+1,\Subset}(P) \\ \psi_S \in \Psi^{k+1,\Subset}(S)} }
	\label{eq:smooth-est} |\fL( \tau^m_{S}(\zeta_P),\psi_{S})| + |\fL(\psi_S, \tau^m_{S}(\zeta_P)| \lesssim   \lb P,S \rb_{\eta,m} \ell(S)^{m-k},	\end{align} 
uniformly over all $P\in \mathcal M, Q\in \mathcal W, S\in \mathcal S$ with $\ell(P)\geq \max\{\ell(Q),\ell(S)\}$.
\end{lemma}

\begin{proof}
The estimate \eqref{eq:smooth-est} is obtained in the same fashion as \cite{diplinio22wrt}*{Lemma 4.16}.
The bound \eqref{eq:size-est} comes directly from the first condition  in \eqref{eq:smooth-kernel} and \eqref{WBP}. Indeed, since $Q \in \mathcal W$, the cubes are either close and comparable in size, or far apart. In the first case, we use \eqref{WBP} to obtain $|\fL(\zeta_P, \zeta_{Q})| \lesssim \ell(Q)^{-d-k}$. Otherwise, 	\[ |\fL(\zeta_P,\zeta_{Q})| = \left|\int_{\Omega \times \Omega }K(x,y) \zeta_P(y)\zeta_{Q}(x) \,\mathrm{ d}x \mathrm dy \right| \lesssim \mathfrak d(P,Q)^{-d-k}\] from  the first estimate in \eqref{eq:smooth-kernel}. The bound for $|\fL(\zeta_{Q},\zeta_P)|$ follows by symmetry, cf.\ Remark \ref{rem:symm}. 
\end{proof}

\addtocounter{other}{1} 
\subsection{Testing conditions} Associate to the  $(k,0)$-form  $\Lambda$  the vector form 
	\begin{equation}\label{e:gradient} 
		\nabla^k \Lambda \coloneqq \left( \partial^\alpha \Lambda : |\alpha|=k \right). 
	\end{equation}
Lemma \ref{prop:alpha} tells us that each component of $\nabla^k \Lambda$ is a $(k,k)$-form. 
For each multi-index $\alpha$ with $|\alpha|=k$, define the \textit{paraproduct symbols of $\Lambda$ of order $(\alpha,\gamma)$} by the sequences 
\begin{equation} \label{def:cz}
	\begin{split}
	& \;\,\bb^{\alpha,\gamma}=\{b^{\alpha,\gamma}_{W,j}: W\in \W, 1\leq j \leq \jmath\}\cup\{b^{\alpha,\gamma}_S: S\in\Sc\},  
	\\ &
	\begin{array}{lllll}
		b^{\alpha,\gamma}_{W,j} &\coloneqq \partial^\alpha \Lambda\big(  \Tr_W\mathsf{x}^\gamma ,\chi_{W,j}\big),
		&b^{\alpha,\gamma}_{S} &\coloneqq \partial^\alpha \Lambda\big( \Tr_S\mathsf{x}^\gamma,\varphi_S\big),
		 &   \gamma\in \{0,\ldots,k-1\}^d ;\\[6pt]
		 		b^{\alpha,\gamma}_{W,j} &\coloneqq 0, 
		&b^{\alpha,\gamma}_{S} &\coloneqq \partial^\alpha \Lambda\big(   \Tr_S\mathsf{x}^\gamma,\varphi_S\big),
		 &  |\gamma| = k;
		 \\[6pt]
		b^{\alpha,\gamma}_{W,j} &\coloneqq 0, 
		&b^{\alpha,\gamma}_{S} &\coloneqq   \partial^\alpha \Lambda(\varphi_S,\mathbf 1) &
		\gamma=\star.
		\end{array}
		\end{split}
	\end{equation}
With reference to \eqref{e:seq-id}, form each function $\bb^{\alpha,\gamma}$ from the corresponding sequence of wavelet coefficients. It is immediate that 
	\begin{equation}\label{e:Finfty} \|\bb^{\alpha,\gamma}\|_{F^{|\gamma|-k,0}_{\infty,\infty}} \lesssim \|\Lambda\|_{\mathrm{SI}(k,0)}.\end{equation} 
However, stronger testing type conditions  on the symbols $\bb^{\alpha,\gamma}$ are needed to ensure Sobolev space bounds for the form $\Lambda$.
To this aim, for $1<p\leq q<\infty$,  introduce the norms
\begin{align}
 \label{e:Fn}
\|\Lambda\|_{\dot F(k,p,q)}&\coloneqq  \sup_{|\alpha|=k} \left[  \left\|\bb^{\alpha,\star}\right\|_{\dot F^{-k,0}_{1,2}(\mathcal M)} 
+ \sup_{ |\gamma|<k-\lfloor \frac dp \rfloor  } \left\|\bb^{\alpha,\gamma}\right\|_{ \dot F^{0,|\gamma|-k}_{p,2}(\mathcal M) }\right.
\\ \nonumber &\left.  \qquad\quad  
+\sup_{k-\lfloor \frac dp \rfloor \leq |\gamma|\leq k-1} \left\|\bb^{\alpha,\gamma}\right\|_{ \dot F^{0,|\gamma|-k}_{q,2}(\mathcal M) }  
 + \sup_{  |\gamma| = k}   \left\|\bb^{\alpha,\gamma}\right\|_{{\dot F^{0,0}_{1,2}(\mathcal M)} } \right], \\
\label{e:carln}
\|\Lambda\|_{\mathrm{Carl}(k,p)} &\coloneqq  \sup_{\substack{|\alpha|=k \\ |\gamma| \leq k-1}  }
\left\|\bb^{\alpha,\gamma}\right\|_{\mathrm{Carl}^{k-|\gamma|,p}(\Omega) },
\\
\label{e:czn}
\|\Lambda\|_{\mathrm{PP}(k,p,q)} &\coloneqq \|\Lambda\|_{\dot F(k,p,q)}+  \|\Lambda\|_{\mathrm{Carl}(k,p)}.
\end{align}
When $p=q$ in \eqref{e:Fn} or \eqref{e:czn}, we simply write $\dot F(k,p)$ or $\mathrm{PP}(k,p)$ instead.
Referring  to Definition \ref{def:pp},   associate to $\Lambda$ the vector of paraproduct forms  
	\[  
	\vec \Pi_\Lambda^k(f,g) 
		\coloneqq \left( \Pi_{\mathfrak{b}^{\alpha,\star}}(g,f) 
		+ \sum_{\gamma \in \{0,\ldots,k \}^d} \Pi_{\mathfrak b^{\alpha,\gamma}}(\partial^\gamma f,g) :|\alpha|=k \right).
	\]
The   norms in \eqref{e:czn} are relevant as Proposition \ref{paraproduct-main}, together with \eqref{e:Fembdot} and \eqref{e:Finfty}, immediately yield the estimates
\begin{equation}
\label{e:ppestthm}
\left|\vec \Pi_\Lambda^k(f,g) \right| \lesssim \begin{cases} \left[ \|\Lambda\|_{\mathrm{SI}(k,0)}+\|\Lambda\|_{\mathrm{PP}(k,p)}\right]  \|f\|_{W^{k,(p,1)}(\Omega)} \|g\|_{L^{p'}(\Omega)}   & 1<p<\infty,\\[6pt] \left[ \|\Lambda\|_{ \mathrm{SI}(k,0)}+\|\Lambda\|_{\mathrm{PP}(k,p,q)}\right]   \|f\|_{W^{k,p}(\Omega)} \|g\|_{L^{p'}(\Omega)}   & 1<p<q<\infty.
\end{cases}
\end{equation}
The next theorem  shows that the $k$-th derivative of a $(k,0)$ form $\nabla^k\Lambda(f,g)$ differs from the associated paraproduct vector by a (constant multiple of) a wavelet form acting on $(\nabla^k f,g)$. The testing type Corollary \ref{cor:sob} below  then ensues as an immediate consequence of \eqref{e:ppestthm}.
\begin{theorem}\label{thm:main}
Let $\Omega \subset \mathbb R^d$ be a uniform domain and $\Lambda$  be a $(k,0)$-form on $\Omega$ with $\|\Lambda\|_{\mathrm{SI}(k,0)}=1$. Then, there exists a wavelet form $\mathrm V$ such that 
	\[ \nabla^k \Lambda( f, g ) = C\mathrm V(\nabla^k f,g) + \vec \Pi^k_\Lambda(f,g) , \qquad f,g \in \mathcal U(\Omega).\] \end{theorem}
\begin{cor}\label{cor:sob}
Let $\Lambda$ be a $(k,0)$-form on a uniform domain $\Omega$. 
Then the adjoint operator 
\begin{equation}
\label{e:adjT}
T: \mathcal U(\Omega) \mapsto  (\mathcal U(\Omega))', \qquad \langle Tf,g\rangle\coloneqq \Lambda(f,g)
\end{equation}
possesses the bounded extension properties
 \begin{align} \|\nabla^k T f\|_{L^p(\Omega)} &\lesssim \left[ \|\Lambda\|_{\mathrm{SI}(k,0)}+\|\Lambda\|_{ \mathrm{PP}(k,p)}\right]  \|f\|_{W^{k,(p,1)}(\Omega)}, \qquad 1<p<\infty,
\\
 \|\nabla^k T f\|_{L^p(\Omega)} &\lesssim\left[ \|\Lambda\|_{\mathrm{SI}(k,0)}+\|\Lambda\|_{\mathrm{PP}(k,p,q)}\right] \|f\|_{W^{k,p}(\Omega)}, \qquad 1<p<q<\infty. \end{align}
\end{cor}

Before the proof of Theorem \ref{thm:main}, we state  a weighted version of Corollary \ref{cor:sob} involving  
\begin{align}\label{e:carlnw}  \|\Lambda \|_{\mathrm{Carl}_w(k,p)} &\coloneqq  \sup_{\substack{|\alpha|=k \\ |\gamma| \leq k-1}  }
\left\|\bb^{\alpha,\gamma}\right\|_{\mathrm{Carl}^{k-|\gamma|,p}(\Omega,w) }.
\end{align}
In service of the reader, in the following result, we state the maximal dependence on the various weight characteristics. The bounds for certain terms from the representation in Theorem \ref{thm:main} may have sharper dependence, which  may  easily be tracked in the proof.

\begin{cor}\label{cor:weight} Let $\Lambda$ be a $(k,0)$-form with   adjoint $T$ as in \eqref{e:adjT}. Then, for $1<p,s<\infty$,
	\begin{align}
		\label{e:wcor1} \|\nabla^k T f \|_{L^p(\Omega,w)} 
			&\lesssim\begin{array}{l} \displaystyle \quad \mathcal G\left(  [w]_{\mathrm{A}_p(\Omega)}, [w]_{\RH_s(\Omega^\circ)}\right) \\ \times \displaystyle
			\left( \left \Vert \Lambda \right \Vert_{\mathrm{SI}(k,0)} + \|\Lambda\|_{\dot F(k,ps')}  +  \|\Lambda\|_{\mathrm{Carl}_w(k,p)} \right) \|f\|_{W^{k,p}(\Omega,w)},\end{array} \\ 
		\label{e:wcor2} \|\nabla^k T f \|_{L^p(\Omega,w)} 
			&\lesssim  \mathcal G\left( [w]_{\mathrm{A}_p(\Omega)} \right) 
			\left(  \left \Vert \Lambda \right \Vert_{\mathrm{SI}(k,0)} + \|\Lambda\|_{\dot F(k,\infty)}  +  \|\Lambda\|_{\mathrm{Carl}_w(k,p)} \right) \|f\|_{W^{k,p}(\Omega,w)}.
	\end{align}
\end{cor}

\begin{proof}
 It suffices to bound each term appearing in the representation of $\nabla^k \Lambda$ in Theorem \ref{thm:main}. Let $\sigma$ be the $p$-th dual weight of $w$ and $g \in L^{p'}(\Omega,\sigma)$ with unit norm. First, by Proposition \ref{p:WFw},
	\[  \abs{\mathrm{V}(\nabla^k f,g)} \lesssim [w]_{\mathrm{A}_p(\Omega)}^{\max\left\{1,\frac{1}{p-1} \right\} }  \left \Vert \nabla^k f \right \Vert_{L^p(\Omega,w)}.\] 
Next, by the second estimate in Proposition \ref{p:sparse2} and the first estimate in Proposition \ref{p:type3w}, 
	\[ \begin{aligned}& \quad 
	\abs{\Pi_{\bb^{\alpha,\gamma} }(\partial^\gamma f,g)} + \abs{\Pi_{\bb^{\alpha,\star}}(g,f)}\\ & \lesssim [w]_{\mathrm{A}_p({\Omega^\circ})}^{\max\left\{1,\tfrac{1}{p-1} \right\} } \left( \left\Vert \bb^{\alpha,\gamma} \right\Vert_{\dot F^{0,0}_{1,2}(\mathcal M)} + \left\Vert \bb^{\alpha,\star} \right\Vert_{\dot F^{-k,0}_{1,2}(\mathcal M)} \right) \|\nabla^k f\|_{L^{p}(\Omega,w)}
	\end{aligned} \]
uniformly over $|\alpha|=|\gamma|=k$. For the remaining paraproduct forms $\Pi_{\bb^{\alpha,\gamma}}(\partial^\gamma f,g)$ for $|\alpha|=k$ and $|\gamma| \le k-1$, we will apply Proposition \ref{p:pp-weight} with $n=k-|\gamma|$ to obtain 	
\[ \begin{split}
	&\quad\;\abs{ \Pi_{\bb^{\alpha,\gamma}}(\partial^\gamma f, g) } \\ &\lesssim \left\{ \begin{array}{l} \mathcal G\left([w]_{\A_p(\Omega^\circ)},[w]_{\RH_s(\Omega^\circ)} \right) \left( \left\Vert \bb^{\alpha,\gamma} \right \Vert_{\dot F^{0,|\gamma|-k}_{ps', 2}(\mathcal M)} + \left\Vert \bb^{\alpha,\gamma} \right\Vert_{\mathrm{Carl}^{n, p}(\Omega, w)} \right) \|\partial^\gamma f\|_{W^{k-|\gamma|,p}(\Omega, w)}, \\
~[w]_{\mathrm{A}_p({\Omega^\circ})}^{2\max\left\{ 1, \frac{1}{p-1} \right\} }  \left( \left\Vert \bb^{\alpha,\gamma} \right \Vert_{\dot F^{0,|\gamma|-k}_{\infty, 2}(\mathcal M)} + \left\Vert \bb^{\alpha,\gamma} \right\Vert_{\mathrm{Carl}^{n, p}(\Omega, w)} \right) \|\partial^\gamma f\|_{W^{k-|\gamma|,p}(\Omega, w)}. \end{array} \right.\end{split} \]
Observing that $\|\partial^\gamma f\|_{W^{k-|\gamma|,p}(\Omega,w)} \le \|f\|_{W^{k, p}(\Omega, w)}$, and recalling \eqref{e:Fn} and \eqref{e:carlnw} finishes the proof.
\end{proof}

\addtocounter{other}{1}
\subsection{Proof of Theorem \ref{thm:main}}
Let $|\alpha|=k$, set $\fL = \partial^\alpha \Lambda$, and let $0<\eta < \delta$. 
Let $f,g\in \mathcal U(\Omega)$. \rev{The general road map is as follows. 
\begin{enumerate}
\item %\textit{Step 1:} 
Expand $f$ and $g$ in the orthonormal wavelet basis $\mathfrak B^2(\Omega,k)$ from Proposition \ref{prop:triebel}:
	\[ \fL(f,g) = \fL(\mathcal W f,\mathcal W g) + \fL(\mathcal S f,\mathcal W g) + \fL(\mathcal W f,\mathcal S g) + \fL(\mathcal S f,\mathcal S g) = \sum_{i=1}^4 \sigma_i. \]

\item %\textit{Step 2:} 
Whenever $f$ is paired with a noncancellative wavelet (i.e. projected onto $\mathcal W$ in $\sigma_1$ and $\sigma _3$), we subtract off a Taylor polynomial $\PM^k_Qf$ for some specially chosen Whitney cube $Q \in \mathcal W$. This will allow us to integrate by parts according to Lemma \ref{lemma:tele}.

\item %\textit{Step 3:} 
When $S$ or $R$ belong to $\mathcal S$, and interactions between scales prevail over distance, i.e. $\frak d(S,R) \lesssim \max\{\ell(R),\ell(S)\}$, 
we have to modify $\fL(\zeta_S,\psi_R)$ to achieve the decay estimates from Lemma \ref{lemma:t}. In fact, if $\ell(S) \ge \ell(R)$ (respectively $\ell(R) \ge \ell(S)$), we subtract off  $\PM^{k+1}_R\zeta_S$ in the first argument (respectively $\PM^{1}_S\psi_R$ in the second argument) of $\fL(\zeta_S,\psi_R)$.

\item %\textit{Step 4:} 
Steps 2 and 3 lead to the splitting
		\[ \begin{aligned} &\sigma_1 = \sigma_{1,\mathrm{wave}} + \sigma_{1,\mathrm{para},k-1} \\
		& \sigma_2 = \sigma_{2,\mathrm{wave}} + \sigma_{2,\mathrm{para},\star} \\
		&\sigma_3 = \sigma_{3,\mathrm{wave}} + \sigma_{3,\mathrm{para},k-1}+\sigma_{3,\mathrm{para},k} \\
		& \sigma_4 = \sigma_{4,\mathrm{wave}} + \sigma_{4,\mathrm{para},k} + \sigma_{4,\mathrm{para},\star} . \end{aligned} \]
To reiterate, the principles from Steps 2 and 3 which lead to the above splitting are
	\[ \begin{aligned} &\mathcal Wf \to \chi_W(f-\PM_Q^kf) \to \text{Lemma \ref{lemma:tele}}, \\ %\chi_{W,Q}^{-k} (\nabla^k f) \\
	&\mathcal Sg \to \fL(\varphi_S-\PM^{k+1}_R\varphi_S,\zeta_R) \to \text{Lemma \ref{lemma:t}}, \\
	&\mathcal Sf \to \fL(\zeta_S,\varphi_R-\PM^{1}_S\varphi_R,\zeta_R) \to \text{Lemma \ref{lemma:t}}. \end{aligned} \]
\item Each
$\sigma_{i,\mathrm{wave}}$ will be directly converted into some component of a wavelet form $\mathrm{V}(\nabla^k f,g)$ using the wavelet averaging Lemmata \ref{lemma:avg-1} and \ref{lemma:avg-2}, together with the estimates on the wavelet coefficients of $\fL$ from Lemma \ref{lemma:t}. 
\item The $\sigma_{i,\mathrm{para},\kappa}$ terms will combine to form the paraproduct $\vec \Pi^k_\Lambda(f,g)$ as follows:
	\[ \begin{aligned} &\sigma_{1,\mathrm{para},k-1} + \sigma_{3,\mathrm{para},k-1} = \sum_{|\gamma| \le k-1} \sum_{1 \le j \le \jmath} \sum_{W \in \mathcal W} b^{\alpha,\gamma}_{W,j} \zeta_{W,j}(\partial^\gamma f) \overline{\chi_{W,j}}(g),\\
	& \sigma_{2,\mathrm{para},k} + \sigma_{4,\mathrm{para},k} = \sum_{|\gamma| \le k} \sum_{S \in \mathcal S} b^{\alpha,\gamma}_{S} \zeta_{S}(\partial^\gamma f) \overline{\varphi_{S}}(g),\\
	& \sigma_{2,\mathrm{para},\star} + \sigma_{4,\mathrm{para},\star} = \sum_{S \in \mathcal S} b^{\alpha,\star}_{S} \varphi_{S}(f) \overline{\zeta_{S}}(g). \end{aligned} .\]
	for suitably chosen $\zeta_{W,j},\zeta_S \in C \Phi^{0,\eta}(W)$.
\end{enumerate}}
Referring to Proposition \ref{prop:triebel}, we first expand $g$ using the orthonormal basis $\mathfrak B^2(\Omega,k)$ so that %and subsequently  introducing the polynomials $\PM^{k}_Q f$ and $\PM^{k}_{W(R)}f$ from \eqref{eq:sob-func-rep},
\begin{equation}\label{eq:decomp}\begin{aligned} 
\fL(f, g)
		&=\sum_{\substack{Q \in \W \\ 1\leq j\leq \jmath }}  |Q| \fL( f, \chi_{Q,j}) \overline{\chi_{Q,j}}(g) + \sum_{R \in \mathcal S} |R| \fL( f,\varphi_{R}) \overline{\varphi_{R}}(g)%= \Sigma_1+\Sigma_2+\Sigma_3,\\
%\Sigma_1&\coloneqq\sum_{\substack{Q \in \W \\ 1\leq j\leq \jmath }}  |Q| \fL( f-\PM^{k}_Qf, \chi_{Q,j}) \overline{\chi_{Q,j}}(g),
%\\ \Sigma_2&\coloneqq \sum_{R \in \mathcal S} |R| \fL( f-\PM^{k}_{W(R)}f,\varphi_{R}) \overline{\varphi_{R}}(g), 
%\\ \Sigma_3&\coloneqq
%		   \sum_{\substack{Q \in \W \\ 1\leq j\leq \jmath }}  |Q| \fL( \PM^{k}_Q f, \chi_{Q,j}) \overline{\chi_{Q,j}}(g) + \sum_{R \in \mathcal S} |R| \fL( \PM^{k}_{W(R)}f,\varphi_{R}) \overline{\varphi_{R}}(g) .
	\end{aligned}\end{equation}
Rewriting $f=f-\PM_Q^kf + \PM_Q^k f$ in the first summation, expanding $f-\PM_{Q}^{k}f$ in the basis $\mathfrak B^2(\Omega,k)$, and subsequently introducing $\tau^0$, %$\Sigma_A$ is the sum of the four terms
	\[ \begin{aligned}  
	  &\sum_{\substack{Q \in \W \\ 1\leq j\leq \jmath }}  |Q| \fL( f, \chi_{Q,j}) \overline{\chi_{Q,j}}(g) = \sigma_{1,\mathrm{wave}} + \sigma_{1,\mathrm{para},k-1} + \sigma_{2,\mathrm{wave}} + \sigma_{2,\mathrm{para},\star}, \\%\Sigma_{1,\mathrm{t}}
	  &\sigma_{1,\mathrm{wave}}\coloneqq \sum_{\substack{Q,P\in \W \\ 1\leq j,\ell \leq \jmath }} |Q| |P| \fL( \chi_{P,\ell},\chi_{Q, j}) \chi_{P,\ell}(f-\PM^{k}_{Q} f) \overline{\chi_{Q,j}}(g), \\
	 &\sigma_{1,\mathrm{para},k-1} = \sum_{\substack{Q \in \W \\ 1\leq j\leq \jmath }}  |Q| \fL( \PM^{k}_Q f, \chi_{Q,j}) \overline{\chi_{Q,j}}(g) \\
	&\sigma_{2,\mathrm{wave}} %= \Sigma_{1,\mathrm{e}}&
	\coloneqq \sum_{\substack{S \in \mathcal S \\ Q \in W \\  1\leq j  \leq \jmath }} |S| |Q| \fL(\varphi_S, \tau^0_S(\chi_{Q,j})) \varphi_S(f) \overline{\chi_{Q,j}}(g), \\ 
	\\
	&\sigma_{2,\mathrm{para},\star}% = \Sigma_{1,\mathrm{p},\star}&
	\coloneqq \sum_{\substack{S \in \mathcal S \\ Q \in \mathcal W \\  1\leq j  \leq \jmath }} |S| |Q| \fL(\varphi_S, \PM^1_S(\chi_{Q,j})) \varphi_S(f) \overline{\chi_{Q,j}}(g).   \end{aligned} \]
The vanishing moments of $\varphi_{S}$ allowed us to erase the polynomial $\PM^{k}_Q f$ in the $\sigma_{2}$ components. Moreover, by support considerations of $\chi_{Q,j}$ and $\theta_S$, the summation in $\sigma_{2,\mathrm{para},\star}$ is equivalent to the summation over $S \in \mathcal S$, $Q \in \mathcal W$ such that $\frak d(S,Q) \lesssim \ell(Q)$ as long as the implicit constant is chosen properly.
Now we proceed to converting $\sigma_{1,\mathrm{wave}}$ and $\sigma_{2,\mathrm{wave}}$ into wavelet forms using the averaging  Lemmata \ref{lemma:avg-1} and  \ref{lemma:avg-2}, thereby diagonalizing the form. We will handle the paraproduct components in the final step of the proof. %To do so, we will rely on the decay of the matrix coefficients proved in Lemma \ref{lemma:t}. 

Let us begin with $\sigma_{1,\mathrm {wave}}$. Using the decomposition \eqref{eq:tele} of $\chi_{P,\ell}(f-\PM^{k}_{Q}f)$ from Lemma \ref{lemma:tele}, we write
	\begin{equation}\label{eq:s1t}\begin{split} \sigma_{1,\mathrm {wave}} &= \sum_{W \in\mathcal W} |W| \eta_W( \nabla^k f,  g), \\ \eta_W &\coloneqq \sum_{\substack{P,Q \in \mathcal W \\ W \in [P,Q]\\ 1\leq j,\ell  \leq \jmath }} \frac{|P| |Q|}{|W|} \ell(W) \mathfrak d(W,P)^{k-1}  \fL(\chi_{P,\ell}, \chi_{Q,j}) \chi_{P,W}^{-k} \otimes \overline{\chi_{Q,j}}.\end{split} \end{equation}
If $\eta_W \in C \Phi^{0,\eta}_2(W)$, then $\sigma_{1,\mathrm{wave}}$ is indeed one component of our desired wavelet form $\mathrm{V}(\nabla^k f,g)$. To see this, first note that since $P$ and $Q$ are Whitney cubes the estimate \eqref{eq:size-est}  from Lemma \ref{lemma:t} immediately yields 
	\[  \left\vert  \fL(\chi_{P,\ell}, \chi_{Q,j}) \right\vert \lesssim \frak d (P,Q)^{-d-k} = (P,Q)_{0,k}.\]  
But this shows $\eta_W$ to be precisely of the form \eqref{eq:avg-theta} from Lemma \ref{lemma:avg-2}, so we obtain that  $\eta_W$ belongs to $C \Phi^{0,k}_2(W)$.
 A similar procedure will be carried out for $\sigma_{2,\mathrm{wave}}$. Here   $f$ is paired with a cancellative wavelet, so we may integrate by parts directly by \eqref{eq:ibp-psi} to obtain
	\begin{equation} \label{eq:s1e} \sigma_{2,\mathrm{wave}} = \sum_{W \in \mathcal W} |W| \phi_W(\nabla^k f,  g), \quad \phi_{W} \coloneqq \sum_{\substack{Q \in \mathcal W \\ 1 \le \ell \le \jmath \\ S \in \mathcal D(W) }} \frac{|S| |Q|}{|W|} \fL(\varphi_{S},\tau^0_S(\chi_{Q,\ell})) \ell(S)^k\varphi_S^{-k} \otimes \overline{\chi_{Q,\ell}}.\end{equation}
We aim to show that $\phi_W$ is of the form \eqref{eq:avg-phi1} from Lemma \ref{lemma:avg-1}. To do so, we only need to check the coefficient estimate
	\begin{equation}\label{eq:coeff1} \ell(S)^k|\fL(\varphi_S,\tau^0_S(\chi_{Q}))| \lesssim \ell(S)^\eta \lb S,Q \rb_{0,\eta}. \end{equation} 
Indeed, if $\ell(S) \gtrsim \ell(Q)$ then $\tau_S^0(\chi_{Q,\ell})=\chi_{Q,\ell}$, and hence we simply appeal to the size estimate \eqref{eq:size-est} in Lemma \ref{lemma:t} to obtain that \eqref{eq:coeff1} holds for any $\eta\le k$. On the other hand, if $\ell(S) \lesssim \ell(Q)$, we directly apply the smoothness estimate \eqref{eq:smooth-est} from the same Lemma \ref{lemma:t}. Therefore $\phi_W \in C\Phi^{0,\eta}_{2}(W)$. 
%Leave alone the term $\Sigma_{1,\mathrm p,\star}$ for now and proceed to $\Sigma_2$. 

Now we turn to the second summation in the initial splitting of $\fL(f,g)$ and expand $f$ into the orthonormal basis $\frak B^2(\Omega,k)$ so that 
	\[ \begin{aligned} &\sum_{R \in \mathcal S} |R| \fL( f,\varphi_{R}) \overline{\varphi_{R}}(g) = \sigma_3 + \sigma_4, \\
	&\sigma_3 \coloneqq \sum_{\substack{P \in \W \\ 1\leq \ell \leq \jmath \\ R\in \mathcal S }} |P||R| \fL (\chi_{P,\ell},\varphi_{R}) \chi_{P,\ell}( f) \overline{\varphi_{R}}(g),
\\ &\sigma_{4} \coloneqq 
\sum_{\substack{S \in \Sc \\ R\in \mathcal S }} |S||R| \fL (\varphi_S,\varphi_{R}) \varphi_{S}(f) \overline{\varphi_{R}}(g). \end{aligned}\]
%	Expanding $f-\PM^{k}_{W(R)}f$ into the basis $\mathcal B^2(\Omega,k)$, introducing $\tau^k$ and $\tau^0$, $\Sigma_2 $ is decomposed into the sum of the five terms
We first expand $\sigma_3$ according to Steps 2 and 3 above. To do so, note that since $\varphi_S(\PM^k_{Q}f)=0$ for any $Q \in \mathcal M$ and $S \in \Sc$, we have the identity
	\[ \PM^k_{Q}f = \sum_{\substack{P \in W \\ 1 \le \ell \le \jmath}} \chi_{P,\ell}(\PM^k_Q f) \overline{\chi_{P,\ell}}.\]
However for any $R \subset Q$, $\tau_R^k(\PM^k_Q f) =0$. Therefore
	\begin{equation}\label{e:zero} \sum_{\substack{P \in W \\ 1 \le \ell \le \jmath}} \chi_{P,\ell}(\PM^k_Q f) \tau^m_R (\overline{\chi_{P,\ell}})=0. \end{equation}
Taking \eqref{e:zero} into account, 
	\[ \begin{aligned} &\sigma_3 = \sigma_{3,\mathrm{wave}} + \sigma_{3,\mathrm{para},k}, \\
	& \sigma_{3,\mathrm{wave}} %= \Sigma_{2,\e}
	\coloneqq \sum_{\substack{P \in \W \\ 1\leq \ell \leq \jmath \\ R\in \mathcal S }} |P||R| \fL (\tau^k_R(\chi_{P,\ell}),\varphi_{R}) \chi_{P,\ell}( f-\PM^{k}_{W(R)} f) \overline{\varphi_{R}}(g), \\
	& \sigma_{3,\mathrm{para},k} %= \Sigma_{2,\mathrm p,\mathrm a} 
	\coloneqq \sum_{\substack{R \in \mathcal S \\ P \in \W \\ 1\leq \ell \leq \jmath \\ }} |P||R| \fL (\PM^{k+1}_R(\chi_{P,\ell}),\varphi_{R}) \chi_{P,\ell}(f)\overline{\varphi_{R}}(g),
\end{aligned} \]
while
\[ \begin{aligned}
& \sigma_4 = \sigma_{4,\mathrm{wave}} + \sigma_{4,\mathrm{para},k} + \sigma_{4,\mathrm{para},\star},\\ 
&\sigma_{4,\mathrm{wave}} %= \Sigma_{2,\mathrm {s}} 
\coloneqq 
\sum_{\substack{S \in \Sc \\ R\in \mathcal S }} |S||R| \fL (\tau^k_R(\varphi_S),\tau^0_S(\varphi_{R})) \varphi_{S}(f) \overline{\varphi_{R}}(g) ,\\
&\sigma_{4,\mathrm{para},k} %= \Sigma_{2,\mathrm p,\mathrm a} 
\coloneqq \sum_{\substack{R\in \Sc \\ S\in \mathcal S \\ \frak d (R,S) \lesssim \ell(S)}} |S||R| \fL (\PM^{k+1}_R(\varphi_S),\varphi_{R}) \varphi_{S}(f) \overline{\varphi_{R}}(g),
\\
&\sigma_{4,\mathrm{para},\star}%\Sigma_{2,\mathrm p,\mathrm b}&\coloneqq -\sum_{\substack{R \in \mathcal S \\ P \in \W \\ 1\leq \ell \leq \jmath \\ }} |P||R| \fL (\PM^{k+1}_R(\chi_{P,\ell}),\varphi_{R}) \chi_{P,\ell}(\PM_{W(R)}^{k} f)\overline{\varphi_{R}}(g),\\ 
%\Sigma_{2,\mathrm p,\star}
\coloneqq \sum_{\substack{S \in \Sc \\ R\in \mathcal S \\ \frak d(R,S) \lesssim \ell(R)}} |S||R| \fL (\varphi_S,\PM^1_S(\varphi_{R})) \varphi_{S}(f) \overline{\varphi_{R}}(g).
\end{aligned} \]
%As in the expansion of $\Sigma_1$, first we have relied on the vanishing moments of $\varphi_S$ in $\Sigma_{2,\mathrm{s}}$, the second term in $\Sigma_{2,\mathrm p ,\mathrm a}$, and $\Sigma_{2,\mathrm p ,\star}$. Second, the same reasoning applied to $\Sigma_{1,\mathrm p ,\star}$ shows the summation in the first term of $\Sigma_{2,\mathrm p,\mathrm a}$ could also be taken over $R \in \mathcal S$, $P \in A(R) \cap \mathcal W$.
Let us again handle the wavelet form terms, postponing the paraproducts until later. Again using the cancellation of $\varphi_S$ to integrate by parts according to \eqref{eq:ibp-psi} in $\sigma_{4,\mathrm {wave}}$,
	\begin{equation} \label{eq:s2s}\sigma_{4,\mathrm{wave}} = \sum_{R \in \DW} |R| \psi_{R}(\nabla^k f)\overline{\varphi_{R}}(g), \quad \psi_{R} \coloneqq \sum_{S \in \mathcal S} |S| \ell(S)^k \fL(\tau^k_R(\varphi_S), \tau^0_S(\varphi_{R})) \varphi_S^{-k}.\end{equation}
In the case $\ell(S) \ge \ell(R)$, using the estimate \eqref{eq:smooth-est} with $m=k$ from Lemma \ref{lemma:t}, 
	\[ \ell(S)^k|\fL(\tau^k_R(\varphi_S),\tau^0_S(\varphi_{R}))| =  \ell(S)^k |\fL(\tau^k_{R}(\varphi_S),\varphi_{R})|\lesssim \ell(S)^k \lb S,R \rb_{\eta,k} \le \lb S,R\rb_{\eta,0}\] for each $\eta<\delta$.
On the other hand, if $\ell(S) \le \ell(R)$, argue as for \eqref{eq:coeff1}. So in both cases 
	\begin{equation}\label{eq:coeff2} \ell(R)^k|\fL(\tau^k_R(\varphi_S),\tau_S^0(\varphi_{R}))|  \lesssim \lb S,R \rb_{\eta,0}. \end{equation} 
This shows $\psi_{R}$ to be of the form \eqref{eq:avg-psi} in Lemma \ref{lemma:avg-1} and therefore $\psi_{R} \in C \Psi^{\eta,0}(R)$.
The term $\sigma_{3,\mathrm {wave}}$ will be handled like $\sigma_{2,\mathrm {wave}}$ above, but with the additional difficulty introduced by the telescoping  \eqref{eq:tele}, since $f$ is no longer paired with a cancellative wavelet. 
 So, decomposing $\chi_{P,\ell}( f-\PM^{k}_{W(R)} f) $ according to \eqref{eq:tele} in Lemma \ref{lemma:tele},
	\begin{equation}\label{eq:s2e}\begin{split}
	 \sigma_{3,\mathrm {wave}} &=\sum_{W \in \W}|S|  \upsilon_W(\nabla^k f, g) ,\\  \upsilon_W& \coloneqq \sum_{\substack{P \in \mathcal W\\ 1\leq \ell \leq \jmath \\ R \in \mathcal S \\  W\in [P,W(R)]}} \frac{|P| |R|}{|W|} \fL( \tau^k_R(\chi_{P,\ell}), \varphi_{R}) \ell(W) \mathfrak d(W,P)^{k-1} \chi_{P,W}^{-k} \otimes \overline{\varphi_{R}}. \end{split}\end{equation}
When $\ell(R) \ge \ell(P)$, apply \eqref{eq:size-est} from Lemma \ref{lemma:t} to see that $\abs{ \fL( \tau^k_R(\chi_{P,\ell}), \varphi_{R}) } = \abs{ \fL(\chi_{P,\ell},\varphi_{R})} \lesssim \lb P,R \rb_{0,k}$. For the other case, $\ell(R) \le \ell(P)$, by  \eqref{eq:smooth-est}, $\abs{\fL(\tau^k_{R}(\chi_{P,\ell}),\varphi_{R})}\lesssim \lb P,R \rb_{\eta,k}$.
	 It follows that $\upsilon_W$ is of the form \eqref{eq:avg-phi2} in Lemma \ref{lemma:avg-2} and belongs to $C\Phi^{0,\rho}_{2} (W)$ for any $\rho < k$.
Summarizing, we have shown
	\begin{equation}\label{e:summ} %\Sigma_{1,\mathrm t} + \Sigma_{1,\mathrm e} + \Sigma_{2,\mathrm e} + \Sigma_{2,\mathrm s} 
	\sum_{i=1}^4 \sigma_{i,\mathrm{wave}} = \sum_{W \in \W} |W| \left[\eta_W + \phi_W + \upsilon_W \right] (\nabla^k f,  g) + \sum_{R \in \mathcal S} |R| \psi_R(\nabla^k f) \overline{\varphi_R}(g), \end{equation}
which, since $\eta_W,\phi_W,\upsilon_W \in C\Phi^{0,\eta}_2(W)$ and $\psi_R,\varphi_R \in C\Psi^{\eta,0}(R)$, is a constant multiple of a wavelet form.
It remains to realize the remaining terms %$\Sigma_{3}$, $\Sigma_{1,\mathrm p,\star}$, $\Sigma_{2,\mathrm p,\mathrm a}$, $\Sigma_{2,\mathrm p ,\mathrm b}$, and $\Sigma_{2,\mathrm p,\star}$ 
as paraproduct forms. 

Our first step towards this end is to obtain a precise formula connecting $\fL(\PM^m_Q f,\zeta_R)$ to the paraproduct coefficients $b^{\alpha,\gamma}_{R}=\fL( \Tr_R \mathsf{x}^\gamma,\zeta_R)$ from \S \ref{def:cz}; see \eqref{e:LT} below.
To do so, let  $Q \in \mathcal M$.% and $R \subset CQ$.
Expanding $(x-y)^\alpha = \sum_{\gamma \le \alpha} \binom{\alpha}{\gamma} (x-c(Q))^\gamma (c(Q)-y)^{\alpha-\gamma}$,
	\[ \PM^{m}_Q f = \sum_{|\gamma| \le m-1} \frac{\Tr_Q \mathsf{x}^\gamma}{\gamma!} \PM^{m-|\gamma|}_Q(\partial^\gamma f)(c(Q)).\]
For each $m=0,1,\ldots,k-1$, refer to Lemma \ref{l:poly} to obtain a family 
	\[ \left\{\upsilon_{Q,m} \in C\Phi^{k+1,\Subset}(Q)  : Q \in \mathcal M\right\}\]
such that 
	\[ \PM^{m}_Q(f)(c(Q)) = \upsilon_{Q,m} (f)\] 
	%\eqref{e:ibp3} holds with $x=c(R)$. 
Then, by definition of $b^{\alpha,\gamma}_Q$, for any $\zeta_Q \in \mathfrak B(\Omega,k)$,% and $CQ \supset R$,
	\begin{equation}\label{e:LT} \fL(\PM^{m}_{Q} f,\zeta_Q) =  
	\sum_{|\gamma| \le m-1} b^{\alpha,\gamma}_Q \frac{\PM^{m-|\gamma|}_Q \partial^\gamma f(c(Q))}{\gamma!} = \sum_{|\gamma| \le m-1} b^{\alpha,\gamma}_R \upsilon_{Q,m-|\gamma|}(\partial^\gamma f) .\end{equation}
In this way,
	\[ \begin{aligned} \sigma_{1,\mathrm{para},k-1} &= \sum_{\substack{Q \in \W \\ 1\leq j\leq \jmath }}  |Q| \fL( \PM^{k}_Q f, \chi_{Q,j}) \overline{\chi_{Q,j}}(g) = \sum_{|\gamma| \le k-1} \sum_{\substack{W \in \mathcal W \\ 1 \le j \le \jmath} } |W| b^{\alpha,\gamma}_{W,j} \upsilon_{W,W,k-|\gamma|}( \partial^\gamma f) \overline{\chi_{W,j}}(g) %\\ &\quad + \sum_{R \in \mathcal S} |R| \mathfrak b^{\alpha,\gamma}_R \upsilon_{W(R),c(R),k-|\gamma|}(\partial^\gamma f) \overline{\varphi_{R}}(g). 	
	\end{aligned}\]
%Likewise, applying \eqref{e:LT} together with Lemma \ref{l:pp-avg} \eqref{e:poly-basis}, Lemma \ref{l:poly} \eqref{e:ibp3},
%	\[ \begin{aligned} \Sigma_{2,\mathrm p,\mathrm b} &= -\sum_{|\gamma| \le k} \sum_{R \in \mathcal S} |R| \mathfrak b^{\alpha,\gamma}_R \overline{\varphi_R}(g) \sum_{\substack{P \in \W \\ 1\leq \ell \leq \jmath \\ }} |P|  \PM^{k+1-|\gamma|}_R(\partial^\gamma \chi_{P,\ell})(c(R)) \chi_{P,\ell}(\PM_{W(R)}^{k} f)  \\
%				&= -\sum_{|\gamma| \le k-1} \sum_{R \in \mathcal S} |R| \mathfrak b^{\alpha,\gamma}_R \upsilon_{W(R),c(R),k-|\gamma|} (\partial^\gamma f) \overline{\varphi_{R}}(g). 
%		\end{aligned}\]
Again using the identity \eqref{e:LT}, %and Lemma \ref{l:pp-avg} \eqref{e:q-avg},
	\[ \begin{aligned} 
	&\sigma_{3,\mathrm{para},k} + \sigma_{4,\mathrm{para},k} %= \Sigma_{2,\mathrm p,\mathrm a} &= \sum_{|\gamma| \le k} \sum_{R \in \mathcal S} |R| \mathfrak b^{\alpha,\gamma}_R \overline{\varphi_R}(g) \\ &\qquad
%	\;\;  \times\left[ \sum_{\substack{P \in \mathcal W \\ 1 \le \ell \le \jmath} } |P| \PM^{k+1-|\gamma|}_R( \partial^\gamma \chi_{P,\ell} ) \chi_{P,\ell}(f) + \sum_{S \in \mathcal S} |S| \PM^{k+1-|\gamma|}_R( \varphi_S) \varphi_S(f) \right] \\
	%= \Sigma_{2,\mathrm p,\mathrm a} 
	= \sum_{|\gamma| \le k} \sum_{R \in \mathcal S} |R| \mathfrak b^{\alpha,\gamma}_R \overline{\varphi_R}(g) \\ &\quad   \times\left[ \sum_{\substack{P \in \mathcal W \\ 1 \le \ell \le \jmath} } |P| \upsilon_{R,k+1-|\gamma|}( \partial^\gamma \chi_{P,\ell} ) \chi_{P,\ell}(f) + \sum_{\substack{S \in \mathcal S \\ \frak d (R,S) \lesssim \ell(R)}} |S| \upsilon_{R,k+1-|\gamma|}( \partial^\gamma \varphi_S) \varphi_S(f) \right] \end{aligned} \] 
Now, appealing to the simple small scale, cancellative averaging lemma from \cite{diplinio22wrt}*{Lemma 3.3} we can collapse the innermost sums so that
	\[ \sigma_{3,\mathrm{para},k} + \sigma_{4,\mathrm{para},k} =  \sum_{|\gamma| \le k} \sum_{R \in \mathcal S} |R| \mathfrak b^{\alpha,\gamma}_R \overline{\varphi_R}(g) q_R(\partial^\gamma f),\]
for some $q_R \in C \Phi^{0,\eta}(R)$. 
%		&= \sum_{|\gamma| \le k} \sum_{R \in \mathcal S} |R|  \mathfrak b^{\alpha,\gamma}_R q_{R,k+1} (\partial^\gamma f) \overline{\varphi_R}(g) \\
Following the same reasoning with $\gamma=k=0$ and the roles of $f$ and $g$ exchanged we obtain $p_S \in C\Phi^{0,\eta}(S)$ such that
	\[ %\Sigma_{1,\mathrm p, \star} + \Sigma_{2,\mathrm p,\star} &
	\sigma_{2,\mathrm{para},\star} + \sigma_{4,\mathrm{para},\star} = \sum_{S \in \mathcal S} |S| \mathfrak b^{\alpha,\star}_S \varphi_S(f) \overline{p_{S}}(g).
	\] 
%The first term in $\Sigma_3$ combines with $\Sigma_{2,\mathrm p,\mathrm a}$ and $\Sigma_{1,\mathrm p, \star} + \Sigma_{2,\mathrm p,\star}$ to produce the $\alpha$-th component of $\vec \Pi_{\Lambda}^k$, while the last term of $\Sigma_3$ cancels out with $\Sigma_{2,\mathrm p,\mathrm b}$. 
This completes the proof of Theorem \ref{thm:main} in the cases $k\geq 1.$

Theorem \ref{thm:main} also holds when $k=0$ with the same approach, though certain modifications are made to account for the fact that $\PM^0 = 0$. Nonetheless, one arrives at \eqref{e:summ} with the drawback that the wavelets $\theta_W,\phi_W,\upsilon_W$ will not belong to $\Phi^{0,\eta}_{2}(W)$. Instead, they will be of the form $\zeta_W \otimes \chi_W$ where $\chi_W \in \Phi^{0,\Subset}(W)$ and $\zeta_W$ satisfies, for $x \in P$,
	\begin{equation}\label{e:bwave} | \zeta_W(x)| \lesssim \log \left(1+\frac{\ell(P)}{\ell(W)}\right) \frac 1{\mathfrak d(P,W)^d}. \end{equation}
The cancellative wavelets $\psi_R$ and $\varphi_R$ in \eqref{e:summ} remain unchanged. Let us demonstrate that these limited decay wavelets still produce an $L^2(\Omega) \times L^2(\Omega)$ bounded wavelet form.

\begin{proposition}
Let $\ep >0$ and $\Omega$ be a uniform domain. For $\zeta_W$ satisfying \eqref{e:bwave}, and $\chi_W \in \Phi^{0,\Subset}(W)$, 
	\[ \left| \sum_{W \in \mathcal W} |W| \zeta_W(f) \chi_W(g) \right| \lesssim \|f\|_{L^2(\Omega)} \|g\|_{L^{2}(\Omega)}.\]
\end{proposition}
\begin{proof}
Decompose $f = \sum_{P \in \mathcal W} f \mathbf 1_P$. By \eqref{e:bwave},
	\[ \left| \sum_{W \in \mathcal W} |W| \zeta_W(f) \chi_W(g) \right| \lesssim \sum_{W,P \in \mathcal W} \frac{|W| |P| }{ \mathfrak d(P,W)^{d}} \log \left( 1+\frac{\ell(P)}{\ell(W)} \right) \l f \r_P \l g \r_W.\]
By Schur's test, the claimed estimate will follow from
	\[ \sup_{W \in \mathcal W} \ell(W)^\eta \sum_{P \in \mathcal W} |P| \log\left( 1 + \frac{\ell(P)}{\ell(W)} \right) \frac{\ell(P)^{-\eta}}{\mathfrak d(P,W)^d} <\infty \]
for every $0 <\eta<\ep$. This can be established using the same reasoning used to prove \eqref{eq:est2} in the proof of Lemma \ref{lemma:whitney}.
\end{proof}

\section{Sobolev regularity of compressions}
\label{sec:compress} This section particularizes the main results of \S\ref{sec:rep} in to the  case of compressions of globally defined singular integral operators. The main results are Theorem \ref{thm:compress} and Corollary \ref{cor:compress}. All proofs relative to this section, as well as the proof of Theorem \ref{t:belt} we stated in the introduction, are given in the concluding \S\ref{pf:belt}.
 \addtocounter{other}{1} 
\subsection{Global $(k,\ell)$ forms}

In what follows, we refer to the $(k,\ell)$ forms of Definition \ref{def:si} for the case $\Omega=\R^d$ as \textit{global} $(k,\ell)$-forms.  In this case, according to Remark \ref{rem:global}, Theorem \ref{thm:main} holds with $\mathcal W = \varnothing$ and $\mathcal S = \mathcal D$, and the same holds for the   paraproducts of Definition \ref{def:pp}.  This occurrence of Theorem \ref{thm:main} has a counterpart in \cite{diplinio22wrt}*{Theorem A}. However, the latter result leads to testing and  cancellation conditions which are only necessary for \textit{homogeneous} Sobolev space $\dot W^{k,p}(\R^d)$ bounds for all $1<p<\infty$. Similar results on more general homogeneous smoothness spaces were also given in \cite{CHO2020,LiSun2022}. Restating the estimates in Corollaries \ref{cor:sob} and \ref{cor:weight} in the global case, where the $\mathrm{Carl}^{n,p}$ norms play no role due to $\mathcal W=\varnothing$, gives new results which we record explicitly here.

\begin{cor}\label{cor:gl} 
Let $\Lambda$ be a global $(k,0)$-form. Then the adjoint operator $T$ defined by \eqref{e:adjT}
 \begin{align} \|\nabla^k T f\|_{L^p(\mathbb R^d)} &\lesssim \left[ \|\Lambda\|_{\mathrm{SI}(k,0)}+\|\Lambda\|_{ \dot F(k,p)}\right]  \|f\|_{W^{k,(p,1)}(\mathbb R^d)}, \qquad 1<p<\infty,
\\
 \|\nabla^k T f\|_{L^p(\mathbb R^d)} &\lesssim\left[ \|\Lambda\|_{\mathrm{SI}(k,0)}+\|\Lambda\|_{\dot F(k,p,q)}\right] \|f\|_{W^{k,p}(\mathbb R^d)}, \qquad 1<p<q<\infty. \end{align}
Furthermore, for $1<s<\infty$,
	\begin{align}
		\label{e:wcor1gl} \|\nabla^k T f \|_{L^p(\mathbb R^d,w)} 
			&\lesssim \mathcal G\left(  [w]_{\mathrm{A}_p}, [w]_{\RH_s}\right) 
			\left( \left \Vert \Lambda \right \Vert_{\mathrm{SI}(k,0)} + \|\Lambda\|_{\dot F(k,ps')} \right) \|f\|_{W^{k,p}(\mathbb R^d,w)}, \\ 
		\label{e:wcor2gl} \|\nabla^k T f \|_{L^p(\mathbb R^d,w)} 
			&\lesssim  \mathcal G\left( [w]_{\mathrm{A}_p} \right) 
			\left( \left \Vert \Lambda \right \Vert_{\mathrm{SI}(k,0)} + \|\Lambda\|_{\dot F(k,\infty)} \right) \|f\|_{W^{k,p}(\mathbb R^d,w)}.
	\end{align}
\end{cor}

 \addtocounter{other}{1} 
\subsection{Compression of global forms}
Let now $\Lambda$ be a global $(k,0)$-form, and denote by $T$ the corresponding   adjoint operator defined by \eqref{e:adjT}. 
 Fixing a uniform domain $\Omega\subset \R^d$ ,  refer to Definition \ref{def:comp} for $\Lambda_\Omega$, the compression of $\Lambda$ to $\Omega$, as  a $(k,0)$-form on $\Omega$. Below, denote by $T_\Omega$  the adjoint operator associated to the form  $ \Lambda_\Omega$ by \eqref{e:adjT}. When particularized to $\Lambda_\Omega$, the testing conditions of Corollary \ref{cor:sob} may be reformulated in terms of   the family of symbols
\begin{equation}\label{e:L-gl-pp} \begin{split}
& \aa^{\alpha,\gamma}=\left\{a^{\alpha,\gamma}_S\coloneqq \partial^\alpha \Lambda\big( \Tr_S\mathsf x^\gamma,\varphi_S \big): S\in\mathcal S\right\}, \qquad |\alpha|=k, |\gamma|\leq k
\end{split}
\end{equation}
obtained by testing $\Lambda$ on polynomials against the cancellative portion of the basis  $\mathfrak{B}(k,\Omega)$.   The specific connection is formalized in the next lemma.

\begin{lemma}\label{l:bbaa}
Let $\{ \bb^{\alpha,\gamma},\bb^{\alpha,\star} : |\gamma| \le k = |\alpha|\}$ be the paraproduct symbols of $\Lambda_\Omega$ according to Definition \ref{def:cz}. Then for any $|\alpha|=k,$ $1 \le p,q \le \infty$,
	\[  
	\sup_{n+m=-k} \|\bb^{\alpha,\star}\|_{\dot F^{n,m}_{p,q}(\mathcal M)} +	 
\sum_{|\gamma| \le k}\, \sup_{n+m=|\gamma|-k} \|\bb^{\alpha,\gamma} - \mathfrak a^{\alpha,\gamma} \|_{\dot F^{n,m}_{p,q}(\mathcal M)}  \lesssim \|\Lambda\|_{\mathrm{SI}(k,0)}	. 
	 \]
\end{lemma}
Lemma \ref{l:bbaa} and a comparison with \eqref{e:czn} leads to testing conditions for $\Lambda_\Omega$ based on the norms 
\begin{align}
  \label{e:Fncomp} 
 \|\Lambda\|_{\dot F(k,p,q,\Omega)} &\coloneqq  \sup_{|\alpha|=k} \left[      
\sup_{ |\gamma|<k-\lfloor \frac dp \rfloor  } \left\|\aa^{\alpha,\gamma}\right\|_{\dot F^{0,|\gamma|-k}_{q,2}(\mathcal M) }
\right. \\ &  \qquad \quad +\left.\sup_{k-\lfloor \frac dp \rfloor \leq |\gamma| < k} \left\|\mathfrak{a}^{\alpha,\gamma}\right\|_{\dot F^{0,|\gamma|-k}_{p,2}(\mathcal M) }  
+ \sup_{  |\gamma| = k}   \left\|\mathfrak a^{\alpha,\gamma}\right\|_{\dot F^{0,0}_{1,2}(\mathcal M) } \right], \\
  \label{e:czncomp} 
 \|\Lambda\|_{\mathrm{PP}(k,p,q,\Omega)}&\coloneqq \|\Lambda\|_{\dot F(k,p,q,\Omega)} + \|\Lambda_\Omega\|_{\mathrm{Carl}(k,p)}.
\end{align}
As before, when $p=q$, we omit the parameter $q$ and write  e.g. $\dot F(k,p,\Omega)$ instead of $\dot F(k,p,p,\Omega)$ in the norms \eqref{e:Fncomp}, \eqref{e:czncomp}.
\begin{theorem}\label{thm:compress}
Let $1<p<q<\infty$ and $\|\Lambda\|_{\mathrm{SI}(k,0)}=1$. Then,
	\begin{align}
	\label{e:compress-unweight-p} 
	\left\| \nabla^k T_\Omega f\right\|_{L^{p}(\Omega) } &\lesssim \left(  1 +   \|\Lambda\|_{\mathrm{PP}(k,p,\Omega)}  \right)   \|f\|_{W^{k,(p,1)}(\Omega)},\\
	\label{e:compress-unweight-pq} 
 	\left\| \nabla^k T_\Omega f\right\|_{L^{p}(\Omega) } &\lesssim \left( 1 +   \|\Lambda\|_{\mathrm{PP}(k,p,q,\Omega)}  \right)   \|f\|_{W^{k,p}(\Omega)}.
	 \end{align}
Furthermore,  for all $1<p,s<\infty$ 
	\begin{align}
		\label{e:wcompress1} \|\nabla^k T_\Omega f \|_{L^p(\Omega,w)} 
			&\lesssim \mathcal G\left( [w]_{\mathrm{A}_p(\Omega)}, [w]_{\mathrm{RH}_s(\Omega^\circ)} \right)  \\ &\quad  \times 
			\left( 1 + \|\Lambda\|_{\dot F(k,ps',ps')}  +  \|\Lambda_\Omega\|_{\mathrm{Carl}_w(k,p)} \right) \|f\|_{W^{k,p}(\Omega,w)}, \\ 
		\label{e:wcompress2} \|\nabla^k T_\Omega f \|_{L^p(\Omega,w)} 
			&\lesssim  \mathcal G \left( [w]_{\mathrm{A}_p(\Omega)} \right)
			\left( 1 + \|\Lambda\|_{\dot F(k,\infty,\infty)}  +  \|\Lambda_\Omega\|_{\mathrm{Carl}_w(k,p)} \right) \|f\|_{W^{k,p}(\Omega,w)}. 
	\end{align}
 \end{theorem}
Before proving the theorem, we state a further corollary which is  easier to compare to previous literature on the subject. 
\begin{cor}\label{cor:compress}
Suppose $\aa^{\alpha,\gamma}=0$ for $|\alpha|=k$ and $|\gamma| \le k-1$. Then, for any $1<p<\infty$, 
	\[ \begin{aligned} \|T_\Omega\|_{\mathcal L(W^{k, p}(\Omega,w))} 
		&\lesssim \mathcal G \left( [w]_{\mathrm{A}_p(\Omega)} \right) \\ &\quad \times \left( \|\Lambda\|_{\mathrm{SI}(k,0)} + \sup_{|\alpha|=|\gamma|=k} \left\Vert \aa^{\alpha,\gamma} \right\Vert_{\dot F^{0,0}_{1,2}(\mathcal M)} + \|\Lambda_\Omega\|_{\mathrm{Carl}(k, p,w)} \right) .
	\end{aligned} \]  Furthermore, suppose that $d<r<p<\infty$. Then,
	\begin{equation}
	\label{e:Apd}\begin{split} \left \Vert T_\Omega \right\Vert_{\mathcal L(W^{k,p}(\Omega,w)} &\lesssim \mathcal G \left({[w]_{\mathrm{A}_r(\Omega)}},[w]_{\RH_s(\Omega)} \right) \\ & \quad \times \left( \|\Lambda\|_{\mathrm{SI}(k,0)} + \sup_{|\alpha|=|\gamma|=k} \left\Vert \aa^{\alpha,\gamma} \right\Vert_{\dot F^{0,0}_{1,2}(\mathcal M)} + \|\Lambda_\Omega\|_{\mathrm{Carl}(k, \frac{ps}{s-1})} \right).\end{split} \end{equation}
\end{cor}
Observe that the conditions on $\aa$ assumed in the above corollary  characterize the $(k,0)$-forms on $\R^d$ which preserve the \textit{homogeneous}  Sobolev space $\dot W^{k, p}(\mathbb R^d)$ \cite{diplinio22wrt}. In particular, the  (adjoint forms to) smooth  convolution type Calder\'on-Zygmund operators on $\R^d$, which are the objects of the $T(P)$ theorem of Prats and Tolsa \cite {prats-tolsa}
 fall into the scope of Corollary \ref{cor:compress}. Lemma \ref{l:worse} below shows that the testing condition \eqref{e:PT2} imposed in \cite[Thms.\ 1.1 and 1.2]{prats-tolsa} controls the unweighted Carleson norm component $\|\Lambda_\Omega\|_{\mathrm{Carl}(k, p)}$  appearing in our corollary.
\begin{lemma}\label{l:worse}
If 
\begin{equation}\label{e:PT2} \nabla^k T_\Omega (\mathsf{P}) \in \mathrm {Carl}^{1,p}(\Omega) \quad \mbox{for all polynomials $\mathsf P$ with } \operatorname{deg}(\mathsf P) \le k-1, \end{equation}
then $\mathfrak b^{\alpha,\gamma} \in \mathrm{Carl}^{k-|\gamma|,p}(\Omega)$  for all $|\alpha|=k$ and $|\gamma| \le k-1$.
\end{lemma}
Thus, Corollary \ref{cor:compress} fully recovers the main result of \cite{prats-tolsa}  extending  the scope  to unbounded uniform domains, suitable non-convolution-type operators, and weighted So\-bo\-lev spaces.
Furthermore, as shown by the results of \S\ref{sec:pp}, the testing condition based on finiteness of $\|\Lambda_\Omega\|_{\mathrm{Carl}(k, p)}$  is necessary and sufficient. In contrast, the  testing condition \eqref{e:PT2} used in \cite{prats-tolsa} is only necessary when $k=1$ and $\Omega$ is bounded.

\addtocounter{other}{1}
\subsection{Proofs} \label{pf:belt}

\begin{proof}[Proof of Lemma \ref{l:bbaa}]
For $|\gamma| =m \le k = |\alpha|$, recall that for each $R \in \mathcal S$, \rev{the scalars $a^{\alpha,\gamma}_R$ and $b^{\alpha,\gamma}_R$ are the wavelet coefficients of the testing symbols $\mathfrak a^{\alpha,\gamma}$ (coming from the global form $\Lambda$) and $\bb^{\alpha,\gamma}$ (from the compressed form $\Lambda_\Omega$).} Define $e^{\alpha,\gamma}_R = a^{\alpha,\gamma}_R-b^{\alpha,\gamma}_R$. Referring to \eqref{e:limitint} and using a standard kernel estimate,	\[ \abs{e^{\alpha,\gamma}_R} =
	\abs{ \; \iint\displaylimits_{\Omega^c \times \mathsf w R} \partial^\alpha_x K(x,y)(y-c(R))^\gamma \varphi_R(y) \, \mathrm{d}x\mathrm{d}y } 
	\lesssim \|\Lambda\|_{\mathrm{SI}(k,0)} \frac{\ell(R)^\delta}{\operatorname{dist}(R,\Omega^c)^{k-m+\delta}}.\]  
Therefore, for any $x \in R$,
	\[ \left \vert \mathrm{S}_{1,R}^0 \mathfrak e^{\alpha,\gamma}(x) \right\vert \lesssim 
	\sum_{\substack{S \in \mathcal D(R) \\ x \in S}} \frac{\ell(S)^\delta}{\operatorname{dist}(S,\Omega^c)^{k-m+\delta}}
	\lesssim 
		\frac{\ell(R)^\delta}{\operatorname{dist}(R,\Omega^c)^{k-m + \delta}}	. \]
Multiplying the above display by $\ell(R)^{k-m}$ and using the fact that $\ell(R) \le \ell(W(R)) \sim \operatorname{dist}(R,\Omega^c)$ shows that $\mathfrak e^{\alpha,\gamma} \in \dot F^{0,m-k}_{\infty,1}$ which by the embeddings \eqref{e:Femb} proves the first estimate.
Notice that the same argument shows $\|\bb^{\alpha,\star}-\aa^{\alpha,\star}\|_{\dot F^{0,-k}_{\infty,1}} \lesssim \|\Lambda\|_{\mathrm{SI}(k,0)}$ where $a^{\alpha,\star}_S \coloneqq \partial^\alpha \Lambda( \varphi_S,\mathbf 1)$. The desired estimate will now be established by showing $\aa^{\alpha,\star}=0$. According to \eqref{e:limitint}, let $\iota_S,\varpi_n \in C^\infty_0(\mathbb R^d)$ with $\mathbf 1_{\mathsf w S} \le \iota_S \le \mathbf 1_{2 \mathsf w S}$ and $\mathbf 1_{B(0,n)} \le \varpi_n \le \mathbf 1_{B(0,n+1)}$. Let $|\kappa|=1$ with $\kappa \le \alpha$ so that integrating by parts,
	\[ \begin{aligned} &\lim_{n \to \infty} \iint_{\mathbb R^d \times \mathbb R^d} \partial^\alpha_x K(x,y) \varphi_S(y) [1-\iota_S(x)]\varpi_n(x) \, \mathrm d y \, \mathrm dx\\
		&=\lim_{n \to \infty} \partial^{\alpha-\kappa} \Lambda(\varphi_S,\partial^\kappa (\iota_S \varpi_n)) - \iint_{\mathbb R^d \times \mathbb R^d} \partial^{\alpha-\kappa}_x K(x,y) \varphi_S(y) \partial^\kappa \omega_n(x) \, dy \, dx . \end{aligned}\]
The first term in the limit converges to $-\partial^\alpha \Lambda(\varphi_S,\iota_S)$ and the second term is $O(n^{-1})$ by  kernel estimates \eqref{eq:smooth-kernel} and the fact that $\partial^\kappa \varpi_n$ is supported in $B(0,n+1) \backslash B(0,n)$. Therefore, appealing to \eqref{eq:poly-action} shows $\aa^{\alpha,\star}=0$.
\end{proof}

\begin{proof}[Proof of Lemma \ref{l:worse}]
Fix a boundary window $W_j^*$. For each $P \in \SH_1(W_j^*)$, compute
	\begin{equation}\label{e:bS}\begin{split}  \|\mathfrak b^{\alpha,\gamma} \|_{L^p(P)} &\le \sum_{\substack{W \sim P \\ 1 \le j \le \jmath }} |W| \ip {\nabla^k T_\Omega(\Tr_W \mathsf x^\gamma)} {\chi_{W,j}}\\ &\quad  + \left\| \sum_{R \in \mathcal D(W)} \abs{ \ip {\nabla^k T_\Omega(\Tr_R \mathsf x^\gamma)} {\varphi_R} }^2 \mathbf 1_R \right\|_{L^{p/2}(P)}.\end{split}\end{equation}
For each $R \in \mathcal D_+(W)$ and $W \sim P$, factor 
	\[ \Tr_R \mathsf x^\gamma 
	= \sum_{\beta \le \gamma} \binom{\gamma}{\beta} (c(R)-c(W_j^*))^{\gamma-\beta} \Tr_{W_j^*} \mathsf x^\beta.\]
However, $|c(R)-c(W_j^*)| \lesssim \ell(W_j^*) \lesssim 1$ by Lemma \ref{l:windows} (S1), whence, referring back to \eqref{e:bS},
	\[ \begin{aligned} \|\mathfrak b^{\alpha,\gamma} \|_{L^p( P)} 
		&\lesssim \sum_{W \sim P} \sum_{\beta \le \gamma} \binom{\gamma}{\beta} \left( \left\| \nabla^k T_\Omega (\Tr_{W_j^*} \mathsf x^\beta) \right\|_{L^p(W)} + \|\mathrm{S}_W (\nabla^k T_\Omega(\Tr_{W_j^*} \mathsf x^\beta)) \|_{L^p(P)} \right) \\
		&\lesssim \sup_{\beta \le \gamma} \left\| \nabla^k T_\Omega (\Tr_{W_j^*} \mathsf x^\beta) \right\|_{L^p(\mathsf w P)}.
	\end{aligned} \]
Therefore, by (C3) in Lemma \ref{l:carl-prop}, $\mathfrak b^{\gamma,\alpha} \in \mathrm{Carl}^{1,p}(\Omega)\hookrightarrow\mathrm{Carl}^{k-|\gamma|,p}(\Omega)$ by statement (C1) of the same lemma and the fact that $|\gamma|\le k-1$.
\end{proof}

\begin{proof}[Proof of Theorem \ref{thm:compress}]
Lemma \ref{l:bbaa} shows for any $1 \le q \le \infty$, $|\alpha|=k$,  and $|\gamma| \le k$,  
	\[\|\bb^{\alpha,\star}\|_{\dot F^{-k,0}_{1,2}(\mathcal M)} \lesssim \|\Lambda\|_{\mathrm{SI}(k,0)}, \qquad \|\bb^{\alpha,\gamma}\|_{\dot F^{0,|\gamma|-k}_{q,2}(\mathcal M)} \lesssim \|\aa^{\alpha,\gamma}\|_{\dot F^{0,|\gamma|-k}_{q,2}(\mathcal M)} + \|\Lambda\|_{\mathrm{SI}(k,0)},\]
whence, for any $1\le p,q \le \infty$,
	\begin{equation}\label{e:Fcompare} \|\Lambda_\Omega\|_{\dot F(k,p,q)} \lesssim \|\Lambda\|_{\mathrm{SI}(k,0)} +  \|\Lambda\|_{\dot F(k,p,q,\Omega)}. \end{equation}
The unweighted estimates 
	in Theorem \ref{thm:compress} follow from applying Corollary \ref{cor:sob} to $\Lambda_\Omega$, and noticing that by \eqref{e:Fcompare},
	\[ \|\Lambda_\Omega\|_{\mathrm{PP}(k,p,q)} \lesssim \|\Lambda\|_{\mathrm{SI}(k,0)} + \|\Lambda\|_{\mathrm{PP}(k,p,q,\Omega)}.\]
Similarly, the weighted estimates  
 follow from \eqref{e:Fcompare} and Corollary \ref{cor:weight}.
\end{proof}

\begin{proof}[Proof of Theorem \ref{t:belt}] 
The Ahlfors-Beurling operator $B$ is a global Calder\'on-Zyg\-mund operator of convolution type which places us in the situation of Corollary \ref{cor:compress} with \textit{all} $\aa^{\alpha,\gamma}=0$. Since $k=1$, the only functions appearing in the Carleson measure norm are
	\[ \bb^{\alpha,0} = \sum_{\substack{W \in \mathcal W \\ 1 \le j \le \jmath }} \left \langle \partial^\alpha B_\Omega (\mathbf 1), \chi_{W,j} \right \rangle \chi_{W,j} = \sum_{S \in \mathcal S} \left \langle \partial^\alpha B_\Omega (\mathbf 1), \varphi_S \right \rangle \varphi_S = \partial^\alpha B_\Omega (\mathbf 1).\] 
 
By the above display and the first statement of Lemma \ref{l:Apd}
	\[ \begin{aligned} \norm{\Lambda}_{\mathrm{Carl}(1,p,w)} &\lesssim \mathcal G\left( [w]_{\A_{\frac{r}{2}}(\Omega)} \right) \sup_{j} \frac{ \norm{ \nabla B_\Omega \mathbf 1 }_{L^p(\Sh_1(W_j^*),w)} }{w(\Sh_1(W_j^*))^{\frac 1p}} \\
	&\lesssim \mathcal G\left( [w]_{\A_{\frac{r}{2}}(\Omega)} \right) \frac{ \norm{ \nabla B_\Omega \mathbf 1 }_{L^p(\Omega,w)} }{w(\Omega)^{\frac 1p}}, \end{aligned} \]
where in the last inequality we used that $\Omega$ is a bounded domain so $w(\Omega) \lesssim \mathcal G\left( [w]_{\A_{\frac{r}{2}}(\Omega)} \right) w(\Sh_1(W_j^*))$. Thus the first statement of Theorem \ref{t:belt} is now a consequence of Corollary \ref{cor:compress}. To prove the second statement in Theorem \ref{t:belt}, we instead apply the second statement in Corollary \ref{cor:compress} and notice that
	\[ \norm{\Lambda}_{\mathrm{Carl}(1,q)} \lesssim \norm{B_\Omega \mathbf 1}_{W^{1,q}(\Omega)}.\]
 
Finally, the main result of \cite{cruz-tolsa}, to which we also send for the specific definition of the Besov space $B^{1-\frac 1q}_{q,q}(\partial \Omega)$, is 
	\[ \left \Vert B_\Omega (\mathbf 1) \right \Vert_{W^{1,q}(\Omega)} \lesssim \left \Vert N_\Omega \right \Vert_{B^{1-\frac 1q}_{q,q}(\partial \Omega)},\]
which establishes the second estimate in Theorem \ref{t:belt}.
\end{proof}
 \bibliography{refs-wrt-domain}
\bibliographystyle{amsplain}

\end{document}